%% file: _main.tex
\documentclass[11pt,reqno]{amsart} 

\input{_preamble}
\input{_shorts}

\title{The Galois characterisation of $p$-adically closed fields \\ A modern perspective}

\author{Leo Gitin, Jochen Koenigsmann, Benedikt Stock}

\begin{document}

\begin{abstract}
In 1927, Artin and Schreier showed that a field is real closed if and only if its absolute Galois group has order two. Inspired by this characterisation and drawing on earlier work of Neukirch, Pop conjectured the following $p$-adic analogue: a field is $p$-adically closed if and only if its absolute Galois group is isomorphic to that of $\IQ_p$. In 1995, the conjecture was independently solved by Efrat for $p \ne 2$ and by Koenigsmann in full generality. Using novel techniques in the theory of valued fields developed over the last 25 years, we give a new, elementary, and self-contained proof of this theorem, with a Galois characterisation of henselianity at the heart of the proof and without relying on Galois cohomology. We further highlight connections to the recent work of Jahnke-Kartas on perfectoid fields and model-theoretic transfer techniques. We provide a systematic account of all of our methods to encourage further investigations.
\end{abstract}

\pagestyle{plain}
\maketitle

\setcounter{tocdepth}{3}
{\normalsize\tableofcontents}
\newpage

\input{part1_intro}
\newpage

\input{part2_val_th}
\newpage

\input{part3_galois}
\newpage

\input{part4_transfer}
\newpage

\input{part5_main}
\newpage

\bibliographystyle{alpha} 
\bibliography{references}

\end{document}

%% file: _preamble.tex
\usepackage[left=3cm,right=3cm,top=3cm,bottom=3cm]{geometry}
\usepackage{parskip}
\usepackage{amsmath, amssymb, amsthm, mathtools}
\usepackage{stmaryrd}
\usepackage[english]{babel}

\numberwithin{equation}{section}

\usepackage[shortlabels]{enumitem}
\setlist[enumerate]{topsep=1ex,itemsep=-0.3ex,leftmargin=6ex}

\usepackage[font=footnotesize]{caption} 

\usepackage{xcolor}
\definecolor{mygreen}{RGB}{27,110,27}

\usepackage[colorlinks=true]{hyperref}
\hypersetup{linkcolor=blue, citecolor=mygreen}

\usepackage{tikz-cd}
\newenvironment{diagram}{\begin{center}\begin{tikzcd}}{\end{tikzcd}\end{center}}

\newtheorem{thm}{Theorem}[section]
\newtheorem*{mainthm*}{Main Theorem}
\newtheorem{prop}[thm]{Proposition}
\newtheorem{lem}[thm]{Lemma}
\newtheorem*{transferlem*}{Transfer Lemma}
\newtheorem*{lem*}{Lemma}
\newtheorem{fact}[thm]{Fact}
\newtheorem{cor}[thm]{Corollary}
\newtheorem{cordef}[thm]{Corollary/Definition}
\newtheorem{cor*}[thm]{Corollary}

\newtheorem*{thm*}{Theorem}
\newtheorem*{prop*}{Proposition}
\newtheorem*{slogan*}{Slogan}

\theoremstyle{definition}
\newtheorem{defi}[thm]{Definition}
\newtheorem{rem}[thm]{Remark}
\newtheorem{example}[thm]{Example}

\newtheorem*{rem*}{Remark}
\newtheorem*{note*}{Note}
\newtheorem{idea}[thm]{Idea}
\newtheorem*{idea*}{Idea}
\newtheorem*{ingr*}{Ingredients}
\newtheorem*{leit*}{Leitfaden}

\newtheoremstyle{claim}
{}                
{}                
{\itshape}        
{}                
{\itshape}       
{.}               
{ }               
{}                
\theoremstyle{claim}

\newenvironment{claimproof}[1]{%
  \textit{Proof of Claim%
    \if\relax\detokenize{#1}\relax
    \else\ #1%
    \fi.}%
}{\hfill$\vcenter{\hbox{$\diamond$}}$}
\newenvironment{claimproof*}[1]{%
  \textit{Proof of Claim%
    \if\relax\detokenize{#1}\relax
    \else\ #1%
    \fi.}%
}{}

\makeatletter
\newcommand{\dbloverline}[1]{\overline{\mathrlap{\dbl@overline{#1}}\phantom{#1}}}
\newcommand{\dbl@overline}[1]{\mathpalette\dbl@@overline{#1}}
\newcommand{\dbl@@overline}[2]{%
  \begingroup
  \sbox\z@{$\m@th#1\overline{#2}$}%
  \ht\z@=\dimexpr\ht\z@-2\dbl@adjust{#1}\relax
  \box\z@
  \ifx#1\scriptstyle\kern-\scriptspace\else
  \ifx#1\scriptscriptstyle\kern-\scriptspace\fi\fi
  \endgroup
}
\newcommand{\dbl@adjust}[1]{%
  \fontdimen8
  \ifx#1\displaystyle\textfont\else
  \ifx#1\textstyle\textfont\else
  \ifx#1\scriptstyle\scriptfont\else
  \scriptscriptfont\fi\fi\fi 3
}
\makeatother

\makeatletter
\newcommand{\stepitem}[2]{
  \item[\textbf{Step #1.}]%
    \phantomsection%
    \def\@currentlabel{Step~#1}%
    \label{#2}%
}
\makeatother

\makeatletter
\def\@tocline#1#2#3#4#5#6#7{\relax
  \ifnum #1>\c@tocdepth 
  \else
    \par \addpenalty\@secpenalty\addvspace{#2}%
    \begingroup \hyphenpenalty\@M
    \@ifempty{#4}{%
      \@tempdima\csname r@tocindent\number#1\endcsname\relax
    }{%
      \@tempdima#4\relax
    }%
    \parindent\z@ \leftskip#3\relax \advance\leftskip\@tempdima\relax
    \rightskip\@pnumwidth plus4em \parfillskip-\@pnumwidth
    #5\leavevmode\hskip-\@tempdima
      \ifcase #1
       \or\or \hskip 1em \or \hskip 2em \else \hskip 3em \fi%
      #6\nobreak\relax
    \dotfill\hbox to\@pnumwidth{\@tocpagenum{#7}}\par
    \nobreak
    \endgroup
  \fi}
\makeatother

%% file: _shorts.tex
\renewcommand{\and}{\quad \text{and} \quad}
\newcommand{\laurent}[1]{(\!(#1)\!)}

\newcommand{\ifftext}{\quad \text{if and only if} \quad}

\newcommand{\mono}{\lhook\joinrel\longrightarrow}
\newcommand{\epi}{\relbar\joinrel\twoheadrightarrow}

\newcommand{\Uu}{\mathcal{U}}
\newcommand{\ultra}{\mbox{\larger\big{/}}}

\newcommand{\Aa}{\mathcal{A}}
\newcommand{\Ll}{\mathcal{L}}
\newcommand{\Mm}{\mathcal{M}}
\newcommand{\Qq}{\mathcal{Q}}
\newcommand{\Ss}{\mathcal{S}}

\DeclareMathOperator{\ring}{ring}
\DeclareMathOperator{\oag}{oag}
\DeclareMathOperator{\VF}{VF}
\DeclareMathOperator{\val}{val}
\DeclareMathOperator*{\ulim}{ulim}

\newcommand{\IN}{\mathbb{N}}
\newcommand{\IZ}{\mathbb{Z}}
\newcommand{\IQ}{\mathbb{Q}}
\newcommand{\IR}{\mathbb{R}}
\newcommand{\IC}{\mathbb{C}}
\newcommand{\IF}{\mathbb{F}}
\newcommand{\IP}{\mathbb{P}}

\DeclareMathOperator{\Hom}{Hom}
\DeclareMathOperator{\Fix}{Fix}
\DeclareMathOperator{\Aut}{Aut}

\DeclareMathOperator{\Gal}{Gal}
\DeclareMathOperator{\charK}{char}
\DeclareMathOperator{\sepc}{sep}

\DeclareMathOperator{\Frac}{Frac}
\DeclareMathOperator{\cd}{cd}
\DeclareMathOperator{\res}{res}
\DeclareMathOperator{\Red}{red}
\DeclareMathOperator{\ord}{ord}

\DeclareMathOperator{\im}{im}
\DeclareMathOperator{\alg}{alg}

\DeclareMathOperator{\Syl}{Syl}
\DeclareMathOperator{\Tr}{Tr}
\DeclareMathOperator{\rad}{rad}

\newcommand{\Oo}{\mathcal{O}}
\DeclareMathOperator{\ur}{ur}

\DeclareMathOperator{\Conv}{Conv}
\DeclareMathOperator{\Val}{Val}

%% file: part1_intro.tex
\section{Introduction} 

The main theme of Galois theory is the duality between extensions of fields and their associated Galois groups.
The \emph{absolute Galois group} of a field $K$, 
\[
    G_K \coloneq \Gal(K^{\sepc}/K),
\]
is the Galois group of the separable closure of $K$ over $K$. It is the largest possible Galois group associated to $K$, and moreover, a profinite group endowed with a natural (Krull) topology. Every Galois group of a finite Galois extension of $K$ occurs as a quotient of $G_K$ by some open subgroup of $G_K$.
Naturally, one might want to determine the structure of the absolute Galois group $G_K$ for fields $K$ of arithmetic interest---and conversely, perhaps more subtly, one might ask how much of the arithmetic of $K$ is captured by $G_K$. This question lies at the heart of \textit{anabelian geometry}. More generally, anabelian geometry seeks to understand how much of the information about a variety $X$ is encoded in its étale fundamental group.

Perhaps the first result in this direction, even though it is not usually construed as such, is the classical theorem of Artin and Schreier that answers the above question for the case of finite absolute Galois groups.

\begin{thm}[Artin-Schreier 1927, {\cite{Artin-Schreier27a,Artin-Schreier27b}}]
\label{thm:Artin-Schreier}
Let $K$ be a field. The absolute Galois group $G_K$ is finite if and only if
\begin{enumerate}[(i)]
    \item $|G_K| = 1$, in which case $K$ is separably closed, or
    \item $|G_K| = 2$, in which case $K$ is real closed, i.e., a field that admits an ordering which does not extend to any proper algebraic extension of $K$. Equivalently, $K$ is real closed iff no quadratic sum $a_1^2 + a_2^2 + \ldots + a_n^2$ in $K$ is equal to $-1$, any polynomial of odd degree has a root, and, for any $a \in K$, there exists $b \in K$ such that $a = b^2$ or $a = -b^2$.
\end{enumerate}
\end{thm}

Traditionally, the starting point of this investigation is considered to be the following result.

\begin{thm}[Neukirch-Uchida-Ikeda-Iwasawa]
\label{thm:Neukirch}
Let $K_1$ and $K_2$ be algebraic number fields. Any topological isomorphism $G_{K_1} \cong G_{K_2}$ is given by conjugation by a field isomorphism
\[
    \sigma : K_1^{\alg} \stackrel{\cong}{\longrightarrow} K_2^{\alg}
\]
restricting to $\sigma|_{K_1} : K_1 \stackrel{\cong}{\longrightarrow} K_2$.
In particular, working within a fixed algebraic closure $\IQ^{\alg}$ of $\IQ$, any two number fields $K_1$ and $K_2$ satisfy
\begin{equation} \label{eq:Neukirch}
    K_1 \cong K_2 \Longleftrightarrow G_{K_1} \cong G_{K_2}.
\end{equation}
\end{thm}
In this result, the absolute Galois group $G_K$ determines $K$ uniquely up to isomorphism within the (rather restrictive) class of number fields.
The search for larger classes for which (\ref{eq:Neukirch}) would hold leads---perhaps surprisingly---to considerations in mathematical logic.

Indeed, if we look back at Theorem~\ref{thm:Artin-Schreier}, the class of real closed fields may look unusual at first sight. However, the reader familiar with mathematical logic will recognise that being real closed is a first-order property. Indeed, at the advent of \emph{model theory}, Tarski identified the real closed fields as precisely those structures that are elementarily equivalent to $\IR$, i.e., satisfying all the same first-order sentences as the structure $\IR$ in a language with symbols for ring operations $\Ll_{\ring} = \{0, 1, +, -, \cdot\}$.

\begin{thm}[Tarski 1931/51]
A field $K$ is real closed if and only if $K \equiv \IR$, i.e., $K$ is elementarily equivalent to $\IR$.
\end{thm}

See \cite[3.3]{Marker02} for a modern presentation of Tarski's Theorem. Crucially, note that if $K$ is real closed, the formula $\exists z\, (y - x = z^2)$ defines the unique ordering $x \le y$ that makes $K$ an ordered field. This is an instance of the phenomenon that $G_K$ codes \textit{extra structure} on $K$.

\begin{cor}
For any field $K$,
\[
    1 < |G_K| < \infty \ifftext K \equiv \IR.
\]
\end{cor}

From a model-theoretic point of view, the arithmetic of $K$ is described by the set of first-order sentences that hold in $K$.
Viewed through this lens, the preceding result shows that $G_{\IR}$ fully determines the arithmetic of $\IR$.

Likewise, the arithmetic of the $p$-adic field $\IQ_p$ is well-understood by the following celebrated result of Ax, Kochen, and Ershov \cite{Ax-Kochen2, Ershov65}.
\begin{thm}[Ax-Kochen, Ershov 1965] \label{thm:Ax-Kochen-axioms}
A field $K$ is elementarily equivalent to $\IQ_p$ if and only if it admits a (Krull) valuation $v$ such that:
\begin{enumerate}[(i)]
    \item $(K,v)$ is henselian of mixed characteristic $(0,p)$;
    \item $v(p)$ is minimal positive;
    \item the value group $vK$ is a $\IZ$-group, i.e., $vK \equiv \IZ$;
    \item the residue field is equal to $\IF_p$.
\end{enumerate}
\end{thm}
This theorem thus describes the class of \emph{$p$-adically closed fields}. Note that a valuation satisfying (i)--(iv) is always unique and given by a uniform formula in the language of rings, see (\ref{eq:JR-formula}). Hence, axioms (i)--(iv) can be rewritten as a set of $\Ll_{\ring}$-sentences.

Abstractly speaking, there are $p$-adically closed fields of any infinite cardinality by the Löwen\-heim-Skolem Theorem. However, this class also includes natural algebraic examples.
\begin{example} \label{ex:puiseux}
\begin{enumerate}
    \item The field of algebraic $p$-adic numbers $k = \IQ_p \cap \IQ^{\alg}$ is $p$-adically closed.
    \item The field of Puiseux series over $\IQ_p$,
    \[
        K = \bigcup_{n > 0} \IQ_p\laurent{t^{1/n}},
    \]
    is $p$-adically closed. This is perhaps the simplest example of rank two.
\end{enumerate}
\end{example}

In an early advance towards the Galois characterisation of $p$-adically closed fields, Pop obtained a formative result extending Neukirch's within a restriction of the class of $p$-adically closed fields \cite[Thm. E.11]{Pop88}.
\begin{thm}[Pop 1988]
\label{thm:Pop}
    Let $K$ be a field with $K^{\alg} = K\IQ^{\alg}$. Then
    \[
        G_K \cong G_{\IQ_p} \Longleftrightarrow K \equiv \IQ_p.
    \]
\end{thm}
Throughout, we will tacitly assume that isomorphisms between profinite groups are topological group isomorphisms.

Pop conjectured that his theorem holds without any extra restrictions on $K$ \cite[E.7]{Pop88}. The conjecture was independently solved by Efrat \cite{Efrat95} (for $p \ne 2$) and Koenigsmann \cite{Koenigsmann95} (for all primes $p$).

\begin{mainthm*}[Efrat/Koenigsmann 1995]
\label{thm:main theorem}
For any field $K$,
\[
    G_K \cong G_{\IQ_p} \ifftext \text{$K$ is $p$-adically closed.}
\]
\end{mainthm*}

To be more precise, both Efrat and Koenigsmann deal with the more general case of finite extensions of $\IQ_p$ ($p$-adic fields). In the literature, the notion of $p$-adically closed (in a wider sense) is sometimes used to denote a field that is elementarily equivalent to some $p$-adic field. It is shown in \cite[Thm.~A]{Efrat95} and \cite[Thm.~4.1]{Koenigsmann95} that a field $K$ is $p$-adically closed in this wider sense if and only if $G_K$ is an open subgroup of $G_{\IQ_p}$. One can show that the Main Theorem follows as a corollary \cite[Cor.~4.2]{Koenigsmann95}.

Perhaps the most notable feature of the Main Theorem is the reappearance of the same phenomenon we have encountered in the case of real closed fields: the absolute Galois group of $\IQ_p$ \emph{encodes extra structure on $K$}---namely, that of a valuation (of any potential rank), satisfying the axioms (i)--(iv) in Theorem~\ref{thm:Ax-Kochen-axioms}. This provides a striking piece of evidence for the point of view that higher rank valuations appear ``in nature''---even outside of model theory.

The main goal of our paper is to give a new and elementary proof of the Main Theorem. A first sketch of such an elementary proof was given by the second author in a lecture series ``Encoding arithmetic in Galois groups''\nocite{Koenigsmann22} at Oxford in Trinity Term 2022. Efrat and Koenigsmann's orig\-inal arguments~\cite{Efrat95, Koenigsmann95} build on Pop's earlier work~\cite{Pop88}, which analyses Galois groups using a range of methods, including Galois cohomology and class field theory.
Their new key ingredient is the ``creation'' of valuations from so-called rigid elements (an idea that had its origins in quadratic form theory).

There are several reasons why we propose a revised proof of the Main Theorem. First, we believe this fundamental result has remained somewhat inaccessible, in part due to its reliance on Galois cohomology and other techniques. This encouraged us to write a self-contained exposition. Secondly, some of the underlying methods were only recently understood, and modern developments shed new light on certain key steps. Additionally, we believe Lemmas~\ref{lem:prime-to-p roots} and \ref{lem:finite->perfectoid} are new; the proofs of Theorem~\ref{thm:Galois code henselianity} in \cite{Koenigsmann03} and \cite[Sec.~5.4]{Engler05} require additional arguments in the case $p = 2$.
Thirdly, we present a modern perspective that highlights interactions with recent developments, while giving a proof that avoids the use of cohomological methods entirely.

\subsection*{Proof strategy}

The main ingredients in our proof are:
\begin{enumerate}[(I)]
    \item A systematic use of coarsenings (Standard Decomposition).
    \item The Galois characterisation of henselianity; in particular,
    \begin{itemize}
        \item the creation of valuations, and
        \item going-down statements for henselianity.
    \end{itemize}
    \item Transfer results of Krasner-Kazhdan-Deligne type.
    \item Explicit constructions replacing abstract cohomological considerations.
    \item Pop's Lemma on mixed characteristic henselian valued fields with absolute Galois group of finite $p$-rank.
\end{enumerate}

The proof proceeds as follows: Let $K$ be a field such that $G_K \cong G_{\IQ_p}$. We begin by constructing a henselian valuation $v_K$ on $K$ using the Galois characterisation of henselianity. Combining explicit criteria inspired by Galois cohomology with a careful analysis of the roots of unity in $K$, we then show that $(K, v_K)$ has mixed characteristic $(0, p)$. This enables us to apply the Standard Decomposition, expressing $v_K$ as a composition of three valuations, with the middle component $\overline{v_p}$ having rank one and mixed characteristic. Utilising Pop's Lemma together with transfer principles of Krasner-Kazhdan-Deligne type, we demonstrate that $\overline{v_p}$ has value group $\IZ$ and residue field $\IF_p$. Finally, we translate these structural properties back to the original valuation $v_K$ and verify that it satisfies all four axioms of Theorem~\ref{thm:Ax-Kochen-axioms}.

\subsection*{Structure of this paper}

Chapter~\ref{chap:val_theory} is a survey of valuation theory, aimed at the reader familiar with rank 1 (i.e., real-valued) valuations---those that correspond to absolute values on fields---but not familiar with higher rank (i.e., Krull) valuations that take values in an ordered abelian group. Some parts of our exposition do not appear in standard textbooks; for example, the Standard Decomposition of a valuation (Theorem~\ref{thm:stand_decomp}), which has become increasingly relevant.

Chapter~\ref{chap:val+abs_Galois} connects valued fields to their absolute Galois groups. Section~\ref{sec:structure of Qp} deals with most of the computational aspects around $\IQ_p$. In Section~\ref{sec:Pop}, we prove Pop's Lemma. In Section~\ref{sec:cohomology disguise}, we present both explicit and elementary criteria for distinguishing fields of characteristic 0 from those of positive characteristic $p$. Although motivated by ideas from Galois cohomology, the results are self-contained and rely only on Hilbert's Theorem 90. In Section~\ref{sec:Galois_char_henselianity}, we prove the Galois characterisation of henselianity, providing a method to recover henselian valuations from absolute Galois groups.

Chapter~\ref{chap:transfer} contains a survey of transfer methods between fields of different characteristic. The main result that we will need for the proof of the Main Theorem is the Transfer Lemma (Corollary~\ref{cor:transfer_lemma1}), for which we give two proofs. In Section~\ref{sec:perfectoid}, we give a new proof based on the tilting construction.
In Section~\ref{sec:sat-decomp_method}, we introduce the Saturation-Decomposition Method---recently employed by Jahnke-Kartas~\cite{Jahnke-Kartas} in a spectacular fashion to understand the model theory of perfectoid fields---and explain how it yields a second proof of the Transfer Lemma. We will motivate this method using ramification theory and provide optimal conditions under which it applies (this is the Taming Theorem~\ref{thm:taming}). As a consequence, we obtain a new proof of a result of Kuhlmann-Rzepka \cite{Kuhlmann-Rzepka} on roughly deeply ramified fields (Corollary~\ref{cor:KR-almost-purity}).
In Section \ref{sec:JK}, we sketch how Jahnke and Kartas' proofs fit into this framework.

Chapter~\ref{chap:main} contains the proof of the Main Theorem.

\textbf{Acknowledgements.}
First and foremost, the authors wish to thank Sylvy Anscombe, Philip Dittmann, Franziska Jahnke, Konstantinos Kartas, and Franz-Viktor Kuhlmann for their continued interest in this project and many discussions. Special thanks go to P.\,D. and Margarete Ketelsen for suggesting counterexamples. We thank Lady Margaret Hall in Oxford for supporting a Mathematical Symposion in September 2024 in Goudargues, France, that inspired us to work on this paper.

The first and third authors warmly thank Franziska Jahnke for inviting them to the University of Münster in 2024/25 for two research visits, during which parts of this paper were written.

The first and second authors gratefully acknowledge the hospitality of the Hausdorff Research Institute in Bonn, where this paper was completed. Their stay was funded by the Deutsche Forschungsgemeinschaft (DFG, German Research Foundation) under Germany's Excellence Strategy -- EXC--2047/1 -- 390685813.

%% file: part2_val_th.tex
\section{Methods in valuation theory}
\label{chap:val_theory}

This chapter will provide us with our basic toolbox. It is not meant to replace textbook treatments such as Engler-Prestel \cite{Engler05} and Efrat \cite{Efrat06}. Rather, we aim for a transparent synopsis. Our main goals in this section are:
\begin{enumerate}[(a)]
    \item Introduce the three equivalent ways of thinking about valuations on a field $K$, as indicated by the following bijections:
    \begin{diagram}
        \{\text{$v$ valuation on $K$}\}/{\sim} \ar[rr, leftrightarrow] \ar[rd, leftrightarrow] &&\{\text{$\Oo$ valuation ring on $K$}\} \ar[ld, leftrightarrow]\\
        &\{\text{$\varphi$ place on $K$}\}/{\sim} &
    \end{diagram}
    Our treatment emphasises the places point-of-view (Section \ref{sec:val+val_ring+places}).
    \item Explain how valuations behave under algebraic extensions of fields and the associated ramification (Hilbert-Dedekind) theory (Section \ref{sec:ram_theory}).
    \item Examine how valuations interact with one another when several valuations are present on a given field (Section \ref{sec:coars&decomp}):
    \begin{enumerate}[(i)]
        \item Explain coarsenings, and how they give rise to the concept of decompositions (Section \ref{sec:comp_places}).
        \item Introduce natural choices of valuations in a general abstract setting---first, the coarsenings of the Standard Decomposition (Section \ref{sec:stand_decomp}) and,
        \item secondly, the canonical henselian valuation (Section \ref{sec:can_hens_val}).
    \end{enumerate}
    \item Explain the creation of abstract valuations from rigid and $p$-rigid elements (Section \ref{sec:rigid elements}).
\end{enumerate}

\subsection{Valuations, valuation rings, and places}
\label{sec:val+val_ring+places}

\subsubsection{Ordered abelian groups}
\label{sec:OAG}

Abstract valuations on fields take values in ordered abelian groups, generalising real-valued valuations.

\begin{defi}
A \emph{convex} subgroup $\Delta$ of an ordered abelian group $(\Gamma,\le)$ is a subgroup such that for any positive $\delta \in \Delta$ and $\gamma \in \Gamma$, $0 \le |\gamma| \le \delta$ implies $\gamma \in \Delta$.
We write $\Gamma_{>\gamma}$ for the interval $(\gamma,\infty)$ in $\Gamma$, and
\[
    \Conv(X) \coloneqq \{\gamma \in \Gamma : |\gamma| \le \delta \text{ for some $\delta \in \langle X \rangle$}\}
\]
for the \emph{convex hull} of a set $X \subseteq \Gamma$ (here, $\langle X \rangle$ denotes the group generated by $X$ in $\Gamma$).
We further write $\Conv(\gamma) \coloneqq \Conv(\{\gamma\})$ for $\gamma \in \Gamma$.
\end{defi}
Any convex subgroup $\Delta$ naturally gives rise to an ordering on the quotient group $\Gamma/\Delta$. Crucially, the set of convex subgroups of $\Gamma$ is linearly ordered by inclusion. The \emph{(archimedean) rank} of $\Gamma$ is the cardinality of the set of proper convex subgroups.
The following is classical:
\begin{fact}[Hölder]
Any ordered abelian group that is archimedean\footnote{Strictly speaking, Eudoxian: Archimedes attributes the theory of magnitudes to Eudoxus, but this has been lost to history.}, i.e., of rank 1, embeds into the ordered group of reals.
\end{fact}

For ordered abelian groups, we have an intrinsic notion of infinitesimals:
\begin{defi}
Let $\varepsilon \in \Gamma$ and $\gamma \in \Gamma_{>0}$. We say $\varepsilon$ is \emph{infinitesimal with respect to} $\gamma$ (in symbols: $\varepsilon \ll \gamma$), if $n|\varepsilon| < \gamma$ for all $n \in \IN$.

To any $\gamma \in \Gamma_{>0}$, we associate two canonical convex subgroups: the minimal convex subgroup containing $\gamma$ and the maximal convex subgroup not containing $\gamma$. These are, respectively,
\begin{align*}
    \Gamma^+_{\gamma} & \coloneqq \Conv(\gamma) \\
    \Gamma^-_{\gamma} & \coloneqq \{\varepsilon \in \Gamma : |\varepsilon| \ll \gamma\}.
\end{align*}

\end{defi}
\begin{fact} \label{fact:rank1}
For any $\gamma \in \Gamma_{>0}$, the ordered quotient group $\Gamma^+_{\gamma}/\Gamma^-_{\gamma}$ is of rank 1.
\end{fact}

\subsubsection{Valuations and valuation rings}
\label{sec:val+val_ring}

Basic examples of discrete valuations are the $p$-adic valuation on $\IQ$ and the order of vanishing at $X$ in a field of rational functions $k(X)$. Abstract (Krull) valuations generalise this notion.
\begin{defi}
    A \emph{valuation} on a field $K$ with values in an ordered abelian group $(\Gamma,\le)$ is a surjective map
    \[
        v: K \epi \Gamma \cup \{\infty\}
    \]
    that satisfies the three properties
    \begin{enumerate}[(i)]
        \item $v(x) = \infty$ if and only if $x = 0$; \label{axiom_val1}
        \item $v(x \cdot y) = v(x) + v(y)$ (additivity);
        \item $v(x + y) \ge \min \{v(x), v(y)\}$ (ultrametric triangle inequality); \label{axiom_val3}
    \end{enumerate}
    where the relations $\gamma < \infty$ and $\gamma  + \infty = \infty + \gamma = \infty + \infty = \infty$ are assumed for all $\gamma \in \Gamma$. 
    The pair $(K,v)$ is called a \emph{valued field}.

    Associated to any $(K,v)$ are the following basic objects:
    \begin{enumerate}[(a)]
        \item $\Gamma$ is called the \emph{value group} of $v$ and is alternatively denoted by $vK$ or $\Gamma_v$. 
        \item $\Oo_v \coloneqq \{x \in K: v(x) \ge 0\}$ is the \emph{valuation ring} of $v$.
        \item $\Mm_v \coloneqq \{ x\in K: v(x)>0\}$ is the unique maximal ideal of $\Oo_v$.
        \item $\Oo_v/\Mm_v$ is the \emph{residue field} of $v$ and denoted by $Kv$. We sometimes write $xv$, $\overline{x}$, or $\res(x)$ for the residue class $x\Mm_v$ of $x \in \Oo_v$.
        \item The \emph{rank} of $(K,v)$ is the archimedean rank of its value group $\Gamma$.
        \item The pair $(\charK K, \charK Kv)$ is the \emph{characteristic} of $(K,v)$.
        \item The \emph{valuation topology} is induced by basic open sets of balls
        \[
            B_\gamma(x) \coloneqq \{y \in K : v(y - x) > \gamma\}, \quad \text{where $\gamma > 0$, $x \in K,$}
        \]
        making $K$ into a topological field. Whenever we refer to topological properties of a valued field, it is with respect to this topology.
    \end{enumerate}
\end{defi}
\begin{fact}
The invariants above form canonical short exact sequences
        \begin{equation} \label{eq:SES}
            \begin{gathered}
                1 \longrightarrow \Oo_v^\times \longrightarrow K^\times\longrightarrow vK\longrightarrow 1 \\
                1 \longrightarrow 1 + \Mm_v \longrightarrow \Oo_v^\times \longrightarrow (Kv)^\times\longrightarrow 1.
            \end{gathered}
        \end{equation}
\end{fact}

The valuation $v$ can be recovered solely from the ring $\Oo_v$, and, moreover, axioms \ref{axiom_val1}--\ref{axiom_val3} correspond to abstract properties of the ring $\Oo_v$.
\begin{defi}
An (abstract) \emph{valuation ring} is an integral domain $A$ such that for any non-zero element $a \in \Frac(A)$, $a \in A$ or $a^{-1} \in A$.
\end{defi}

Let $K = \Frac(A)$. If $A$ is an abstract valuation ring, then $K$ can be endowed with a canonical valuation with valuation ring $A$. This is the quotient map 
\[
    v_A : K^\times \epi K^{\times}/A^{\times},
\]
where the latter becomes an ordered abelian group by declaring
\[
    xA^{\times} \le yA^{\times} :\Longleftrightarrow \tfrac y x \in A
\]
and switching from multiplicative to additive notation (in particular, $1A^{\times}$ is identified with the neutral element 0). Clearly, $\Oo_{v_A} = A$. Conversely, given a valued field $(K,v)$, the valuations $v$ and $v_{\Oo_v}$ are \emph{equivalent} in the sense of the following commutative diagram:
\begin{center}
\begin{tikzcd}[row sep=small]
    & \Gamma \rlap{$\,\cup\,\{\infty\}$} \\
    K \ar[ru,"v",two heads] \ar[rd,"v_{\Oo_v}",two heads, swap, near end] & \\
    & K^{\times}/\Oo_v^{\times} \rlap{$\,\cup\,\{\infty\}$} \ar[uu,"\cong"]
\end{tikzcd}    
\end{center}

\subsubsection{Places}
\label{sec:places}

We give a third and final description of the concept.

\begin{defi}
A \emph{place} on $K$ is given by a field $k$ together with a surjective map
\[
    \varphi : K \epi k \cup \{\infty\}
\]
that satisfies the three ``ring homomorphism'' properties
\begin{gather*}
    \varphi(x + y) = \varphi(x) + \varphi(y), \\
    \varphi(x \cdot y) = \varphi(x) \cdot \varphi(y), \\
    \varphi(1) = 1,
\end{gather*}
where addition and multiplication on $k$ are extended to $k \cup \{\infty\}$ subject to the relations
\[
    a + \infty = \infty + a = \infty \quad \text{and} \quad a \cdot \infty = \infty \cdot a = \infty \cdot \infty = \infty
\]
for any $a \in k$. We leave $\infty + \infty$, $0 \cdot \infty$, and $\infty \cdot 0$ undefined and the conditions on $\varphi$ vacuous in these three cases.
\end{defi}

The relationship between valuations and places can be described as follows. Given a valuation $v$ on $K$ with residue field $Kv = \Oo_v/\Mm_v$, taking residues induces a place $\varphi_v$:
\begin{align*}
    \varphi_v : K & \longrightarrow Kv \cup \{\infty\} \\
    x & \longmapsto xv \coloneqq \begin{cases}
        x\Mm_v & \text{if $x \in \Oo_v$} \\
        \infty & \text{if $x \notin \Oo_v$.}
    \end{cases}
\end{align*}
Conversely, given a place $\varphi : K \epi k \cup \{\infty\}$, it induces a valuation
\[
    v_{\varphi} : K \epi K^{\times}/\Oo_{\varphi}^{\times} \cup \{\infty\}
\]
via the valuation ring $\Oo_{\varphi} \coloneqq \varphi^{-1}(k)$.

These two constructions are mutually inverse---except that we need to identify two valuations $v_1$, $v_2$ and places $\varphi_1$, $\varphi_2$ according to the following commutative diagrams:

\begin{center}
\begin{tikzcd}[row sep=small]
    & \Gamma_1 \rlap{$\,\cup\,\{\infty\}$} \ar[dd,"\cong"] \\
    K \ar[ru,"v_1",two heads] \ar[rd,"v_2",two heads,swap] & \\
    & \Gamma_2 \rlap{$\,\cup\,\{\infty\}$}
\end{tikzcd}
\hspace{0.2\textwidth}
\begin{tikzcd}[row sep=small]
    & k_1 \rlap{$\,\cup\,\{\infty\}$} \ar[dd,"\cong"] \\
    K \ar[ru,"\varphi_1",two heads] \ar[rd,"\varphi_2",two heads,swap] & \\
    & k_2 \rlap{$\,\cup\,\{\infty\}$}
\end{tikzcd}
\end{center}

Under this identification, $v$ and $v_{\varphi_v}$, respectively, $\varphi$ and $\varphi_{v_{\varphi}}$, become equivalent.

It will be crucial to translate between these three equivalent ways of viewing valuations.

\subsubsection{Extensions of valuations}
\label{sec:ext_of_val}

Let $L/K$ be a field extension, and assume that $L$ and $K$ come equipped with valuations $w$ and $v$, respectively. We say that $(L,w)/(K,v)$ is an \emph{extension of valued fields}, and that $w$ is a \emph{prolongation} of $v$, if $w|_K$ is equivalent to $v$, or, correspondingly, $\Oo_v = \Oo_w \cap K$. Conversely, any valuation on $(K,v)$ can be restricted to any subfield $F$ of $K$ such that $(K,v)/(F,v|_F)$ is an extension of valued fields.

The following facts are fundamental (for proofs, see e.g. \cite[Chap. 3]{Engler05}):

\begin{thm}[Chevalley] \label{thm:chevalley}
Let $L/K$ be an arbitrary extension of fields. Then any valuation $v$ on $K$ extends to some valuation $w$ on $L$.
\end{thm}

\begin{fact} \label{fact:pure_insep}
Let $K$ be a field of characteristic $p > 0$ and $L/K$ a purely inseparable extension. Then any valuation $v$ on $K$ extends uniquely to $L$ via $v(x^{1/p^n}) = \frac{1}{p^n} v(x)$ for any $x^{1/p^n} \in L$.
\end{fact}

\begin{fact} \label{fact:finite_orders}
Let $(L,w)/(K,v)$ be an algebraic extension of valued fields. Then any element in $wL/vK$ has finite order, i.e., $wL$ is contained in the divisible hull of $vK$, and $Lw/Kv$ is an algebraic extension of residue fields.
\end{fact}

\begin{fact} \label{fact:compatibility}
Let $(K,v)$ be a valued field and $L/K$ an algebraic extension. If $w_1$ and $w_2$ denote valuations on $L$ extending $v$, then $\Oo_{w_1} \subseteq \Oo_{w_2}$ implies $\Oo_{w_1} = \Oo_{w_2}$.
\end{fact}

In particular, it follows from the above that if $w_1$ and $w_2$ are not equivalent, then $\Oo_{w_1} \not\subseteq \Oo_{w_2}$ and $\Oo_{w_2} \not\subseteq \Oo_{w_1}$, i.e., the valuation rings of $w_1$ and $w_2$ are \emph{incomparable}.
Incomparable valuation rings satisfy the following form of the Chinese remainder theorem:
\begin{fact}[Weak approximation theorem] \label{fact:weak_approx}
Let $\Oo_{v_1}, \ldots, \Oo_{v_n}$ be pairwise incomparable valuation rings on $K$. Then, for any tuple $(x_i)_{i = 1}^n \in \prod_{i = 1}^n \Oo_{v_i}$, there exists $x \in K$ such that $x - x_i \in \Mm_{v_i}$ for $i = 1, \ldots, n$.
\end{fact}

The following describes the most basic connection to Galois theory:

\begin{thm}[Conjugacy Theorem] 
\label{thm:Conjugacy}
    Let $(K,v)$ be a valued field and $L/K$ a normal algebraic extension. Then $\Aut(L/K)$ acts transitively on the set of extensions $w$ of $v$, i.e., for any two extensions $w$ and $w'$ on $L$, there exists an automorphism $\sigma \in \Aut(L/K)$ such that $\Oo_{w'} = \sigma\Oo_w$.
\end{thm}

Beyond that, one can show that the conjugate roots of a polynomial $f(X)$ vary continuously in the coefficients of $f(X)$.

\begin{fact}[Continuity of roots] \label{fact:cont_roots}
Let $(K,v)$ be a valued field and $\tilde{v}$ an extension of $v$ to $K^{\sepc}$. Assume $f(X) = a_nX^n + \ldots + a_1X + a_0$ is a polynomial over $K$ with pairwise distinct roots $x_1, \ldots, x_n$. Then for any $\gamma \in vK$, there exists $\delta \in vK$ and a permutation $\sigma \in \mathfrak{S}_n$, such that for any polynomial $g(X) = b_nX^n + \ldots + b_1X + b_0 \in K[X]$ with roots $y_1, \ldots, y_n$ satisfying
\[
    \min_{1 \le i \le n} v(a_i - b_i) \ge \delta,
\]
we have
\[
    \min_{1 \le i \le n} \tilde{v}(x_i - y_{\sigma(i)}) \ge \gamma.
\]
\end{fact}
\begin{rem} \label{rem:cont_roots}
Note that if $\gamma > \max_{i \ne j} \tilde{v}(x_i - x_j)$, then additionally,
\[
    \tilde{v}(y_{\sigma(i)} - y_{\sigma(j)}) = \tilde{v}(x_i - x_j).
\]
\end{rem}

The following refines Fact~\ref{fact:finite_orders} into a quantitative form:
\begin{fact}[Fundamental inequality, simple form] \label{fact:fund_ineq}
Let $(L,w)/(K,v)$ be a finite extension of valued fields. Assume that
\[
    A = \{a_i\}_{1 \le i \le e} \subseteq L^\times \and B = \{b_j\}_{1 \le j \le f} \subseteq \Oo_w^\times
\]
are given such that $wA$ is a set of representatives for $wL/vK$ and $Bw \subseteq Lw$ is $Kv$-linearly independent. Then
\[
    A \cdot B = \{a_ib_j\}_{\substack{1 \le i \le e \\ 1 \le j \le f}}
\]
is $K$-linearly independent.
In particular,
\begin{equation} \label{eq:fund_ineq}
    [L : K] \ge (wL : vK) [Lw : Kv].
\end{equation}
\end{fact}

We have the following nomenclature for extensions of valued fields.
\begin{defi} \label{def:ram_terminology}
Let $(L,w)/(K,v)$ be a finite extension of valued fields. The basic ramification invariants of the extension are:
\begin{enumerate}[(a)] 
    \item the \emph{ramification index} $e(w/v) \coloneqq (wL : vK) < \infty$, and
    \item the \emph{inertia degree} $f(w/v) \coloneqq [Lw : Kv] < \infty$.
\end{enumerate}
We say that the extension $(L,w)/(K,v)$ is
\begin{enumerate}[(a)]
    \setcounter{enumi}{2}
    \item \emph{immediate} if $e(w/v) = f(w/v) = 1$;
    \item \emph{inert} if $f(w/v) = [L : K]$;
    \item \emph{unramified} if it is inert and $Lw/Kv$ is separable;
    \item \emph{totally ramified} if $e(w/v) = [L : K]$;
    \item \emph{tame} if $e(w/v)f(w/v) = [L : K]$, $p \nmid e(w/v)$, and $Lw/Kv$ is separable ($\charK Kv = p > 0$);
    \item \emph{totally tamely ramified} if $e(w/v) = [L : K]$ and $p \nmid e(w/v)$ whenever $\charK Kv = p > 0$;
    \item \emph{purely wild} if $e(w/v)$ is a power of $p$ and $Lw/Kv$ is purely inseparable ($\charK Kv = p > 0$).
\end{enumerate}
For any of the above properties $P$, we say that an algebraic extension $(L,w)/(K,v)$ \emph{is $P$} if all finite subextensions are $P$. In residue characteristic 0, by definition, all algebraic extensions are tame, and there are no non-trivial purely wild extensions.
\end{defi}

\subsection{Structure theory} \label{sec:ram_theory}

Ramification theory (sometimes also called Dedekind-Hilbert theory) relates the structure of Galois groups of extensions of valued fields to the structure of their basic invariants, the value group and residue field. For proofs, we point the reader to \cite[Chap. 15--17]{Efrat06}. For valuations of higher rank, the theory is due to Krull and Deuring.

\subsubsection{General ramification theory} \label{sec:Hilbert theory}

We associate to any Galois extension $(L,w)/(K,v)$ of valued fields a filtration
\[
    \Gal(L/K) \supseteq D_w(L/K) \supseteq I_w(L/K) \supseteq R_w(L/K) \qquad \text{(shortened: $G \supseteq D \supseteq I \supseteq R$)}
\]
of the Galois group $G = \Gal(L/K)$.
\begin{defi}
    We define
    \begin{enumerate}[(a)]
        \item the \emph{decomposition group}
        \[
            D \coloneqq \{\sigma \in G: \sigma \Oo_w = \Oo_w\} = \{\sigma \in G: \sigma(x) -x \in \Oo_w \text{ for all $x \in \Oo_w$}\};
        \]
        \item the \emph{inertia group} $I \coloneqq \{ \sigma \in G: \sigma(x)-x \in \Mm_w \text{ for all $x \in \Oo_w$}\}$;
        \item the \emph{ramification group} $R \coloneqq \{ \sigma \in G: \sigma(x)-x \in x\Mm_w \text{ for all $x \in \Oo_w$}\}$.
    \end{enumerate}
    We denote the corresponding fixed fields by $L^D$, $L^I$, and $L^R$, respectively.
\end{defi}

The groups $D$, $I$, and $R$ encode structural information about the extension $(L,w)/(K,v)$.

\begin{fact}[Decomposition group $D$]
\label{fact:decomposition subgroup}
Let $(L,w)/(K,v)$ be a Galois extension of valued fields. Then
\begin{enumerate}[(i)]
    \item $(L^D, w|_{L^D})/(K,v)$ is immediate. \label{item:D_immediate}
    \item $w$ is the \emph{unique} prolongation of $w|_{L^D}$ to $L$.
    \item $v$ has exactly $r = [L^D : K] = (G : D)$ prolongations of $v$ to $L$.
\end{enumerate}    
\end{fact}

\begin{fact}[Inertia group $I$]
    \label{fact:inertia group}
    The homomorphism
    \begin{align*}
        \Phi: D & \longrightarrow \Aut(Lw/Kv)\\
        \sigma & \longmapsto \bigl(\overline{x} \longmapsto \overline{\sigma(x)}\bigr)
    \end{align*}
    is well-defined. It induces a short exact sequence
    \begin{equation}
        \label{eq: 1st Exact Sequence}
        \begin{tikzcd}
          1 \ar[r] & I \ar[r] & D \ar[r, "\Phi"] & \Aut(Lw/Kv) \ar[r] & 1,
        \end{tikzcd}
    \end{equation}
    which we call the \emph{First Exact Sequence}. In particular, $I$ is normal and closed in $D$.
    Moreover,
    \begin{enumerate}[(i)]
        \item $L^I w|_{L^I}/Kv$ is relatively separably closed in $Lw/Kv$.
        \item $L^I/L^D$ is unramified.
        In particular, $w|_{L^I}L^I = vK$.
    \end{enumerate}
\end{fact}

\begin{fact}[Ramification group $R$]
    \label{fact:ramification group}
    Let $\mu_{Lw}$ denote the set of roots of unity in $Lw$. Then, the homomorphism 
    \begin{align*}
        \Psi : I & \longrightarrow \Hom(wL/vK, \mu_{Lw})\\
        \sigma & \longmapsto \Bigl(wx + vK \longmapsto \overline{\sigma(x)/x}\Bigr)
    \end{align*}
    is well-defined and induces a short exact sequence
    \begin{equation}
        \label{eq: 2nd Exact Sequence}
        \begin{tikzcd}
          1 \ar[r] & R \ar[r] & I \ar[r, "\Psi"] & \Hom(wL/vK, \mu_{Lw}) \ar[r] & 1,
        \end{tikzcd}
    \end{equation}
    which we call the \emph{Second Exact Sequence}. Moreover,
    \begin{enumerate}[(i)]
        \item $R$ is a normal closed subgroup of $I$;
        \item $R$ is the unique normal Sylow $p$-subgroup of $I$, whenever $\charK Kv = p > 0$. It is trivial if $\charK Kv = 0$. \label{item:R_trivial}
        \item $L^R/L^I$ is totally tamely ramified and maximal with this property.
    \end{enumerate}
\end{fact}

These findings are summarised in the diagram on the next page.
\begin{figure}[ht]
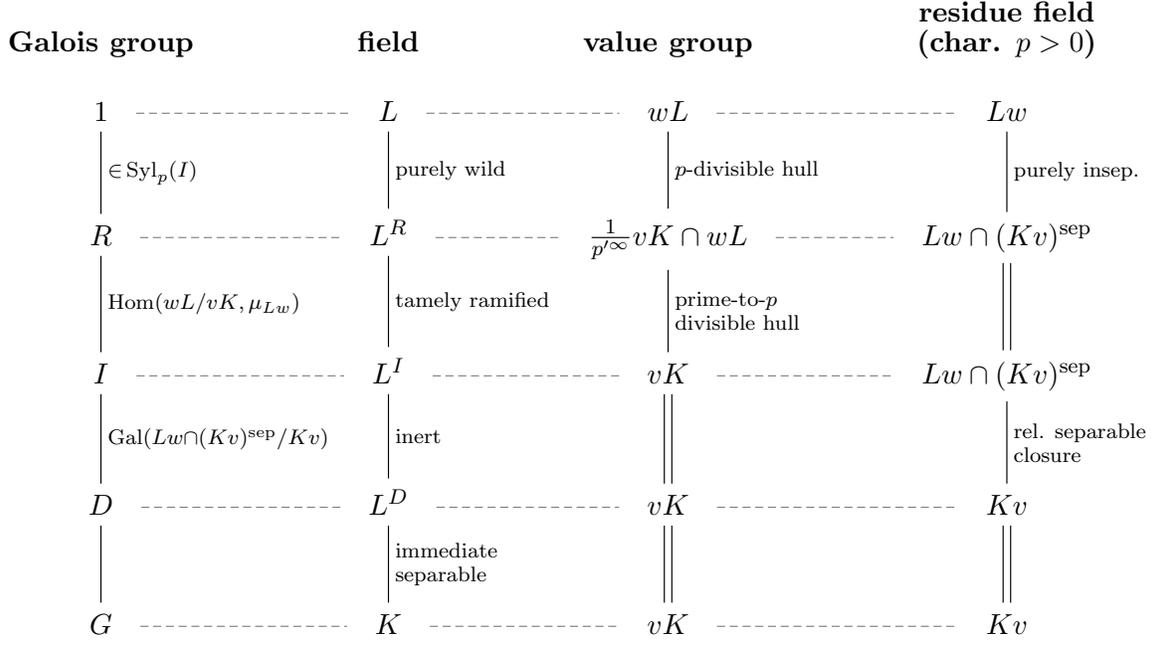

\centering
\newcommand{\shorten}{0.2cm}
\newcommand{\equshift}{0.05cm}
\newcommand{\parboxlength}{2cm}
\tikzcdset{semi-dash/.style={dashed, no head, draw opacity=0.5}}
\begin{diagram}[column sep=huge, row sep= large]
    \text{\textbf{Galois group}} & \text{\textbf{field}} & \text{\textbf{value group}} &
    \text{\textbf{\shortstack{residue field\\ (char. $p > 0$)}}} \\ \\[-1.8cm]
        1 \ar[d, dash, "\in\,\Syl_p(I)"] \ar[r, semi-dash, shorten = \shorten]
            & L \ar[d, dash, "\text{purely wild}"] \ar[r, semi-dash, shorten = \shorten]
                & wL \ar[d, dash, "\text{$p$-divisible hull}"] \ar[r, semi-dash, shorten = \shorten]
                    & Lw  \ar[d, dash, "\text{purely insep.}"] \\
     R \ar[d, dash, "\Hom(wL/vK{,\,}\mu_{Lw})"]  \ar[r, semi-dash, shorten = \shorten]
            & L^R \ar[d, dash, "\text{\parbox{\parboxlength}{\setlength{\baselineskip}{0.7\baselineskip} totally tamely\\ ramified}}"] \ar[r, semi-dash, shorten = \shorten]
                &  \frac{1}{p'^\infty}vK \cap wL \ar[d, dash, "\text{\parbox{\parboxlength}{\setlength{\baselineskip}{0.7\baselineskip} prime-to-$p$\\ divisible hull}}"] \ar[r, semi-dash, shorten = \shorten]
                    & Lw \cap (Kv)^{\sepc} \ar[d, dash, xshift = \equshift] \ar[d, dash, xshift = -\equshift] \\
     I \ar[d, dash, "\Gal(Lw\cap (Kv)^{\sepc}/Kv)"] \ar[r, semi-dash, shorten = \shorten]
        & L^I  \ar[d, dash, "\text{unramified}"] \ar[r, semi-dash, shorten = \shorten]
            & vK \ar[d, dash, xshift = \equshift] \ar[d, dash, xshift = -\equshift] \ar[r, semi-dash, shorten = \shorten]
                & Lw \cap (Kv)^{\sepc} \ar[d, dash, "\text{\parbox{\parboxlength}{\setlength{\baselineskip}{0.7\baselineskip} rel. separable\\ closure}}"] \\
     D \ar[d, dash] \ar[r, semi-dash, shorten = \shorten]
        & L^D  \ar[d, dash, "\text{\parbox{\parboxlength}{\setlength{\baselineskip}{0.7\baselineskip} immediate\\ separable}}"] \ar[r, semi-dash, shorten = \shorten]
            & vK \ar[d, dash, xshift = \equshift] \ar[d, dash, xshift = -\equshift] \ar[r, semi-dash, shorten = \shorten]
                & Kv \ar[d, dash, xshift = \equshift] \ar[d, dash, xshift = -\equshift] \\
     G \ar[r, semi-dash, shorten = \shorten]
        & K \ar[r, semi-dash, shorten = \shorten]
            & vK \ar[r, semi-dash, shorten = \shorten]
                & Kv
\end{diagram}
\caption*{\textsc{Figure.} Diagram summarising the properties of the decomposition subgroup $D$, the inertia group $I$, and the ramification group $R$, as well as their respective fixed fields and corresponding value groups and residue fields (Facts~\ref{fact:decomposition subgroup}, \ref{fact:inertia group}, \ref{fact:ramification group}). For $\charK Kv = 0$, the first two rows collapse. See also \cite[p.~171]{Endler72}.}
\end{figure}

Not all extensions $L/K$ of valued fields are tame, or, more generally, satisfy $[L : K] = ef$.
\begin{example}
Consider $K = \IF_p(t)^{1/p^{\infty}} = \IF_p(t, \ldots, t^{1/p^n}, \ldots)$ with the $t$-adic valuation and $L = K(\alpha)$ with $\alpha^p - \alpha - t^{-1} = 0$. Then $L/K$ is an immediate Artin-Schreier extension of degree $p$, so $[L : K] > ef = 1$.

\end{example}
Ramification theory gives rise to yet another multiplicative invariant called the \emph{defect}, which measures the failure of the fundamental inequality. It does not prominently appear in local class field theory, because separable extensions of discretely valued fields always satisfy the fundamental equality $[L : K] = ef$, cf. \cite[Chap. II, (6.8)]{Neukirch68}.

\begin{cordef}[Lemma of Ostrowski] \label{cor/def:defect}
Let $(K,v)$ be a valued field and $L/K$ a finite extension such that $v$ has a unique prolongation $w$ to $L$. Then
\[
    d(w/v) \coloneqq \frac{[L : K]}{e(w/v) f(w/v)}
\]
is a power of $p$, whenever $\charK Kv = p > 0$, and equal to 1 if $\charK Kv = 0$. The invariant $d(w/v)$ is called the \emph{defect} of $L/K$.
The extension is called \emph{defectless} if $d(w/v) = 1$.
\end{cordef}

One proves the Lemma of Ostrowski by passing to a splitting field and applying the following consequence of the Conjugacy Theorem \ref{thm:Conjugacy} and ramification theory:

\begin{cor}[Fundamental equality]
\label{cor:fundamental equality}
Let $(L,w)/(K,v)$ be a finite Galois extension of valued fields. Let $g$ be the number of prolongations of $v$ to $L$ (all of which are conjugate).
Then
\[
    [L : K] = d\;\!efg,
\]
where $d, e, f$ are the ramification invariants of $(L,w)/(L^D,w|_{L^D})$.
\end{cor}

As our main concern is with absolute Galois groups, we now consider the case $L = K^{\sepc}$.

\begin{defi} \label{def:abs_ram}
For a fixed prolongation $\tilde{v}$ of $v$ to $K^{\sepc}$, write
\begin{alignat*}{2}
    D_K & \coloneqq D_{\tilde{v}}&&(K^{\sepc}/K) \\
    I_K & \coloneqq \, I_{\tilde{v}}&&(K^{\sepc}/K) \\
    R_K & \coloneqq R_{\tilde{v}}&&(K^{\sepc}/K)
\end{alignat*}
to denote the \emph{absolute} decomposition, inertia, and ramification groups, respectively. Furthermore, we say that $(K,v)$ is
\begin{enumerate}[(a)]
    \item \emph{henselian} if $D_K = G_K$ (equivalently, $K = L^D$);
    \item \emph{algebraically maximal} if every finite immediate extension is trivial;
    \item \emph{henselian defectless} if it is henselian and any finite extension is defectless;
    \item \emph{separably tame} if $K$ is henselian and $R_K = 1$ (equivalently, any finite separable extension is tame);
    \item \emph{tame} if $K$ is perfect, henselian, and $R_K = 1$ (equivalently, any finite extension is tame);
    \item \emph{absolutely unramified}\footnote{It is customary to use the term \emph{unramified} to denote a henselian valued field of mixed characteristic $(0,p)$ with $vp$ minimal positive; hence we avoid this terminology here.} if $K$ is perfect, henselian, and $I_K = 1$ (equivalently, any finite extension is unramified).
\end{enumerate}
\end{defi}

In the absolute case, we have the following splitting result (cf. \cite[Cor. 22.1.2, Prop. 22.1.3]{Efrat06}):
\begin{thm} \label{thm:Splitting of D/R}
    Let $L = K^{\sepc}$. There is a semi-direct product decomposition
    \[
        D_K/R_K \cong  \Hom(wL/vK, \mu_{Lw}) \rtimes G_{Kv},
    \]
    where $G_{Kv}$ acts naturally on the left. Moreover,
    \[
        \Hom(wL/vK, \mu_{Lw}) \cong \prod_{q \neq p} \IZ_q^{r_q},
    \]
    where $q$ runs over primes away from $p = \charK Kv$, and
    \[
        r_q \coloneqq \dim_{\IF_q}(vK/qvK)
    \]
    denotes the \emph{$q$-rank} of vK.
\end{thm}

We can use the above results to calculate concrete absolute Galois groups.

\begin{example}
\label{ex:Galois calculations}
\begin{enumerate}
    \item \label{item: C((t))} Consider $K = \IC\laurent{t}$ with the $t$-adic valuation. It has value group $\IZ$ and residue field $\IC$. As a complete field of rank 1, $\IC\laurent{t}$ is henselian (see Section~\ref{sec:hens}).
    By Fact~\ref{fact:ramification group}\ref{item:R_trivial}, $R_{\IC\laurent{t}} = 1$. Moreover, $\Aut(\IC^{\sepc}/\IC) = 1$.
    Therefore, combining the First (\ref{eq: 1st Exact Sequence}) and Second Exact Sequences (\ref{eq: 2nd Exact Sequence}) with Theorem~\ref{thm:Splitting of D/R}, we obtain
    \[
        G_{\IC\laurent{t}} = D_K \cong I_K \cong \prod_{q} \IZ_q \cong \widehat{\IZ}. \\[-0.6em]
    \]
    Note that any finite field has absolute Galois group isomorphic to $\widehat{\IZ} \cong G_{\IC\laurent{t}}$!
    \item \label{item: K((Gamma))} Let $k$ be an arbitrary field of characteristic 0. Consider the field of Puiseux series
    \[
        K = \bigcup_{n > 0} k\laurent{t^{1/n}}
    \]
    with the $t$-adic valuation. The value group is $\IQ$ and the residue field is $k$. The valuation on $K$ is again henselian. Here, the First and Second Exact Sequences yield $I_K = 1$ and therefore
    \[
        G_K = D_K \cong \Aut(k^{\sepc}/k) = G_k.
    \]
\end{enumerate}
\end{example}

\begin{prop} \label{prop:Syl G_Q_p}
Fix rational primes $p \ne q$. Then any Sylow $q$-subgroup of $G_{\IQ_p}$ is isomorphic to $\IZ_q \rtimes \IZ_q$.
\end{prop}
\begin{proof}
Let $G_q$ be a Sylow $q$-subgroup of $G_{\IQ_p}$, and denote by $(F_q, w_q)$ the fixed field of $G_q$ with prolongation $w_q$ of $v_p$. As $G_q$ is a pro-$q$ group, we must have $R_{F_q} = 1$ by Fact~\ref{fact:ramification group}\ref{item:R_trivial}. Moreover, $w_qF_q = \IZ_{(q)}$ and $G_{F_qw_q} \cong \IZ_q$ by infinite Galois theory. Therefore, Theorem \ref{thm:Splitting of D/R} yields
\[
    G_q \cong G_q/R_{F_q} \cong  \IZ_q \rtimes \IZ_q. \qedhere
\]
\end{proof}
\begin{rem} \label{rem:Z_q-action}
The group action in $\IZ_q \rtimes \IZ_q$ is given by $\phi_{\tau}(\sigma) = \sigma^p$, where $\tau$ is the Frobenius element topologically generating the right copy of $\IZ_q$. This is the natural action given by Theorem~\ref{thm:Splitting of D/R}, cf. also \cite[Ex. 22.1.4]{Efrat06}. In particular, $\IZ_q \rtimes \IZ_q$ is non-abelian.
\end{rem}

The structure of Sylow $p$-subgroups of $G_{\IQ_p}$ is also known, though somewhat more complicated. It was determined by Labute \cite{Labute66}, building on earlier work of Shafarevich \cite{Shafarevitch47}, Demushkin \cite{Demushkin61}, and Serre \cite{Serre62}.

\subsubsection{Henselianity}
\label{sec:hens}

We have already introduced the single most important notion in valuation theory in passing: a valued field $(K,v)$ is called \emph{henselian} if $D_K = G_K$. Equivalently, $(K,v)$ is henselian if and only if $v$ extends uniquely to any fixed algebraic closure $K^{\alg}$ of $K$ (recall that by Chevalley's Theorem~\ref{thm:chevalley}, there exists at least one such extension, and valuations extend uniquely to inseparable extensions, see Fact~\ref{fact:pure_insep}). We point the reader to \cite[Sec. 4.1]{Engler05}.

It is possible to construct a henselian valued field from any given one in a universal fashion:
\begin{defi}
Let $(K,v)$ be a valued field. Fix a separable closure $K^{\sepc}$ of $K$ and a prolongation $\tilde{v}$ of $v$ to $K^{\sepc}$. We define $K^h \coloneqq (K^{\sepc})^D$ to be the fixed field of $K^{\sepc}$ with respect to $D = D_{\tilde{v}}(K^{\sepc}/K)$ and call it a \emph{henselisation} of $K$.
\end{defi}
By construction, $K^h$ is henselian.
A henselisation is only unique up to conjugation by an automorphism $\sigma \in G_K$---nonetheless, one often speaks of ``the'' henselisation of $K$. The extension $K^h/K$ is separable and immediate (Fact~\ref{fact:decomposition subgroup}\ref{item:D_immediate}), and further satisfies the universal property that for any henselian valued field $(L,w)$ extending $(K,v)$, there exists a factorisation
\begin{diagram}
    K^h \ar[r, "\exists", hook] & L \\
    K \ar[u, hook] \ar[ur, hook] &
\end{diagram}
of valued fields.
The terminology arises from the fact that henselian valued fields are precisely those valued fields that satisfy \emph{Hensel's Lemma}, which relates the roots of a polynomial $f(X) \in \Oo_v[X]$ with the roots of its coefficient-wise reduction $\overline{f}(X) \in (Kv)[X]$. Hensel orig\-inally proved his lemma for $\IQ_p$.
\begin{lem}[Hensel's Lemma] \label{lem:hensel}
A valued field $(K,v)$ is henselian if and only if it satisfies the following property: given a monic polynomial $f(X) \in \Oo_v[X]$ and $a \in \Oo_v$ satisfying
\[
    v(f(a)) > 2v(f'(a)),
\]
then $f(X)$ has a unique root $\alpha \in \Oo_v$ with $v(\alpha - a) > v(f'(a))$. In particular, if $f(X)$ is such that the reduced polynomial $\overline{f}(X) \in (Kv)[X]$ has a simple root $\overline{a} \in Kv$, then $f(X)$ has a unique root $\alpha \in \Oo_v$ lifting $\overline{a}$.
\end{lem}

The reasons why henselianity is a central notion are manifold. It plays a role analogous to that of completeness in algebraic number theory: there, instead of the henselisation, one would consider the completion of a global field with respect to a non-archimedean absolute value. Indeed, for all valued fields of rank 1, completeness implies henselianity. Moreover, one may develop ramification theory in an \emph{absolute} sense: working with a henselian valued field $(K,v)$ and a fixed Galois extension $L/K$, the invariants and diagrams of Section~\ref{sec:ram_theory} do not depend on the choice of a valuation $w$ on $L$.
Neukirch's classical text \cite{Neukirch99} adopts this perspective, formulating ramification theory not just for local fields but more generally for henselian valued fields of rank 1.

Another avatar of henselianity is the following equivalent condition, which links the properties of distance (in the sense of the valuation topology) and algebraicity. It is remarkably useful in practice.
\begin{lem}[Krasner's Lemma] \label{lem:krasner}
A valued field $(K,v)$ is henselian if and only if it satisfies the following property: given a separable element $x \in K^{\sepc}$ with conjugates $x = x_1, \ldots, x_n$ and any algebraic element $y \in K^{\alg}$, the condition
\[
    v(y - x) > \max_{2 \le i \le n} v(x_i - x),
\]
implies $K(x) \subseteq K(y)$.
\end{lem}

Perhaps the most far-reaching reason for the importance of henselian valued fields is the fact that henselianity is a \emph{first-order} property (it can be expressed as an axiom scheme via Hensel's Lemma~\ref{lem:hensel}). This allows the class of henselian valued fields (structures) to be studied via model theory. In contrast, being rank 1 complete is not a first-order property (since it entails a cardinality bound on the value group).
In model-theoretic practice, henselianity is used---among other things---to lift residue field embeddings to valued field embeddings, which has model-theoretic consequences through methods pioneered by A. Robinson. One key formative result in this direction is the Ax-Kochen-Ershov Theorem \ref{thm:AKE}, which we will mention at a later stage.

We conclude with two applications of Krasner's Lemma.

\begin{prop}
\label{prop:abs_gal_completion}
Let $(K,v)$ be a henselian valued field of rank 1. Then
\[G_{\widehat{K}} \cong G_{\vphantom{\widehat{K}}K}.\]
\end{prop}
\begin{proof}
Note that $\widehat{K}$ is henselian as a complete rank 1 valued field.
By density of $K$ in $\widehat{K}$ and Krasner's Lemma~\ref{lem:krasner}, $K^{\sepc} \cap \widehat{K} = K$; hence $K^{\sepc}$ and $\widehat{K}$ are linearly disjoint. Therefore, we obtain an injective map between sets of finite Galois extensions over $K$ and $\widehat{K}$,
\begin{equation} \label{eq:bij_Gal}
    \{\text{$E/K$ fin. Galois ext.}\} \longrightarrow \{\text{$F/\widehat{K}$ fin. Galois ext.}\}, \quad E \longmapsto E\widehat{K},
\end{equation}
inducing isomorphisms $\Gal(E/K) \cong \Gal(E\widehat{K}/\widehat{K})$.

If we can show that (\ref{eq:bij_Gal}) is surjective, then passing to inverse limits yields
\[
    G_K \cong \varprojlim_{\text{$E/K$ fin.\,Gal.}} \Aut(E/K) \cong \varprojlim_{\text{$F/\widehat{K}$ fin.\,Gal.}} \Aut(F/\widehat{K}) \cong G_{\widehat{K}}.
\]
So let $\widehat{K}(\alpha)/\widehat{K}$ be a finite Galois extension, and let $f(X) \in \widehat{K}[X]$ be the minimal polynomial of $\alpha$. By continuity of roots (Fact \ref{fact:cont_roots}), we may find a monic polynomial $g(X) \in K[X]$ of the same degree sufficiently close to $f(X)$, such that $g(X)$ has a unique root $\beta$ satisfying
\[
    v(\beta - \alpha) > \max\{v(\sigma(\alpha) - \alpha) : \sigma \in \Gal(F/\widehat{K})\}.
\]
By Krasner's lemma once again, this implies that $F = \widehat{K}(\beta)$ is generated by an element algebraic over $K$. In particular, we map $K(\beta)$ to $F = K(\beta)\widehat{K}$.
\end{proof}

\begin{cor}[``Fundamental Theorem'' of non-archimedean analysis] \label{cor:fund_non-arch}
Let $(K,v)$ be a henselian valued field of rank 1 with $\charK K = 0$. Then the completion of the algebraic closure of $K$ is algebraically closed.
\end{cor}

\subsubsection{\texorpdfstring{$p$}{p}-henselianity}

An important weakening of henselianity in the context of pro-$p$ groups is the notion of \emph{$p$-henselianity}.
\begin{defi}
A valued field $(K,v)$ is called \emph{$p$-henselian} if $v$ extends uniquely to any finite Galois extension of $p$-power degree, or equivalently, to the maximal pro-$p$ Galois extension of $K$, which we denote by $K(p)$.
\end{defi}
Hensel's Lemma still holds with the same statement as in Lemma~\ref{lem:hensel} for polynomials $f(X) \in \Oo_v[X]$ that split in $K(p)$.
The notion of $p$-henselianity admits the following useful characterisation:
\begin{lem}[{\cite[Thm. 4.2.2]{Engler05}}]
    \label{lem:p-henselianity criterion}
    Let $(K,v)$ be a valued field. Then $v$ is $p$-henselian if and only if $v$ extends uniquely to any $C_p$-extension, i.e., any Galois extension of degree $p$.
\end{lem}
\begin{proof}[Proof.]
Without loss of generality, let $F/K$ be a finite Galois extension of $p$-power degree, and assume that $v$ extends uniquely to any $C_p$-extension of $K$. By the Conjugacy Theorem~\ref{thm:Conjugacy}, $G = \Gal(F/K)$ acts transitively on the set of valuation rings $\{\Oo_1, \ldots, \Oo_r\}$ of $F$ extending $\Oo_v$. If $r > 1$, then $D = D_{\Oo_1}(F/K) \lneq G$, and we may choose a maximal proper subgroup $D \le N < G$. For any finite $p$-group, the maximal proper subgroups are precisely the normal subgroups of index $p$ (see, for instance, \cite[Thm. 4.6(ii)]{Rotman95}). So $N \vartriangleleft G$ is of index $p$. By the Galois correspondence, $F^N/K$ is a Galois extension of degree $p$, so $\Oo_v$ extends uniquely to $F^N$ by assumption. Thus, $N = \Gal(F/F^N)$ acts transitively on $\{\Oo_1, \ldots, \Oo_r\}$ as well. But now $D$ is the stabiliser of $\Oo_1$ with respect to the action of either $G$ or $N$, so
\[
    |G| = r \cdot |D| = |N|,
\]
by the Orbit-Stabiliser Theorem, which is a contradiction.
\end{proof}

\begin{cor}
Let $(K,v)$ be a valued field. If $G_K$ is a pro-$p$ group, then $v$ is henselian if and only if $v$ extends uniquely to any $C_p$-extension.
\end{cor}

\begin{cor}[{\cite[Cor. 4.2.4]{Engler05}}]
\label{cor:Kummer p-henselianity criterion}
Let $(K,v)$ be a valued field such that $\charK(Kv) \ne p$ and $\zeta_p \in K$. Then $v$ is $p$-henselian if and only if $1 + \Mm_v \subseteq K^{\times p}$.
\end{cor}
\begin{proof}
$\Rightarrow$. Fix $a \in \Mm_v$ and consider the polynomial $f(X) = X^p - (1 + a) \in \Oo_v[X]$, which splits in $K(p)$. By the $p$-analogue of Hensel's Lemma, it has a root in $K^\times$.

$\Leftarrow$. By Lemma~\ref{lem:p-henselianity criterion} and Kummer theory, it suffices to show that $v$ extends uniquely to $K(x^{1/p})$ for any fixed $x \in K^\times \setminus K^{\times p}$. If $vx \notin pvK$, then
\begin{equation} \label{eq:p-root_min_id}
    w\left(a_0 + a_1x^{1/p} + \ldots + a_{p - 1} x^{(p - 1)/p}\right) = \min_{0 \le i \le p - 1} w\bigl(a_ix^{i/p}\bigr) \quad \text{for all $a_0, \ldots, a_{p - 1} \in K$},
\end{equation}
holds for any extension $w$ of $v$ and thus determines it uniquely.
If $vx \in pvK$, we may assume without loss of generality that $vx = 0$. We must have $\overline{x} \notin (Kv)^{\times p}$, for otherwise, $1 + \Mm_v \subseteq K^{\times p}$ would imply $x \in K^{\times p}$. But then again, the identity (\ref{eq:p-root_min_id}) will hold because $\bigl\{1, \overline{x}^{1/p}, \ldots, \overline{x}^{(p - 1)/p}\bigr\}$ is linearly independent over $Kv$. In this case, $v$ also extends uniquely.
\end{proof}

\subsection{Coarsenings and decomposition}
\label{sec:coars&decomp}

Fix a field $K$. One may endow $K$ with several different valuations (equivalently, valuation rings). The following relations are of particular significance.
\begin{defi}
Let $v,w$ be valuations on a field $K$. Then $w$ is a \emph{coarsening} of $v$ (and $v$ is a \emph{refinement} of $w$) if $\Oo_v \subseteq \Oo_w$. Moreover, we say $v$ and $w$ are \emph{dependent} if $\Oo_v$ and $\Oo_w$ have a non-trivial common coarsening, i.e., the compositum $\Oo_v\Oo_w$ is not equal to $K$. Otherwise, we say $v$ and $w$ are \emph{independent}.
\end{defi}

The following two facts provide a first indication of the usefulness of these notions. Note that Fact~\ref{fact:approx} is a partial strengthening of the Weak Approximation Theorem (Fact~\ref{fact:weak_approx}).

\begin{fact}[{\cite[Sec. 2.3]{Engler05}}]
Let $v, w$ be valuations on a field $K$. Then $v$ and $w$ induce the same topology on $K$ if and only if $\Oo_v$ and $\Oo_w$ are dependent.
\end{fact}

\begin{fact}[Approximation Theorem, {\cite[Thm. 2.4.1]{Engler05}}] \label{fact:approx}
Let $\Oo_{v_1}, \ldots, \Oo_{v_n}$ be pairwise independent, non-trivial valuation rings on $K$. Given tuples $(x_i) \in K^n$ and $(\gamma_i) \in \prod_{i = 1}^n v_iK$, there exists $x \in K$ such that
\[
    v_i(x - x_i) > \gamma_i \quad \text{for $i = 1, \ldots, n$.}
\]
\end{fact}

The full strength of coarsenings only becomes apparent when we shift our focus back to the valuation and place formalism.

\begin{lem}
    \label{lem:coarsening equivalence}
    Let $v,w$ be valuations on a field $K$. The following are equivalent:
    \begin{enumerate}[(i)]
        \item $\Oo_v \subseteq \Oo_w$
        \item $\Oo_v^\times \subseteq \Oo_w^\times$
        \item $\Mm_w \subseteq \Mm_v$ [sic]
        \item there exists a convex subgroup $\Delta \le \Gamma_v$ such that $w$ is given by the composition
        \begin{diagram}
            K^\times \ar[r, "v", two heads] \ar[rrr, bend right=30, "w", two heads] & \Gamma_v \ar[r] & \Gamma_v/\Delta \ar[r, phantom, "\cong"] &[-4ex] \Gamma_w,
        \end{diagram}
        and $v$ induces a valuation $\overline{v}$ on $Kw$ with $\overline{v}(Kw) = \Delta$ and $(Kw)\overline{v} = Kv$.
    \end{enumerate}
\end{lem}
\begin{proof}[Proof (sketch).]
Note that
\[
    K = \Oo_v \sqcup \Mm_v^{-1} = \Oo_w \sqcup \Mm_w^{-1}
\]
can be written as the disjoint union of elements of valuation $\ge 0$ and $<0$. Taking complements, this implies the---at first sight counterintuitive---inclusion $\Mm_w \subseteq \Mm_v$.

Given a convex subgroup $\Delta \le \Gamma_v$, the composition of $v$ with the natural quotient map
\begin{diagram}[cramped]
    w : K^{\times} \ar[r, "v", two heads] & \Gamma_v \ar[r, two heads] & \Gamma_v/\Delta
\end{diagram}
is a valuation with valuation ring
\[
    \Oo_w = \{x \in K : w(x) \ge 0 \} \supseteq \{x \in K : v(x) \ge 0 \} = \Oo_v.
\]
Conversely, given a coarsening $\Oo_v \subseteq \Oo_w$, we obtain a commutative diagram
\begin{center}
\begin{tikzcd}[remember picture]
    K^\times \ar[r, two heads] \ar[rd, "v"'] & K^\times/\Oo_v^\times \ar[r, two heads, "q"] & K^\times/\Oo_w^\times \\
    & \Gamma_v \ar[r, dashed, two heads, "{q'}"] & \Gamma_w
\end{tikzcd}
\begin{tikzpicture}[overlay,remember picture]
    \path (\tikzcdmatrixname-1-2) to node[midway,sloped]{$\cong$}
    (\tikzcdmatrixname-2-2);
    \path (\tikzcdmatrixname-1-3) to node[midway,sloped]{$\cong$}
    (\tikzcdmatrixname-2-3);
\end{tikzpicture}
\end{center}
where the dashed arrow $q'$ is a quotient map induced by $q$ and the two isomorphisms. Unravelling definitions, one sees that $\Delta \coloneqq \ker q' = v(\Oo_w^\times)$ and $w = q' \circ v$. Furthermore, note that $v$ restricts to a homomorphism $\Oo_w^{\times} \epi \Delta$, which induces a well-defined valuation
\begin{diagram}[cramped]
    \overline{v} : (Kw)^{\times} = \Oo_w^{\times}/\Mm_w \ar[r, two heads] & \Delta,
\end{diagram}
where $\Mm_w = \{x \in K : v(x) > \Delta\}$ is the coset of value $\infty$ with respect to $\overline{v}$.
Moreover, we have
\begin{gather*}
    \Oo_{\overline{v}} = \overline{\Oo_v} = \Oo_v/\Mm_w \subseteq Kw = \Oo_w/\Mm_w \\
    \Mm_{\overline{v}} = \overline{\Mm_v} = \Mm_v/\Mm_w,
\end{gather*}
and consequently, $(Kw)\overline{v} = Kv$.
\end{proof}

\subsubsection{Composition of places}
\label{sec:comp_places}

The second---and perhaps most practical---description of coarsenings is given by the formalism of places. Indeed, for a valued field $(K,v)$ with a coarsening $\Oo_w \supseteq \Oo_v$, one can separate the structural information encoded in $v$ into two parts using the coarsening $\Oo_w$. In the mixed characteristic case, there are two canonical choices of coarsenings, leading to the notion of \emph{Standard Decomposition} that we will introduce in Section~\ref{sec:stand_decomp}.

Let us now explain in more detail what is meant by ``structural information'' and how decompositions are obtained.

As seen above, the valuation $w$ can be written as a composition
\begin{diagram}[cramped]
    w : K^{\times} \ar[r,"v"] & \Gamma_v \ar[r] & \Gamma_v/\Delta \cong \Gamma_w,
\end{diagram}
where $\Delta = v(\Oo_w^{\times})$ is a convex subgroup of $\Gamma_v$, and $v$ induces a valuation $\overline{v}$ on the residue field $Kw$. How can we express this setup in terms of places---namely, what is the relationship between $\varphi_v$, $\varphi_w$, and $\varphi_{\overline{v}}$, the places associated to $v$, $w$, and $\overline{v}$, respectively?

Let $a \in K$. By construction, these places act on $a$ according to the following rules:
\[
    a \stackrel{\varphi_w}{\longmapsto} \begin{cases}
        \text{if $a \in \Oo_w:$} & a\Mm_w \stackrel{\varphi_{\overline{v}}}{\longmapsto} \begin{cases}
            a\Mm_v & \text{if $a \in \Oo_v$} \\
            \infty & \text{if $a \notin \Oo_v$}
        \end{cases}\\
        \text{if $a \notin \Oo_w:$} & \infty
    \end{cases}
\]
Therefore, $\varphi_v$ is---as a matter of fact---the composition of $\varphi_w$ and $\varphi_{\overline{v}}$:
\begin{diagram}
    K \ar[r,"\varphi_w"] \ar[rr,bend right=30,"\varphi_v"] & Kw \cup \{\infty\} \ar[r,"\varphi_{\overline{v}}"] & Kv \cup \{\infty\}
\end{diagram}
(Note that here we extend $\varphi_{\overline{v}}$ by the additional assignment $\infty \longmapsto \infty$.)

Conversely, the composition of two places $\varphi_w : K \longrightarrow Kw \cup \{\infty\}$ and $\psi : Kw \longrightarrow k \cup \{\infty\}$,
\begin{diagram}[cramped]
    K \ar[r,"\varphi_w"] & Kw \cup \{\infty\} \ar[r,"\psi"] & k \cup \{\infty\},
\end{diagram}
gives rise to a coarsening $\Oo_w = \varphi_w^{-1}(Kw)$ of the valuation ring $\Oo_v \coloneqq \varphi_w^{-1}(\psi^{-1}(k))$. Again, these two constructions are mutually inverse (subject to the same proviso that equivalent places are identified; see the end of Section~\ref{sec:places}).

\begin{slogan*}
Any coarsening induces a decomposition into two places.
\end{slogan*}

We note that the discussions of Sections~\ref{sec:val+val_ring},~\ref{sec:places},~\ref{sec:coars&decomp}, and~\ref{sec:comp_places} can be succinctly summarised in categorical terms:

\begin{thm}
Let $K$ be a fixed field. The following thin categories are equivalent:
\begin{enumerate}[(a)]
    \item the poset category of valuation rings $\Oo$ on $K$ ordered by inclusion $\subseteq$;
    \item objects are valuations $v$ on $K$; morphisms between $v$ and $v'$ are pairs $(q,\overline{v})$, such that $q$ is a factoring of $v'$ through the value group of $v$ and $\overline{v}$ is the induced valuation on $Kv'$;
    \item objects are places $\varphi$ on $K$; morphisms between $\varphi$ and $\varphi'$ are factorings $\psi$ of $\varphi$ through the residue field of $\varphi'$.
\end{enumerate}
In symbols (omitting $\infty$ for readability):
\[
    \{\Oo \mid\; \subseteq\} \simeq
    \left\{
         \smash[b]{\begin{tabular}[ht]{@{}c@{}}
            $v : K \epi \Gamma$ \\
            valuations
        \end{tabular}} \ \middle|\ \phantom{\varphi' :\;}
        \begin{tikzcd}[cramped, sep=small, remember picture]
            \makebox[0pt][r]{$v :\;$} K \ar[r, two heads] & \Gamma \ar[d,"\, q", two heads] & Kv' \ar[d,"\overline{v}", near start, two heads]\\
            \makebox[0pt][r]{$v' :\;$} K \ar[r, two heads] & \Gamma' & \ker q
        \end{tikzcd}
        \begin{tikzpicture}[overlay,remember picture]
            \path (\tikzcdmatrixname-1-1) to node[midway,sloped]{$=$}
            (\tikzcdmatrixname-2-1);
        \end{tikzpicture}
    \right\} \simeq
    \left\{
        \smash[b]{\begin{tabular}[h]{@{}c@{}}
            $\varphi : K \epi k$ \\
            places
        \end{tabular}} \ \middle|\ \phantom{\varphi' :\;}
        \begin{tikzcd}[cramped, sep=small, remember picture]
            \makebox[0pt][r]{$\varphi :\;$} K \ar[r, two heads] & k \\
            \makebox[0pt][r]{$\varphi' :\;$} K \ar[r, two heads] & F \ar[u,"\,\psi"']
        \end{tikzcd}
        \begin{tikzpicture}[overlay,remember picture]
            \path (\tikzcdmatrixname-1-1) to node[midway,sloped]{$=$}
            (\tikzcdmatrixname-2-1);
        \end{tikzpicture}
    \right\}
\]
\end{thm}

Once and for all, we remark that, by abuse of notation, we will use valuation rings, valuations, and places interchangeably whenever there is no danger of confusion.
As a final remark, we note that one of the key properties of decompositions is that they preserve and reflect henselianity.
\begin{fact}[{\cite[Cor. 4.1.4]{Engler05}}] \label{fact:hensel_preserve}
Let $v$ be a valuation that decomposes as $v = \overline{v} \circ w$. Then $v$ is henselian if and only if $\overline{v}$ and $w$ are henselian.
\end{fact}

\subsubsection{Standard Decomposition} \label{sec:stand_decomp}

For any parameter in the maximal ideal of a valued field, there are two canonical coarsenings. Our description will be place-theoretic. The main theorem is:

\begin{thm}[Standard Decomposition at $\varpi$] \label{thm:stand_decomp}
Let $(K,v)$ be any valued field and $\varpi \in \Mm_v$ a fixed non-zero parameter. Then $v$ can be written as the composition of three places,
\begin{diagram}
    v : K \ar[r,"v_0"] & Kv_0 \cup \{\infty\} \ar[r,"\overline{v_{\varpi}}"] & Kv_{\varpi} \cup \{\infty\} \ar[r,"\dbloverline{v}"] & Kv \cup \{\infty\},
\end{diagram}
where
\begin{enumerate}[(a)]
    \item $v_0$ is the finest coarsening such that $v_0\varpi = 0$;
    \item $v_{\varpi}$ is the coarsest coarsening such that $v_{\varpi}\varpi > 0$;
    \item $\overline{v}$ is the valuation induced by $v$ on $Kv_0$;
    \item $\overline{v_{\varpi}}$ is the valuation induced by $v_{\varpi}$ on $Kv_0$;
    \item $\dbloverline{v}$ is the valuation induced by $v$ on $Kv_{\varpi}$.
\end{enumerate}
Furthermore, these valuations satisfy the following properties:
\begin{enumerate}[(i)]
    \item \label{item:induced value groups} The valuations $\overline{v}$ and $\dbloverline{v}$ have value groups
    \[
        \Gamma_0 \coloneqq \Gamma^+_{v{\varpi}} = \Conv(v{\varpi}) \and \Gamma_{\varpi} \coloneqq \Gamma^-_{v{\varpi}} = \{\gamma \in vK : |\gamma| \ll v{\varpi}\},
    \]
    respectively, and moreover, $\overline{v_{\varpi}}$ is of rank 1 and $\overline{v} = \dbloverline{v} \circ \overline{v_{\varpi}}$.
    \item \label{item:O v_0} The valuation ring of $v_0$ is given by $\Oo_{v_0} = \Oo_v[\varpi^{-1}]$. \label{item:O_v_0}
    \item \label{item:O induced} The valuation rings of $\overline{v}$ and $\dbloverline{v}$ are given by
    \[
        \Oo_{\overline{v}} = \Oo_v \ultra \bigcap_{n \ge 1} \varpi^n\Oo_v \and \Oo_{\dbloverline{v}} = \Oo_v/\rad(\varpi\Oo_v) \cong (\Oo_v/\varpi\Oo_v)_{\Red},
    \]
    where $\rad(\cdot)$ denotes the radical of an ideal and $(\cdot)_{\Red}$ the reduction of a ring.
    \item \label{item:O varpi} The valuation ring of $v_{\varpi}$ is given by $\Oo_{v_{\varpi}} = \{x \in K : x^n\varpi \in \Oo_v \text{ for all $n \in \IN$}\}$.
    \item \label{item:v_0 trivial} If $\IZ v\varpi$ is cofinal in the value group of $v$, then $v_0$ is the trivial valuation. 
    \item \label{item:induced v trivial} If $\overline{v}$ is of rank 1, then $\dbloverline{v}$ is trivial. 
\end{enumerate}
\end{thm}
\begin{proof}
We begin by choosing two convex subgroups $\Gamma_0$ and $\Gamma_{\varpi}$ of $\Gamma = vK$. Let $\Gamma_0 \coloneqq \Gamma_{\varpi}^+$ be the convex hull of $\IZ v\varpi$ and $\Gamma_{\varpi} \coloneqq \Gamma_{\varpi}^-$ the maximal convex subgroup not containing $v\varpi$. In particular, we have $\Gamma_{\varpi} \subsetneq \Gamma_0$. Coarsening by these convex subgroups, we obtain valuations
\[
    \begin{tikzcd}[cramped, row sep=tiny]
        v_0 : K^{\times} \ar[r, two heads] & \Gamma \ar[r, two heads] & \Gamma/\Gamma_0 \\
        v_{\varpi} : K^{\times} \ar[r, two heads] & \Gamma \ar[r, two heads] & \Gamma/\Gamma_{\varpi}.
    \end{tikzcd}
\]
The operation of coarsening by convex subgroups preserves inclusion, so $v_0$ and $v_{\varpi}$ are as in (a) and (b).
By our discussion in Section~\ref{sec:comp_places}, the valuation $v$ can be decomposed as
\[
    \begin{tikzcd}
        K \ar[r,"v_{\varpi}"] & Kv_{\varpi} \cup \{\infty\} \ar[r,"\dbloverline{v}"] & Kv \cup \{\infty\}.
    \end{tikzcd}
\]
Since $v_0$ is a coarsening of $v_{\varpi}$, we can decompose $v_{\varpi}$ further, which yields the decomposition into three places as stated in the theorem.
The value group of $v_{\varpi}$ is $\Gamma/\Gamma_{\varpi}$. The value group of the induced valuation $\overline{v_{\varpi}}$ must consequently be $\Gamma_0/\Gamma_{\varpi}$, which is of rank 1 by Fact~\ref{fact:rank1}.

Consider the three maximal ideals $\Mm_{v_0} \subseteq \Mm_{v_{\varpi}} \subseteq \Mm_v$ of $\Oo_{v_0} \supseteq \Oo_{v_{\varpi}} \supseteq \Oo_v$.\footnote{Recall that passing from valuation rings to their maximal ideals reverses inclusion.} The place $\overline{v}$ is given by
\[
    x\Mm_{v_0} \longmapsto \begin{cases}
            x\Mm_v & \text{if $x \in \Oo_v$} \\
            \infty & \text{if $x \notin \Oo_v$},
        \end{cases}
\]
whereas $\dbloverline{v} \circ \overline{v_{\varpi}}$ is the assignment
\[
    x\Mm_{v_0} \longmapsto \begin{cases}
        \text{if $x \in \Oo_{v_{\varpi}}:$} & x\Mm_{v_{\varpi}} \longmapsto \begin{cases}
            x\Mm_v & \text{if $x \in \Oo_v$} \\
            \infty & \text{if $x \notin \Oo_v$}
        \end{cases}\\
        \text{if $x \notin \Oo_{v_{\varpi}}:$} & \infty \longmapsto \infty.
    \end{cases}
\]
Clearly, these two maps are the same, and hence \ref{item:induced value groups} follows. By definition, $\Oo_{v_0}$ must be the finest coarsening for which $\varpi$ becomes invertible---this proves \ref{item:O v_0}. Further note that $\Oo_{\overline{v}} = \Oo_v/\Mm_{v_0}$ and $\Oo_{\dbloverline{v}} = \Oo_v/\Mm_{v_{\varpi}}$, where the ideals involved are given by
\begin{gather*}
    \Mm_{v_0} = \{x \in \Oo_{v_0} : v_0(x) > 0\} = \{x \in \Oo_v : v(x) > \IZ v\varpi\} = \bigcap_{n \ge 1} \varpi^n\Oo_v \\
    \Mm_{v_{\varpi}} = \{x \in \Oo_{v_{\varpi}} : v_{\varpi}(x) > 0\} = \{x \in \Oo_v : v(x) > \Gamma_{\varpi}\} = \rad(\varpi\Oo_v).
\end{gather*}
This yields \ref{item:O induced} and \ref{item:O varpi}.

If $\IZ v\varpi$ is cofinal in $\Gamma$, then $\Gamma = \Gamma_0$ and $v_0$ is trivial. If $\Gamma_0$ has rank 1, then $\Gamma_{\varpi} = \{0\}$ and therefore $v = v_{\varpi}$, so the induced valuation $\dbloverline{v}$ must be trivial. This proves \ref{item:v_0 trivial} and \ref{item:induced v trivial}.
\end{proof}

To simplify notation even further, we will from now on omit the symbol ``$\infty$'' when working with places.

\begin{rem}[Standard Decomposition for mixed characteristic] \label{rem:stand_decomp_mixed}
For valued fields of mixed characteristic $(0,p)$, the key distinguished example is the \emph{Standard Decomposition at $\varpi = p$}.
In this case, the decomposition can be written as follows:
\begin{diagram}
    K \ar[r,"(0{,}0)","\Gamma/\Gamma_0"'] \ar[rr,bend right=49,"\Gamma/\Gamma_p"] & Kv_0 \ar[r,"(0{,}\,p)","\Gamma_0/\Gamma_p"'] \ar[rr,bend left=40,"\Gamma_0"] & Kv_p \ar[r,"(p{,}\,p)","\Gamma_p"'] & Kv
\end{diagram}
Each arrow is labelled with the value group of the corresponding valuation. The three middle arrows are additionally labelled with the characteristic of the valuations.

The valued field $(Kv_0,\overline{v})$ is often called the \emph{core field} of $(K,v)$.
Note that here,
\begin{enumerate}[(a)]
    \item $v_0$ is the finest coarsening with residue characteristic 0, and
    \item $v_p$ is the coarsest coarsening with residue characteristic $p$.
\end{enumerate}
What is distinguishing about the parameter $\varpi = p$ is that it isolates the place at which a change in characteristic occurs for valuations of mixed characteristic $(0,p)$.
\end{rem}

\begin{slogan*}
The \emph{(0,\:\!0)} and $(p,p)$-equicharacteristic places in the Standard Decomposition are chosen maximally possible, thus creating a rank 1 place in the middle.
\end{slogan*}

Fortunately, the above notation is consistent with writing $v_p$ for the $p$-adic valuation on $\IQ_p$.

\subsubsection{Canonical henselian valuation}
\label{sec:can_hens_val}

Fix a field $K$ that is not separably closed.
We will introduce the tree structure of the set of henselian valuations on $K$ and determine a distinguished element $v_K$ in it, which we call the \emph{canonical henselian valuation}.
We will show that the canonical henselian valuation admits a ``going-down'' statement (Theorem~\ref{thm:henseling down}). More specifically, for the proof of the Main Theorem, $v_K$ will be the source of a non-trivial valuation that gives $K$ the structure of a $p$-adically closed field.

Everything that follows builds on F. K. Schmidt's Theorem, initially proven for complete rank 1 valued fields, but easily adapted to the henselian setting. Recall that two valuations $v$ and $w$ are \emph{independent} if any common coarsening is trivial, i.e., $\Oo_v\Oo_w = K$ (equivalently: they induce different topologies on $K$).

\begin{thm}[F. K. Schmidt] \label{thm:schmidt}
If $K$ admits two independent and non-trivial henselian valuations, then $K$ must be separably closed.
\end{thm}
\begin{proof}
Let $v_1$ and $v_2$ be independent non-trivial henselian valuations and assume $K \ne K^{\sepc}$.
Furthermore, let $f(X) \in K[X]$ be a separable irreducible polynomial and $g(X) \in K[X]$ any polynomial of the same degree $n > 1$. If $g(X)$ is sufficiently close to $f(X)$ in each coefficient (with respect to a henselian valuation), then by continuity of roots (Fact~\ref{fact:cont_roots}), Remark~\ref{rem:cont_roots}, and Krasner's Lemma~\ref{lem:krasner}, $f(X)$ and $g(X)$ have the same splitting field. Now let $h(X)$ be any polynomial with precisely $n$ distinct roots in $K$ (i.e., a product of linear factors). By the Approximation Theorem (Fact~\ref{fact:approx}), we may choose $g(X)$ sufficiently close to $f(X)$ with respect to $v_1$ and, at the same time, sufficiently close to $h(X)$ with respect to $v_2$, so that $f(X)$ and $h(X)$ have the same splitting field. This is a contradiction.
\end{proof}

\begin{example}
Let $\IC_p$ be the completion of $\IQ_p^{\alg}$. By Fact~\ref{cor:fund_non-arch}, $\IC_p$ is algebraically closed. If $q \ne p$ is another prime, then there exists a non-canonical field isomorphism $\IC_p \cong \IC_q$ between algebraically closed fields of characteristic 0 and of size continuum (we forget the topology). In this way, $\IC_p$ can be endowed with a natural $p$-adic topology and an abstract $q$-adic topology induced by the non-canonical isomorphism. The point of F. K. Schmidt's Theorem is that such an anomaly can only occur if $K$ is separably closed.
\end{example}

\begin{defi}
Let $K \ne K^{\sepc}$ be a field. We partition the set of henselian valuation rings on $K$ into two sets
\begin{align*}
    H_1(K) & \coloneqq \{\Oo_v : \text{$\Oo_v$ henselian with $Kv$ \underline{not} separably closed}\} \\
    H_2(K) & \coloneqq \{\Oo_v : \text{$\Oo_v$ henselian with $Kv$ separably closed}\}.
\end{align*}
\end{defi}

\begin{lem}[Tree structure] \label{lem:tree}
The coarsening relation induces the following structure:
\begin{enumerate}[(i)]
    \item $H_1(K)$ is non-empty (it contains the trivial valuation) and linearly ordered by $\subseteq$.
    \item Any valuation in $H_1(K)$ is coarser than any valuation in $H_2(K)$.
    \item Any non-empty subset of $H_2(K)$ has a supremum (finest common coarsening) in $H_2(K)$.
    \item If $H_2(K) = \emptyset$, then $H_1(K)$ has a minimum (a finest coarsening).
\end{enumerate}

\end{lem}
In short, we get the following picture:

\tikzset{every picture/.style={line width=0.75pt}}
\usetikzlibrary{decorations.pathreplacing}
\begin{center}
\hspace*{85pt}
\begin{tikzpicture}[x=1pt,y=1pt]
\draw (180, -37.5) -- (180, -112.5);
\fill (180, -112.5) circle (.3ex);
\node[left] at (180, -112.5) {$v_K$};
\draw (180, -112.5) -- (120.6, -187.5);
\draw (180, -112.5) -- (240, -187.5);
\draw (150, -150) -- (180, -187.5);
\draw (150, -187.5) -- (165, -168.75);
\draw (225, -168.75) -- (210, -187.5);
\draw [decorate, decoration={brace,mirror}, thick]
  (255, -106) -- (255, -37) node[midway, right=5pt] {$H_1(K)$};
\draw [decorate, decoration={brace,mirror}, thick]
  (255, -187.5) -- (255, -108) node[midway, right=5pt] {$H_2(K)$};
\draw[<->] (340, -37) -- (340, -187.5)
  node[pos=0.2, right=3pt] {\small coarser}
  node[pos=0.8, right=3pt] {\small finer};
\end{tikzpicture}
\end{center}

\begin{proof}[Proof (sketch).]
All valuations we consider will be henselian by virtue of Fact~\ref{fact:hensel_preserve}.
Let $\Oo_w = \Oo_u\Oo_v$ be the finest common coarsening of two incomparable henselian valuation rings $\Oo_u$ and $\Oo_v$ on $K$. The proof boils down to the following observation: the induced valuations $\overline{u}$ and $\overline{v}$ on $Kw$ are---by design---non-trivial independent henselian valuations. By F. K. Schmidt's Theorem~\ref{thm:schmidt}, this implies that $Kw$ is separably closed, which in turn implies that $Ku = (Kw)\overline{u}$ and $Kv = (Kw)\overline{v}$ are separably closed as well. This proves (i), (ii), and the fact that any finite subset in $H_2(K)$ has a supremum in $H_2(K)$. One can further show that for any non-empty subset $S \subseteq H_2(K)$, its supremum $\Oo_S$ (i.e., the ring generated by all $\Oo \in S$) is given by
\begin{equation} \label{eq:rest_prod_val}
    \Oo_S = \left\{x_1 \cdots x_n : x_1, \ldots, x_n \in \bigcup_{\Oo \in S} \Oo,\, n \in \IN \right\},
\end{equation}
and that it lies in $H_2(K)$ (cf. \cite[Thm.~4.4.2]{Engler05}). Likewise, one can see that
\[
    \Oo_1 \coloneqq \bigcap_{\Oo \in H_1(K)} \Oo
\]
is a valuation ring and the infimum of $H_1(K)$.
\end{proof}

\begin{defi}[Canonical henselian valuation]
    Let $K \ne K^{\sepc}$ be a field. The \emph{canonical henselian valuation} $v_K$ (with valuation ring $\Oo_{v_K}$) is defined to be
    \begin{enumerate}[(i)]
        \item the coarsest valuation in $H_2(K)$ if $H_2(K) \neq \emptyset$, and
        \item the finest valuation in $H_1(K)$ otherwise.
    \end{enumerate}
    If $H_2(K) \neq \emptyset$, then $\Oo_{v_K}$ does not lie in $H_1(K)$. We further define:
    \[
        H(K) \coloneqq H_1(K) \cup \{\Oo_{v_K}\}
    \]
    Note that any valuation in $H(K)$ is comparable to any other henselian valuation and that $\Oo_{v_K}$ is the minimum of $H(K)$.
\end{defi}

\begin{example} \label{ex:valuations on Qp}
     The canonical henselian valuation of $\IQ_p$ is the $p$-adic valuation $v_p$. The only other henselian valuation on $\IQ_p$ is the trivial one.
\end{example}
\begin{proof}
    As $\IQ_pv_p = \IF_p$ is not separably closed, $\Oo_{v_p}$ is comparable to all other henselian valuation rings. It does not admit any proper refinement, since such a refinement would induce a non-trivial valuation on $\IF_p$, which does not exist. Similarly, the value group of any coarsening is a quotient of $\IZ$ by a convex subgroup. The only non-trivial convex subgroup is $\IZ$ itself---the corresponding coarsening is the trivial valuation.
\end{proof}

\subsubsection{``Henseling Down''}
\label{sec:henseling down}

By definition, henselianity ``goes up''---any algebraic extension of a henselian valued field is henselian. The following partial converse was established only in 2003.

\begin{thm}[Henseling Down, {\cite[Prop. 3.1]{Koenigsmann03}}]
    \label{thm:henseling down}
    Let $L/K$ be an algebraic extension with $L$ not separably closed. Consider a valuation $w$ on $L$ with $\Oo_w \in H(L)$.
    Then, the valuation ring of the restriction $v \coloneqq w|_K$  lies in $H(K)$ in each of the following cases:
    \begin{enumerate}[(i)]
        \item $L/K$ is normal.
        \item $L/K$ is finite.
        \item $G_L$ is a Sylow $p$-subgroup of $G_K$ for some prime $p$. If $p = 2$, additionally assume that no proper coarsening of $w$ has real closed residue field. \label{item:hens_down3}
    \end{enumerate}
\end{thm}

\begin{proof}
    (i) We need to establish two properties: first, that $v$ is henselian, and second, that it is a coarsening of $v_K$. To prove that $v$ is henselian, it suffices to show that $w$ is the unique extension of $v$ to $L$. So let $w'$ be any prolongation of $v$ to $L$. By the Conjugacy Theorem~\ref{thm:Conjugacy}, $w'$ is henselian. Since $\Oo_w \in H(L)$, it follows that $\Oo_w$ and $\Oo_{w'}$ are comparable. Hence, by Fact~\ref{fact:compatibility}, $\Oo_w = \Oo_{w'}$.

    For the second part, note that if $\Oo_w \in H_1(L)$, then $\Oo_v \in H_1(K)$ by the definition of $H_1(K)$. Thus, it remains to consider the case $w = v_L$. All proper coarsenings of $v_L$ lie in $H_1(L)$, and hence their restrictions to $K$ (and therefore, all proper coarsenings of $v$) lie in $H_1(K)$. This implies that $\Oo_v \in H(K)$ for $w = v_L$.
    
    (ii) Essentially, the strategy is to pass to the splitting field of $L/K$ and apply (i). Let $N$ be the splitting field of $L/K$. Then $w$ has a unique henselian prolongation $w'$ to $N$. If $\Oo_{w'} \in H(N)$, the claim follows directly from (i). Otherwise, $v_N$ is a proper coarsening of $w'$, and both $Nv_N$ and $Nw'$ are separably closed. Hence, by Fact~\ref{fact:compatibility}, $v_N|_L$ must be a proper coarsening of $w$ as well. Since $\Oo_w \in H(L)$, it follows that $L(v_N|_L)$ is not separably closed. However, $Nv_N/L(v_N|_L)$ is finite, so $L(v_N|_L)$ must be real closed by the Artin-Schreier Theorem~\ref{thm:Artin-Schreier}.
    Note that $L(v_N|_L)/K(v_N|_K)$ is finite, and therefore $L(v_N|_L) = K(v_N|_K)$ by the same theorem. In this setup, $v_N$ restricts to a coarsening of $w$ on $L$ and to a coarsening of $v$ on $K$, giving rise to induced valuations $\overline{w}$ and $\overline{v}$:
    \vspace{-0.5em}
    \begin{center}
    \begin{minipage}{.4\textwidth}
        \centering
        \begin{tikzcd}
            L \ar[rr, bend left=37, "w"] \ar[r, "v_N|_L"]& L(v_N|_L) \ar[r, "\overline{w}"] & Lw
        \end{tikzcd}
    \end{minipage}
    \begin{minipage}{.4\textwidth}
        \centering
        \begin{tikzcd}
            K \ar[rr, bend left=37, "v"] \ar[r, "v_N|_K"]& K(v_N|_K) \ar[r, "\overline{v}"] & Kv
        \end{tikzcd}
    \end{minipage}
    \end{center}
    \vspace{-0.5em}
    By (i), $v_N|_K$ is henselian. Hence, by Fact \ref{fact:hensel_preserve}, to show that $v$ is henselian, it suffices to prove that $\overline{v}$ is henselian. Since $w$ is henselian, so is $\overline{w}$ (Fact \ref{fact:hensel_preserve}, again). Moreover, as $v = w|_{K}$, we have $\Oo_{\overline{v}} = \Oo_{\overline{w}}$, which implies that $\overline{v}$ is henselian. We conclude $\Oo_v \in H(K)$ as before.

    (iii) To show that $v$ is henselian, let $F/K$ be a finite Galois extension. We must show that $v$ has a unique prolongation to $F$. Clearly, one such prolongation is $v_1 \coloneqq \widetilde{w}|_F$, where $\widetilde{w}$ denotes the unique prolongation of $w$ to $K^{\sepc}$. Let $v_2$ be any other prolongation of $v$ to $F$, and let $\widetilde{v_2}$ denote a prolongation of $v_2$ to $K^{\sepc}$. By the Conjugacy Theorem~\ref{thm:Conjugacy}, there exists $\sigma \in G_K$ such that $\Oo_{\widetilde{v_2}} = \sigma \Oo_{\widetilde{w}}$. 
    
    Additionally, set $M \coloneqq LF$. Recall that $G_L \in \Syl_p(G_K)$, where $\Syl_p(G_K)$ denotes the set of Sylow $p$-subgroups of $G_K$.
    
    \newcommand{\addcomment}[1]{\llap{\raisebox{0.35ex}{\color{mygreen}\scriptsize$#1$}}}
    \begin{figure}[ht]
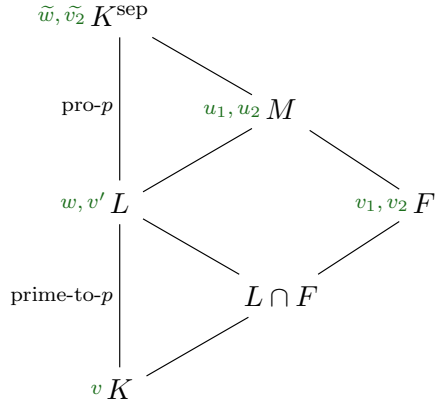

        \begin{diagram}
            \addcomment{\widetilde{w},\widetilde{v_2}\,}K^{\sepc} \ar[dr, dash] \ar[dd, dash, "\text{pro-$p$}"'] && \\
            & \addcomment{u_1, u_2\,}M\ar[dl, dash] \ar[dr, dash] & \\
            \addcomment{w, v'\,}L \ar[dd, dash, "\text{prime-to-$p$}"'] \ar[dr, dash] && \addcomment{v_1, v_2\,}F \ar[dl, dash] \\
            & L \cap F \ar[dl, dash] & \\
            \addcomment{v\,}K &&
        \end{diagram}
        \caption*{\textsc{Figure.} The diagram of fields for our proof. The coloured decorations are the valuations that are considered for each field.}
    \end{figure}
    Since $K^{\sepc}/L$ is a pro-$p$ extension, so is $K^{\sepc}/M$. Similarly, $L/(L \cap F)$ is a prime-to-$p$ extension, hence so is $M/F$. Therefore, $G_M \in \Syl_p G_F$.
    
    As $F/K$ is Galois, we have $\sigma F = F$, so $F \subseteq \sigma M$. Applying the preceding argument to $\sigma M$ in place of $M$ (with $\sigma L$ replacing $L$), we find that $G_{\sigma M} \in \Syl_p G_F$ as well. By the second pro-finite Sylow Theorem (see \cite[Corollary 2.3.6 (c)]{Ribes10}), $G_M$ and $G_{\sigma M}$ are conjugate in $G_F$, so there exists $\tau \in G_F$ with $\tau(M) = \sigma M$. Hence, $u_2 \coloneqq \widetilde{v_2}|_{\sigma M} \circ \tau$ is a henselian prolongation of $v$ to $M$, which restricts to $v_2$ on $F$.

    To show that $v_1 = v_2$, it suffices to prove that $u_1 \coloneqq \widetilde{w}|_M$ and $u_2$ coincide. By Fact~\ref{fact:compatibility}, it is enough to show that $\Oo_{u_1}$ and $\Oo_{u_2}$ are comparable. This holds if $u_1 \in H(M)$ or $u_2 \in H(M)$. Suppose, for contradiction, that $v_M$ is a proper coarsening of both $u_1$ and $u_2$. Then, $v' \coloneqq v_M|_L$ is a proper coarsening of $w = u_1|_L$ (again by Fact~\ref{fact:compatibility}), and hence its residue field $Lv'$ is not separably closed. Consequently, the residue field extension $Mv_M / Lv'$ is finite with $Mv_M$ separably closed, which is only possible if $Lv'$ is real closed. In this case, $[Mv_M : Lv'] = 2$ implies $2 \mid [M : L]$ by the fundamental equality (Corollary~\ref{cor:fundamental equality}), so we must have $p = 2$. Our standing assumption for $p = 2$ excludes the existence of the valuation $v'$, which yields a contradiction.

    Finally, by the same argument as in part (i), we conclude $v \in H(K)$.
\end{proof}

Later on, we will need two additional technical facts about real closed fields to handle the special case of $p = 2$ in the Galois characterisation of henselianity in Section \ref{sec:Galois_char_henselianity}.

\begin{fact}[{\cite[p. 109]{Engler05}}] \label{fact:coarsest_real_closed}
Let $K$ be a field and $\Oo_w \in H_1(K)$ a henselian valuation ring with real closed residue field. Then any henselian refinement $\Oo_v \in H_1(K)$ of $w$ will also have real closed residue field. In particular, if $K$ admits at least one henselian valuation with real closed residue field, then it has a coarsest one.
\end{fact}

The proof relies on the observation that henselian valuation rings on $K$ are convex with respect to any ordering on $K$ (see \cite[Lem.~4.3.6]{Engler05}), and that any ordering on $K$ induces an ordering on the residue field $Kv$ of any convex valuation ring $\Oo_v \subseteq K$.

The following was pointed out to us by Philip Dittmann:
\begin{lem}[cf. {\cite[1.11 Folg.]{Geyer69}}] \label{lem:no_normal_C2}
Let $F$ be a field with pro-$2$ absolute Galois group $G_F$. If $F$ is not real closed, then $G_F$ contains no normal subgroup isomorphic to $C_2$.
\end{lem}
\begin{proof}
Assume for a contradiction that $C_2 \vartriangleleft G_F$ is a proper normal subgroup, and let $R \coloneq (F^{\alg})^{C_2}$. Then $R$ is real closed by Theorem~\ref{thm:Artin-Schreier} and $R/F$ is a Galois extension.
Let $F(\sqrt{x})/F$ be any quadratic extension contained in $R$.
\begin{diagram}[cramped]
    F^{\alg} \ar[d,no head,"C_2"] \\
    R \ar[d,no head] \\
    F(\sqrt{x}) \ar[d,no head,"C_2"] \\
    F
\end{diagram}
The involution on $F(\sqrt{x})$ over $F$ extends to an automorphism of $R$, which induces two different orderings on $F(\sqrt{x})$; $\sqrt{x}$ is positive for one and negative for the other. This contradicts the fact that $R$ can only be the real closure of one of the two orderings on $F(\sqrt{x})$ by the uniqueness of the ordering on $R$.
\end{proof}

\subsection{Creation of valuations}
\label{sec:rigid elements}

One of our main ingredients, the Galois characterisation of henselianity (Theorem~\ref{thm:Galois code henselianity}), requires us to \emph{create} a valuation on a field $K$ based on abstract properties of $K$---out of thin air, as it would seem. The constructions originate in quadratic form theory, developed in the works of Arason, Bröcker, Elman, Jacob, and Ware \cite{Broecker76, Jacob81, Ware81, Ware92, AEJ}, and later adapted by Hwang, Jacob, and Koenigsmann~\cite{HwJ95, Koenigsmann95}.

\begin{leit*}
The purpose of this section is to explain the proof of Theorem~\ref{thm:creation_p-rigid}, which provides a sufficient condition for the existence of a non-trivial valuation with non-$p$-divisible value group on $K$, based on the interplay between $K^{\times p}$ and the additive structure of the field $K$. Historically, the somewhat computational proof was hard to motivate. It turns out that the computations involved work in a much more general setup---this setup, however, is easier to motivate. We follow Efrat's abstract approach (\S11 and \S26.4 in his book \cite{Efrat06}) using $S$-compatible valuations (see Definition~\ref{def:S-comp}). Later, we will specialise to our case of interest.
\end{leit*}

\subsubsection{Creating valuations from rigid elements}

Certainly, any field $K$ can be endowed with the trivial valuation. What we are looking for are abstract properties of $K$ that guarantee the existence of a valuation $v$ that is \emph{as non-trivial as possible}. This is realised by the idea of \emph{multiplicative stratification}, which is made precise as follows:

\begin{defi} \label{def:S-comp}
Let $S \le H$ be multiplicative subgroups of $K^\times$. Define two sets of valuation rings on $K$,
\[
    \Val(S) \coloneqq \{\Oo_v \subseteq K \text{ val. ring} : 1 + \Mm_v \le S\} \text{ \ and \ } \Val(S,H) \coloneqq \{\Oo_v \in \Val(S) : \Oo_v^\times \le H\}.
\]
A valuation $v$ is \emph{$S$-compatible} if $\Oo_v \in \Val(S)$.
\end{defi}

The parameters $S$ and $H$ \emph{stratify} possible valuation rings into the sets $\Val(S)$ and $\Val(S,H)$.
If $v$ is trivial on $K$, then $1 + \Mm_v = \{1\}$ and $\Oo_v^\times = K^\times$. Hence, $S$ and $H$ control the non-triviality of $v$.
Strikingly, $\Val(S)$ and $\Val(S,H)$ fit into a hierarchy analogous to the tree $H_1(K) \cup H_2(K)$ of henselian valuations (cf. Lemma~\ref{lem:tree}). In this sense, the condition $1 + \Mm_v \le S$ defining $\Val(S)$ can be seen as a weak analogue of henselianity.

\begin{prop}[{\cite[Lem. 11.1.2]{Efrat06}}] \label{prop:S-val-tree}
Let $S \le H \le K^\times$ be multiplicative subgroups. The set $\Val(S)$, ordered by the coarsening relation, has the following properties:
\begin{enumerate}[(i)]
    \item $\Val(S)\setminus\Val(S,H)$ is linearly ordered and any $\Oo_w \in \Val(S)\setminus\Val(S,H)$ is coarser than any $\Oo_v \in \Val(S,H)$.
    \item $\Val(S,H)$ is closed under refinements.
    \item Any non-empty subset of $\Val(S,H)$ has a supremum (i.e., a finest common coarsening) in $\Val(S,H)$. If $\Val(S,H) \ne \emptyset$, we write $v(S,H)$ for the coarsest valuation in $\Val(S,H)$. In particular, $v(S,H)$ is comparable to any valuation in $\Val(S)$.
\end{enumerate}
\end{prop}

\tikzset{every picture/.style={line width=0.75pt}}
\begin{center}
\hspace*{100pt}
\begin{tikzpicture}[x=1pt,y=1pt]
\draw (180, -37.5) -- (180, -112.5);
\fill (180, -112.5) circle (.3ex);
\node[left=3pt] at (180, -112.5) {$v(S,H)$};
\draw (180, -112.5) -- (120.6, -187.5);
\draw (180, -112.5) -- (240, -187.5);
\draw (150, -150) -- (180, -187.5);
\draw (150, -187.5) -- (165, -168.75);
\draw (225, -168.75) -- (210, -187.5);
\draw [decorate, decoration={brace,mirror}, thick]
  (255, -106) -- (255, -37) node[midway, right=5pt] {$\Val(S) \setminus \Val(S,H)$};
\draw [decorate, decoration={brace,mirror}, thick]
  (255, -187.5) -- (255, -108) node[midway, right=5pt] {$\Val(S,H)$};
\draw[<->] (365, -37) -- (365, -187.5)
  node[pos=0.2, right=3pt] {\small coarser}
  node[pos=0.8, right=3pt] {\small finer};
\end{tikzpicture}
\end{center}

In our reasoning, the question now is how to identify suitable candidates for subgroups $S \le H \le K^\times$ such that $S$ is sufficiently large, $H$ is sufficiently small, and, moreover, to establish a criterion for the non-emptiness of $\Val(S,H)$.
The key to this is the following explicit description of $v(S,H)$.

\begin{defi} \label{def:O(S,H)}
Given $S \le H \le K^\times$, define
\begin{align*}
    \Oo^-(S,H) & \coloneqq (1 - S)\setminus H \\
    \Oo^+(S,H) & \coloneqq \{x \in H : x\Oo^-(S,H) \subseteq \Oo^-(S,H)\} \\
    \Oo(S,H) & \coloneqq \Oo^-(S,H) \cup \Oo^+(S,H).
\end{align*}
\end{defi}

Consider the case that $S = 1 + \Mm_v$ and $H = \Oo_v^{\times}$ for some valuation $v$ on $K$. Then
\[
    \Oo^-(S,H) = \Mm_v,\ \Oo^+(S,H) = \Oo_v^\times, \text{ and }\ \Oo(S,H) = \Oo_v.
\]
Therefore, $\Oo^-(S,H)$ is supposed to model the maximal ideal, $\Oo^+(S,H)$ the units, and $\Oo(S,H)$ the valuation ring of the valuation we aim to construct.

\begin{prop}[{\cite[Thm. 11.3.6]{Efrat06}}]
The following statements are equivalent:
\begin{enumerate}[(i)]
    \item $\Val(S,H)$ is non-empty;
    \item $\Oo(S,H)$ lies in $\Val(S,H)$;
    \item $\Oo(S,H)$ is equal to the valuation ring $\Oo_{v(S,H)}$ of $v(S,H)$.
\end{enumerate}
\end{prop}

By doing computations with $\Oo(S,H)$, which gives an explicit description of $v(S,H)$, one can derive a criterion for the non-emptiness of $\Val(S,H)$. This, in turn, relies crucially on the notion of \emph{rigidity} coming from the theory of quadratic forms.

\begin{defi}
Fix a subgroup $S \le K^\times$. An element $x \in K^{\times}$ is \emph{$S$-rigid} (or simply: \emph{rigid}) if $S + xS \subseteq S \cup xS$, or equivalently, $1 + xS \subseteq S \cup xS$.
Write
\[
    A(S) \coloneqq \{-x \in K^\times : 1 + x \notin S \cup xS\} = \{x \in K^\times : 1 - x, 1 - x^{-1} \notin S\}
\]
for a distinguished set of additive inverses of non-rigid elements, which we use in the proposition below.
\end{defi}

The notion of rigidity is in harmony with the ultrametric triangle inequality for valuations: if $S = 1 + \Mm_v$, and $H = \Oo_v^\times$, then any element in $K^\times \setminus H$ is rigid and $A(S) = \Oo_v^\times$.
Note that our first goal is to formulate a full criterion determining for which $H$ (depending on $S$) the set $\Val(S,H)$ is non-empty.

\begin{prop} \label{prop:rigid_crit}
Let $S \le H \le K^\times$. The following conditions are equivalent:
\begin{enumerate}[(i)]
    \item $\Val(S,H)$ is non-empty;
    \item $-1$ and $A(S)$ are contained in $H$, and $\Oo^-(S,H)$ satisfies the closure property
    \begin{equation*}
        \Oo^-(S,H)\Oo^-(S,H) \subseteq 1 - S; \tag{$\dagger$} \label{eq:closure}
    \end{equation*}
    \item every $x \in K^\times \setminus H$ is rigid and $\Oo^-(S,H)$ satisfies \emph{(\ref{eq:closure})}.
\end{enumerate}
\end{prop}
\begin{proof}
(i) $\Leftrightarrow$ (ii). See \cite[Thm. 11.3.6]{Efrat06} for the proof of this equivalence.

(ii) $\Leftrightarrow$ (iii). Assume $-1 \in H$ and $A(S) \subseteq H$. Let $x$ be non-rigid, i.e., there exists $s \in S$ such that $1 + xs \notin S \cup xS$. Then $-xs \in A(S) \subseteq H$, and therefore, $x \in H$.

Conversely, assume that each $x \in K^\times \setminus H$ is rigid. Note that $-1$ is non-rigid, so $-1 \in H$. If $x \in A(S)$, then $-x$ is non-rigid, hence $x \in H$.
\end{proof}

This proposition shows that non-rigid elements and (\ref{eq:closure}) are the only obstructions to creating valuations in the aforementioned sense, motivating the following:
\begin{defi}
Fix $S \le K^\times$. We define
\[
    H_S \coloneqq \langle -1, S, A(S) \rangle \le K^\times.
\]
Equivalently, $H_S$ is the group generated by $S$ and all non-rigid elements with respect to $S$. We call $H_S$ the \emph{non-rigid hull of $S$}.
\end{defi}

The last ingredient we need is an additional fact about property (\ref{eq:closure}): it allows us to tweak $H$ slightly to ensure that (\ref{eq:closure}) holds.

\begin{lem}[{\cite[Prop. 11.4.2]{Efrat06}}] \label{lem:rigid_closure}
Given $S \le H_S \le H$, assume that $a,b \in \Oo^-(S,H)$ violate \emph{(\ref{eq:closure})}, i.e., $ab \notin 1 - S$. Then $a^2 \in S$ and $\Oo^-(S,\langle H, a \rangle)$ satisfies \emph{(\ref{eq:closure})}.
\end{lem}

\begin{thm}[Creating valuations from rigid elements] \label{thm:creation_rigid}
The set $\Val(S,H)$ is non-empty if and only if
\begin{enumerate}[(i)]
    \item $H_S \le H$ in case that $\Val(S,H_S) \ne \emptyset$, or
    \item $H_S < H' \le H$, where $(H' : H_S) = 2$, $H' = \langle H, a \rangle$, and $a \in \Oo^-(S,H_S)$ is an element as in the preceding lemma (in case that $\Val(S,H_S) = \emptyset$).
\end{enumerate}
In particular, if no element in $K^\times/S$ has order two, (ii) cannot hold, and we must have $\Val(S,H_S) \ne \emptyset$. In that case, $H_S \ne K^\times$ implies that $\Oo(S,H_S)$ is a non-trivial valuation ring.
\end{thm}
\begin{proof}
In essence, this is \cite[Thm. 11.4.4]{Efrat06}. The theorem follows directly from the above: Proposition~\ref{prop:rigid_crit} shows that $H_S \le H$ is a necessary condition and Lemma~\ref{lem:rigid_closure} describes how to choose $H'$ in the case $\Val(S,H_S) = \emptyset$.
\end{proof}

We may now drop the $S$-compatibility requirement for valuation rings and replace it with a condition on the value group:

\begin{cor}[{\cite[Thm. 11.4.5]{Efrat06}}] \label{cor:nontr_val_crit}
Let $S \le K^\times$ be a subgroup such that no element in $K^\times/S$ has order two. Then the following are equivalent:
\begin{enumerate}[(i)]
    \item there exists a (non-trivial) valuation $v$ on $K^\times$ with $v(K^\times) \ne v(S)$;
    \item the non-rigid hull $H_S$ is a proper subgroup of $K^\times$.
\end{enumerate}
\end{cor}
\begin{proof}
(i) $\Rightarrow$ (ii). Let $H = \Oo_v^\times S = v^{-1}(v(S))$, which is a proper subgroup of $K^\times$ since $v(K^\times) \ne v(S)$. The ultrametric triangle inequality for $v$ implies that $A(S) \subseteq H$. Therefore, $H_S \le H$ are proper subgroups of $K^\times$.

(ii) $\Rightarrow$ (i). By Theorem~\ref{thm:creation_rigid}, $\Oo_v \coloneqq \Oo(S,H_S)$ is a non-trivial valuation ring. We cannot have $v(K^\times) = v(S)$, for then $K^\times = \Oo_v^\times S \le H_S$, but $H_S$ is supposed to be a proper subgroup. Hence, $v$ is non-trivial with $v(K^\times) \ne v(S)$.
\end{proof}

In \cite{Jacob81, Ware81, AEJ, Koenigsmann95}, $\Oo(S,H)$ is constructed directly as in Definition~\ref{def:O(S,H)}. Without the context of the tree of $S$-compatible valuations and its canonical member $v(S,H)$, the construction of the non-trivial valuation ring $\Oo(S,H_S)$, together with the associated choice of $H_S$ via rigid elements, can appear somewhat mysterious (see the discussion in \cite[p.~xi]{Efrat06}).

\subsubsection{Creating valuations from \texorpdfstring{$p$}{p}-rigid elements}

In the context of the Galois characterisation of henselianity, we do not construct $H_S$ directly. Instead, we will see that considerations from Kummer theory more naturally lead to the study of subgroups $T \le K^\times$ that satisfy a weaker form of rigidity.

\begin{defi} \label{def:p-rigid}
Fix a prime $p$, and let $-1 \in S \le K^\times$. We say an element $x \in K^\times$ is \emph{$p$-rigid} if
\vspace{-0.5em}
\[
    S + xS \subseteq \bigcup_{i = 0}^{p - 1} x^iS.
\]
Let $T_S \le K^\times$ be the \emph{non-$p$-rigid hull of $S$}, i.e., the subgroup generated by $S$ and all non-$p$-rigid elements $x \in K^\times$. Equivalently, $T_S$ is generated by $S$ and all $x \in K^\times$ that satisfy
\[
    1 - x \notin \bigcup_{i = 0}^{p - 1} x^iS.
\]
\end{defi}

Note that any $x \in K^\times \setminus S$ is $p$-rigid, whenever we consider a valuation $v$ on $K$ and we choose $S = \Oo_v^\times K^{\times p}$. This example trivialises if $vK$ is $p$-divisible, because then $\Oo_v^\times K^{\times p} = K^\times$.

Since there are more non-rigid elements than non-$p$-rigid elements, it is clear that $T_S \le H_S$. The goal is to bound the index $(H_S : T_S)$, so that the non-triviality of $(K^\times : H_S)$ can be expressed as a condition on $(K^\times : T_S)$. This final step was nicknamed the ``wonderful creation of rigid elements'' in \cite[Chap.~2.5]{Koenigsmann03}.

\begin{lem}[From $p$-rigid to rigid elements]
\label{lem:p-rigid->rigid}
Let $S \le K^\times$ be a subgroup that contains $-1$ and $K^{\times p}$. If $a, b \in A(S)$, then $a$ and $b$ have $\IF_p$-linearly dependent cosets\footnote{Note that since $K^{\times p} \le T_S$, the abelian group $K^\times/T_S$ has exponent $p$.} in $K^\times/T_S$. In particular, $(H_S : T_S) \le p$.
\end{lem}
\begin{proof}
Our proof is an adaptation of \cite[Lem. 26.4.2, Prop. 26.4.3]{Efrat06}, which itself is based on \cite[Prop. 2.5]{HwJ95}.

Let $x,y \in K^\times \setminus T_S$ represent $\IF_p$-linearly independent cosets in $K^\times/T_S$. By assumption, $x, y$, and $xy \in K^\times \setminus T_S$ are $p$-rigid, so we there are unique $i, j, k \in \{0, \ldots, p - 1\}$ such that
\[
    1 - x \in x^iS, \quad 1 - y \in y^jS, \quad 1 - xy \in (xy)^kS.
\]
We expect some kind of relationship between $i$, $j$, and $k$, which we now determine as a first step. We may write
\[
    1 - xy = (1 - x) + x(1 - y) \in x^iS + xy^jS = x^i(S + x^{1 - i}y^jS).
\]
If $(i,j) \ne (1,0)$, then $x^{1 - i}y^j \notin T_S$ is $p$-rigid and hence
\[
    1 - xy \in x^i(S + x^{1 - i}y^jS) \subseteq \bigcup_{l = 0}^{p - 1} x^i(x^{1 - i}y^j)^l S.
\]
Moreover, recall that $1 - xy \in (xy)^kS$.
Viewing the exponents of $x$ and $y$ as $\IF_p$-coordinates in $\langle x, y, S\rangle/S$, we deduce that there exists a unique $l \in \{0, \ldots, p - 1\}$ such that
\[
    k \equiv i + (1 - i)l \equiv jl \pmod{p}.
\]
This implies
\[
    i \equiv (i + j - 1)l \pmod{p} \Longrightarrow ij \equiv (i + j - 1)jl \equiv (i + j - 1)k \pmod{p}.
\]
If $(i,j) \ne (0,1), (1,0)$, then $k \ne 0$ implies $i,j \ne 0$. In that case, we can rearrange as
\begin{equation} \label{eq:mult_identity}
    (1 - i^{-1})(1 - j^{-1}) \equiv (1 - k^{-1}) \pmod{p}.
\end{equation}
This analysis therefore yields the desired relationship between $i, j, k$, and allows us to define a sign function $\chi : K^\times \setminus T_S \longrightarrow \IF_p$ via
\[
    \chi(x) \coloneqq \begin{cases}
        1 - i^{-1} & \text{if $i \ne 0$ and $1 - x \in x^iS$} \\
        0 & \text{if $1 - x \in S$.}
    \end{cases}
\]
From this definition and (\ref{eq:mult_identity}), we conclude that $\chi$ satisfies the properties:
\begin{enumerate}[(i)]
    \item $1 \notin \im(\chi)$;
    \item $\chi(x) = 0$ if and only if $1 - x \in S \cup xS$;
    \item if $x, y \in K^\times \setminus T_S$ have linearly independent cosets in $K^\times/T_S$, and $\chi(x) \ne 0$ or $\chi(y) \ne 0$, then $\chi(xy) = \chi(x) \chi(y)$.
\end{enumerate}
Finally, assume there exist $a, b \in A(S)$ that represent linearly independent cosets in $K^\times/T_S$. We will show that this is impossible by the properties of $\chi$.

Note that (ii) implies $\chi(a), \chi(b) \ne 0$. Using (iii) inductively, we obtain
\[
    \chi(a^{p - 1}b) = \chi(a)^{p - 1} \chi(b) \and \chi(a^{p - 1}b) = \chi(a^{p - 1}) \chi(b).
\]
But since $\chi(a)^{p - 1} = 1$, it follows that $\chi(a^{p - 1}) = 1$, contradicting (i).
\end{proof}

Putting everything together, we conclude:

\begin{thm}[Creating valuations from $p$-rigid elements]
\label{thm:creation_p-rigid}
Let $p$ be any prime, and let $S \le K^\times$ be a subgroup containing $-1$ and $K^{\times p}$. If $(K^\times : T_S) > p$, then there exists a non-trivial valuation $v$ on $K$ with non-$p$-divisible value group, i.e., $vK \ne pvK$.
\end{thm}
\begin{proof}
For $p = 2$, we have $T_S = H_S$. By Theorem~\ref{thm:creation_rigid}, there exists $H_S \le H' \lneq K^\times$ for which $\Oo_v = \Oo(S,H')$ is a non-trivial valuation ring. We cannot have $vK = 2vK$, for otherwise, $K^\times = \Oo_v^\times K^{\times 2} \subseteq H'$.

For $p > 2$, Lemma~\ref{lem:p-rigid->rigid} implies that $(K^\times : H_S) > 1$. Note that any element in $K^\times/S$ has odd order dividing $p$.
Finally, by Corollary~\ref{cor:nontr_val_crit}, there exists a valuation $v$ on $K$ with
\[
    p vK = v(K^{\times p}) \subseteq v(S) \ne v(K^\times) = vK. \qedhere
\]
\end{proof}

As a final remark, let us note that all known proofs of our Main Theorem rely on the creation of valuations: in each case, the above construction serves as the original source of an abstract valuation for fields with absolute Galois group isomorphic to $G_{\IQ_p}$.

\subsection{First-order languages for valued fields}

In Sections~\ref{chap:transfer} and \ref{chap:main}, we view ordered groups, valued fields, etc., as model-theoretic structures.
For the avoidance of doubt, we specify here the precise first-order languages that we use.

We regard fields as structures in the language of rings
\[
    \Ll_{\ring} \coloneqq \{0, 1, +, -, \cdot\},
\]
and ordered abelian groups as structures in the language
\[
    \Ll_{\oag} \coloneqq \{0, +, -, \le\}.
\]
Here, $+, -, \cdot$ are binary function symbols, and $\le$ is a binary relation symbol.
Our language for valued fields is
\[
    \Ll_{\text{VF}} \coloneqq (\mathsf{VF},\mathsf{VG},\mathsf{RF};\, v, \res),
\]
a three-sorted language consisting of sorts for the underlying valued field, the value group, and the residue field. For $\mathsf{VF}$ and $\mathsf{RF}$ we have the same symbols as in $\Ll_{\ring}$, and $\mathsf{VG}$ is given the symbols in $\Ll_{\oag}$; additionally, $\Ll_{\VF}$ contains two unary function symbols $v : \mathsf{VF} \longrightarrow \mathsf{VG}$ and $\res : \mathsf{VF} \longrightarrow \mathsf{RF}$, interpreted as the valuation and residue maps, respectively.

A more economical, one-sorted language for valued fields is
\[
    \Ll_{\val} \coloneqq \Ll_{\ring} \cup \{\Oo\},
\]
where $\Oo$ is a unary predicate interpreted as a valuation ring.
The languages $\Ll_{\VF}$ and $\Ll_{\val}$ have the same expressive power; formulas may differ in complexity.

\begin{fact} \label{fact:val_lang}
The $\Ll_{\VF}$- and $\Ll_{\val}$-structure of a valued field $(K,v)$ are uniformly bi-interpretable. In particular, for any two valued fields $(K,v)$ and $(K,v')$, we have
\[
    (K,vK,Kv;\, v, \res) \equiv_{\Ll_{\VF}} (K',v'K',K'v';\, v', \res') \Longleftrightarrow (K, \Oo_v) \equiv_{\Ll_{\val}} (K', \Oo_{v'}).
\]
\end{fact}

For brevity, we will simply write $(K,v)$ for the entire $\Ll_{\VF}$-structure.

%% file: part3_galois.tex
\section{Valuations and absolute Galois groups}
\label{chap:val+abs_Galois}

\subsection{\texorpdfstring{The explicit structure of $\IQ_p$}{The explicit structure of Qₚ}}
\label{sec:structure of Qp}

\subsubsection{Cyclotomic extensions}

\begin{defi}
Let $F$ be a field. Denote the set of roots of unity in $F$ by $\mu_F$. Whenever $F$ contains all $n$-th roots of unity for some $n \in \IN_{>0}$, we denote this set by $\mu_n$. Furthermore, write $\zeta_n$ for a primitive $n$-th root of unity and $\zeta_{q^{\infty}}$ for the set of all $q$-power roots of unity.
\end{defi}

When several roots of unity are involved, we tacitly assume that they form a compatible system (e.g. $\zeta_{q^2}^q = \zeta_q$).

The following is well-known:

\begin{lem}
	\label{lem:cycExt}
	Let $L = \IQ_p(\mu_n)$ be a cyclotomic extension of $\IQ_p$.
    \begin{enumerate}[(i)]
        \item If $(n,p) = 1$, then $L/\IQ_p$ is unramified of degree $k = \ord_n(p)$, the order of $p$ in $(\IZ/n\IZ)^\times$. In particular,
        \[
            \Gal(L/\IQ_p) \cong \Gal(\IF_{p^k}/\IF_p) \cong C_k.
        \]
        \item If $n = p^r$, then $L/\IQ_p$ is totally ramified of degree $(p - 1)p^{r - 1}$ and $1 - \zeta_{p^r}$ is a uniformiser. Moreover,
        \[
            \Gal(L/\IQ_p) \cong \bigl(\IZ/p^r\IZ\bigr)^\times.
        \]
    \end{enumerate}
    In particular, cyclotomic extensions are abelian.
\end{lem}

We therefore obtain the following description of the roots of unity in $\IQ_p$.
\begin{cor} \label{cor:rootsOfUnity}
In $\IQ_p$, the roots of unity are
\[
    \mu_{\IQ_p} = \begin{cases}
        \mu_{p - 1} & \text{if $p \ne 2$} \\
        \mu_2 & \text{if $p = 2$.}
    \end{cases}
\]
\end{cor}

In the context of our Main Theorem---the Galois-theoretic characterisation of $p$-adically closed fields---we consider the situation where a field $K$ satisfies $G_K \cong G_{\IQ_p}$. To draw conclusions about the roots of unity contained in $K$, we seek a purely Galois-theoretic criterion that identifies the roots of unity in $K$. This motivates the following lemma.

\begin{lem}
    \label{lem:prime-to-p roots}
    Let $(K,v)$ be a henselian valued field, and let $q$ be any prime. Assume that
    \begin{enumerate}[(i)]
        \item $\charK Kv \ne q$;
        \label{item:unitroots_assum1}
        \item $vK$ is not $q$-divisible;
        \label{item:unitroots_assum2}
        \item $\Gal(Kv(q)/Kv) \cong \IZ_q$, where $Kv(q)$ is the maximal pro-$q$ Galois extension. \label{item:unitroots_assum3}
    \end{enumerate}
    Then, for any $r \ge 1$,
    \[
        \zeta_{q^r} \in K \ifftext \text{$K$ admits a $C_{q^r} \times C_{q^r}$-extension.}
    \]
    In particular, the lemma applies to $(K,v) = (\IQ_p,v_p)$ for $p \ne q$.
\end{lem}
\begin{proof}
    Note that the assumptions hold for $K = \IQ_p$ because $Kv = \IF_p$ admits a unique cyclic extension of degree $q^n$ for every $n \ge 1$.

    Assume $\zeta_{q^r} \in K$. By assumption \ref{item:unitroots_assum3}, $K$ admits an unramified $C_{q^r}$-extension. Moreover, by \ref{item:unitroots_assum2}, we may choose $x \in K$ with $vx$ not $q$-divisible and consider the totally ramified Kummer extension $K(\!\!\sqrt[q^r]{x})/K$, which is well-defined by \ref{item:unitroots_assum1} and the existence of a primitive $q^r$-th root of unity. Any unramified and totally ramified extensions of $K$ are linearly disjoint, so taking the compositum of these two $C_{q^r}$-extensions yields the desired $C_{q^r} \times C_{q^r}$-extension.

    Conversely, assume that $K$ has a $C_{q^r} \times C_{q^r}$-extension $L/K$. Consider the maximal unramified extension $K^{\ur}/K$ inside $L$. By assumption \ref{item:unitroots_assum3}, $\Gal(K^{\ur}/K)$ must be cyclic and a quotient of $C_{q^r} \times C_{q^r}$. Therefore, $K$ admits a $C_{q^r}$-extension $E/K$ linearly disjoint from $K^{\ur}/K$. Then, the extension $E/K$ is totally ramified, and tamely so by assumption \ref{item:unitroots_assum1}.
    
    Let $w$ be the unique prolongation of $v$ to $E$. By Hensel's Lemma and \ref{item:unitroots_assum1}, it suffices to show that $\zeta_{q^r} \in Ew = Kv$. We can see that this follows from the Second Exact Sequence (\ref{eq: 2nd Exact Sequence}):
    \[
        \Gal(E/K) = I_w(E/K) \cong \Hom(wE/vK, \mu_{Ew}).
    \]
    Indeed, $\Hom(wE/vK, \mu_{Ew}) \cong C_{q^r}$ if and only if $wE/vK \cong C_{q^r}$ and $\zeta_{q^r} \in \mu_{Ew} = \mu_{Kv}$.
\end{proof}

\begin{rem} \label{rem:prime-to-p roots finite extensions}
Observe that if a valued field $(K,v)$ satisfies \ref{item:unitroots_assum1}--\ref{item:unitroots_assum3} in the above lemma, and $L/K$ is a finite abelian extension with unique prolongation $w$ of $v$, then the valued field $(L,w)$ also satisfies \ref{item:unitroots_assum1}--\ref{item:unitroots_assum3}.
\end{rem}

We will further consider infinite cyclotomic extensions of $K$, to which the following subtle observation applies.
\begin{lem} \label{lem:inf_cyclotom}
Let $F$ be a field with $\charK F \ne q$. Then $F(\zeta_{q^{\infty}}) = F(\zeta_q)$, $F(\zeta_{2^{\infty}}) = F(\zeta_4)$ (if $q = 2$)\footnote{These exceptional cases do occur: for example, consider the field $\IQ(\{\zeta_{q^n} + \zeta_{q^n}^{-1}\}_{n \ge 1}) \subseteq \IR$. Its only root of unity is $\zeta_2$---after adjoining $\zeta_q$ or $\zeta_4$ (if $q = 2$), it will contain all $q^n$-th roots of unity.}, or $\Gal(F(\zeta_{q^{\infty}})/F) \cong \IZ_q \times C$, where $C$ is cyclic of order dividing $q - 1$ or 2 (if $q = 2$).
\end{lem}
\begin{proof}
Consider $F_n \coloneqq F(\zeta_{q^n})$ for $n \in \IN$, which form a tower of field extensions. Assume $q \ne 2$ and that there exists $n \ge 2$ such that $\zeta_{q^n} \notin F_{n - 1}$. If $\zeta_{q^{n + 1}} \in F_n$, then
\[
    (N_{F_n/F_{n - 1}}(\zeta_{q^{n + 1}}))^q = N_{F_n/F_{n - 1}}(\zeta_{q^n}) = \zeta_{q^{n - 1}} \in F_{n - 1}.
\]
However, $\zeta_{q^{n - 1}}$ has no $q$-th root in $F_{n-1}$ by assumption. Hence, $\zeta_{q^{n + 1}}$ cannot lie in $F_n$, and by induction, the $F_n$ form an infinite degree tower, except when $F(\zeta_q) = F(\zeta_{q^\infty})$. The same idea works for $q = 2$, only that a change in sign in the norm forces us to require $n \ge 3$. If we are not in one of the exceptional cases, the tower is of infinite degree, and the claim follows by taking the inverse limit of $\Gal(F_n/F)$.
\end{proof}

\subsubsection{\texorpdfstring{Residues and generators in $p$-adic fields}{Residues and generators in p-adic fields}}

Let $F$ be a $p$-adic field, i.e., a finite extension of $\IQ_p$. In this section, if not stated otherwise, we will always write $v$ for the unique valuation on $F$ extending the $p$-adic valuation on $\IQ_p$, and $\Oo_F$ for its valuation ring.

The goal of this section is to study the multiplicative structure of $p$-adic fields by explicit means.
Our starting point is the filtration
\[
    F^\times \supseteq U_F^{(0)} \supseteq U_F^{(1)} \supseteq \ldots \supseteq U_F^{(k)} \supseteq \ldots
\]
of the multiplicative group $F^\times$ by higher unit groups.

\begin{defi}
Let $F$ be a $p$-adic field with uniformiser $\pi$. The \emph{$k$-th higher unit groups} are defined as
\begin{align*}
    U^{(0)}_F & \coloneqq \Oo_F^\times \\
    U^{(k)}_F & \coloneqq \bigl\{x \in \Oo_F^\times: x \equiv 1 \; (\text{mod }{\pi^k})\bigr\} \quad \text{for $k \ge 1$.}
\end{align*}
\end{defi}

Some of the calculations given below are standard; others may be less well-known.

\begin{fact}
\label{fact:multgrp_direct_prod}
Let $F$ be a $p$-adic field with uniformiser $\pi$. Denote its prime-to-$p$ roots of unity by $\mu'_F$. The multiplicative group $F^\times$ can be decomposed as the direct product
\[
    F^\times = \pi^{\IZ} \times \mu'_F \times U_F^{(1)}.
\]
\end{fact}

\begin{fact}
    \label{fact:TowerLawU}
    Let $F$ be a $p$-adic field. For any $k \ge 1$, there exist isomorphisms
    \begin{align*}
        U_F^{(0)}/U_F^{(1)} & \cong ((Fv)^\times, \cdot) \\
        U_F^{(k)}/U_F^{(k+1)} & \cong (Fv,+).
    \end{align*}
\end{fact}
\begin{proof}
The isomorphisms are given by
\begin{align*}
    xU_F^{(1)} & \longmapsto x + \Mm_F \\
    (1 + \pi^k a)U_F^{(k + 1)} & \longmapsto a + \Mm_F,
\end{align*}
see \cite[Chap. II, (3.10)]{Neukirch99}.
\end{proof}

\begin{rem} \label{rem:mod_rule}
The fact that the above maps are indeed isomorphisms relies on the following rule linking additive and multiplicative congruences: for any $a, a' \in U_F^{(0)}$,
\[
    a \equiv a' \pmod{\pi^k} \ifftext a \equiv a' \quad\bigl(\mathrm{mod}\ U_F^{(k)}\bigr).
\]
The equivalence is easily verified:
\[
    \exists b \in \Oo_F : a' = a + b\pi^k \Longleftrightarrow \exists b \in \Oo_F : \frac{a'}a = 1 + a^{-1}b\pi^k \Longleftrightarrow \frac{a'}a \in U_F^{(k)}.
\]
\end{rem}

\begin{cor} \label{cor:representation}
Let $F$ be a $p$-adic field and $\Aa \subseteq \Oo_F$ a set of representatives for $Fv$. Then any $x \in U_F^{(1)}/U_F^{(k+1)}$ has a unique representation of the form $x = aU_F^{(k+1)}$, where
\[
    a = 1 + a_1 \pi + \ldots + a_k \pi^k,
\]
and all coefficients $a_i$ lie in $\Aa$.
\end{cor}
\begin{proof}
Assume that $a = 1 + \sum_{i = 1}^k a_i\pi^i$ and $a' = 1 + \sum_{i = 1}^k a_i'\pi^i$ represent the same element modulo $U_F^{(k + 1)}$. In other words, there exists $1 + \pi^{k + 1}b \in U_F^{(k + 1)}$ with $a = a'(1 + \pi^{k + 1}b)$. Assume $a_i \ne a_i'$ for some $i$, in which case we choose the minimal such $i$. Then
\[
    v(a - a') = v((a_i - a_i') \pi^i) = v(\pi^i),
\]
which contradicts $a - a' = a' \pi^{k + 1}b$. Therefore, $a_i = a_i'$ for all $i = 1, \ldots, k$.

The fact that any element in $U_F^{(1)}/U_F^{(k+1)}$ can be written as claimed follows from a counting argument. Observe that
\[
    \bigl|U_F^{(1)}/U_F^{(k+1)}\bigr| = |Fv|^k
\]
by Fact~\ref{fact:TowerLawU}, which is exactly the number of distinct representatives we consider.
\end{proof}

The following residue calculation is crucial for the results that follow.

\begin{lem} \label{lem:p/(1-zeta)^p}
Let $F$ be a $p$-adic field containing $\zeta_p$. Then
\[
    \frac{p}{\pi^{p - 1}} \equiv - 1 \pmod{\pi},
\]
where $\pi = \zeta_p - 1$. In particular, $p/\pi^{p - 1}$ has residue $-1$ in the residue field.
\end{lem}
\begin{proof}
Consider
\[
    f(X) = X^{p - 1} + \ldots + X + 1 = \prod_{i = 1}^{p - 1}(X-\zeta_p^i).
\]
Observe that
\[
    \frac{p}{\pi^{p - 1}} =\frac{f(1)}{(\zeta_p - 1)^{p - 1}} = (-1)^{p - 1}\prod_{i = 1}^{p - 1}\biggl(\frac{\zeta^i_p - 1}{\zeta_p - 1}\biggr) = (-1)^{p - 1}\prod_{i = 1}^{p - 1}(1 + \zeta_p + \ldots + \zeta_p^{i - 1}).
\]
For each factor, we compute
\[
    1 + \zeta_p + \ldots + \zeta_p^{i - 1} = 1 + (1 + \pi)+ \ldots + (1 + \pi)^{i - 1} \equiv i \pmod{\pi}.
\]
Therefore, the claim follows from $p \in (\pi)$, $(-1)^{p - 1} \equiv 1 \pmod{p}$, and Wilson's Theorem
\[
    (p - 1)! \equiv -1 \pmod{p}. \qedhere
\]
\end{proof}

\begin{lem}[{cf. \cite[Lem. 3.2]{Koenigsmann03}}]
\label{lem:U_Kpowers_general}
Let $F/\IQ_p(\zeta_p)$ be a finite extension with ramification index $e \coloneqq e(F/\IQ_p(\zeta_p))$. Then
\[
    U_F^{(ep+1)} \subseteq \bigl(U_F^{(e)}\bigr)^p.
\]
Moreover, for $F = \IQ_p(\zeta_p)$, we have $U_F^{(p + 1)} = \bigl(U_F^{(1)}\bigr)^p$.
\end{lem}
\begin{proof}
Recall that $\IQ_p(\zeta_p)/\IQ_p$ is totally ramified of degree $p - 1$ (Lemma \ref{lem:cycExt}). Let $\pi = \zeta_p - 1$ be a uniformiser for $\IQ_p(\zeta_p)$ and $\Pi$ any uniformiser for $F$. If the valuation $v$ on $F$ is normalised such that $vp = 1$, then $v\pi = \tfrac 1{p - 1}$ and $v\Pi = \tfrac 1{e(p - 1)}$. In particular,
\[
    U_F^{(e)} = 1 + (\pi) \quad\text{and}\quad
    U_F^{(ep + 1)} = 1 + (\pi^p\Pi).
\]
Let $x = 1 + \pi^p a$ with $a \in (\Pi)$ be given. Then it suffices to show that
\[
    (1 + \pi X)^p - (1 + \pi^p a) = 0
\]
has a solution in $\Oo_F$. Consider
\[
    f(X) \coloneqq \pi^{-p}\Bigl[(1 + \pi X)^p - (1 + \pi^pa)\Bigr] = 
    X^p - a + \sum_{i = 1}^{p - 1} \binom p i \pi^{i - p} X^i.
\]
Noting that
\begin{alignat*}{2}
    v\left[\binom p i \pi^{i - p}\right] & = v(\pi^{i - 1}) > 0 && \text{ for $i = 2, \ldots, p - 1$}, \\
    p \pi^{1 - p} & \equiv -1 \pmod{\pi}\ && \text{ by Lemma~\ref{lem:p/(1-zeta)^p},}
\end{alignat*}
we conclude that $f(X)$ reduces to $\overline{f}(X) = X^p - X \in (Fv)[X]$. So the first claim follows by Hensel's Lemma.

Now assume $F = \IQ_p(\zeta_p)$ and let $\Pi = \pi$. The above considerations then yield
\[
    (1 + \pi b)^p \equiv 1 + p \pi b + \pi^p b^p \equiv 1 + \pi^p(b^p - b) \pmod{\pi^{p + 1}}
\]
for any $b \in \Oo_F$. Since $Fv = \IF_p$, we have $b^p \equiv b \pmod{\pi}$, and therefore
\[
    \bigl(U_F^{(1)}\bigr)^p \subseteq U_F^{(p + 1)}. \qedhere
\]
\end{proof}

The unit group $U_F^{(ep + 1)}$ will, from now on, play a central role.

\begin{lem} \label{lem:generators}
Let $F/\IQ_p(\zeta_p)$ be a finite extension with uniformiser $\pi$ in $F$. Write $e$ for the ramification index and $f$ for the inertia degree of this extension. Assume further that $\{b_1, \ldots, b_f\} \subseteq \Oo_F$ represent an $\IF_p$-basis for $Fv$. Then
\[
    \{\pi\} \cup \{1 + b_i \pi^k : 1 \le i \le f,\, 1 \le k \le ep\}
\]
represents a generating set in $F^\times/F^{\times p}$. Moreover, for $F = \IQ_p(\zeta_p)$,
\[
    \{\pi\} \cup \{1 + b_1 \pi^k : 1 \le k \le p \}
\]
represents an $\IF_p$-basis of $F^\times/F^{\times p}$.
\end{lem}
\begin{proof}
By the direct sum decomposition of Fact~\ref{fact:multgrp_direct_prod}, we obtain an isomorphism
\[
    F^\times/F^{\times p} \cong \pi^{\IZ}/\pi^{p\IZ} \times U_F^{(1)}/\bigl(U_F^{(1)}\bigr)^p,
\]
where the first factor is generated by $\overline{\pi}$. We now focus on the second factor.

The idea is to compare sets modulo $U_F^{(ep + 1)}$, which we know lies in $\bigl(U_F^{(1)}\bigr)^p$ by Lemma~\ref{lem:U_Kpowers_general}. Therefore, it suffices to show that the elements $1 + b_i \pi^k$ generate $U_F^{(1)}$ modulo $U_F^{(ep + 1)}$.

\textit{Claim. $\langle 1 + b_i \pi^k : 1 \le i \le f,\, l \le k \le ep \rangle \equiv U_F^{(l)}\; \bigl(\mathrm{mod}\ U_F^{(ep + 1)}\bigr)$ for any $l = 1, \ldots, ep + 1$.}

\begin{claimproof}{}
Via backward induction from $l = ep + 1$ to $l = 1$. The case $l = ep + 1$ is trivial. By Fact~\ref{fact:TowerLawU},
\[
    \langle 1 + b_1\pi^l, \ldots, 1 + b_f\pi^l \rangle U_F^{(l + 1)} = U_F^{(l)},
\]
for any $l = 1, \ldots, ep + 1$. This identity proves the inductive step $l + 1 \rightarrow l$. For $l = 1$, we recover our original claim.
\end{claimproof}

Consider now the special case $F = \IQ_p(\zeta_p)$. Recall that by Lemma~\ref{lem:U_Kpowers_general}, 
\[
    U_F^{(1)}/\bigl(U_F^{(1)}\bigr)^p = U_F^{(1)}/U_F^{(p + 1)}.
\]
By the same proof, $\{1 + b_1 \pi^k : 1 \le k \le p\}$ is a generating set of cardinality $p$ for $U_F^{(1)}/U_F^{(p + 1)}$. Hence, it must be a basis, since
\[
    \bigl|U_F^{(1)}/U_F^{(p + 1)}\bigr| = |Fv|^p = p^p. \qedhere
\]
\end{proof}

\subsubsection{Degree counting}

The lemma we are about to prove counts the $q$-rank of a $p$-adic field $F$ (the $\IF_q$-dimension of $F^\times/F^{\times q}$). It is most commonly proved via a direct sum decomposition based on the $p$-adic logarithm, which turns $F^\times$ into an additive structure, see \cite[Chap. II, (5.8)]{Neukirch99}. Here, we present an alternative proof, which has the added benefit of working in a much larger class of valued fields that includes $p$-adically closed fields.

\begin{lem}[Degree Lemma]
    \label{lem:dimOfPowerQuotient}
    Let $F/\IQ_p$ be a finite extension of degree $n$. Let $q$ be a prime such that $\zeta_q \in F$. Then
    \[
    \dim_{\IF_q}(F^\times/F^{\times q}) = 
    \begin{cases}
        2 & \text{if } q \ne p \\
        n+2 & \text{if } q = p.
    \end{cases}
    \]
\end{lem}
\begin{proof}
We first consider the case $q \ne p$.
Let $f \coloneqq f(F/\IQ_p)$ be the degree of inertia.
Note that any element in $U^{(1)} \coloneqq U^{(1)}_F$ has a $q$-th root by Hensel's Lemma. Using the direct sum decomposition for $F^\times$ (Fact~\ref{fact:multgrp_direct_prod}), we obtain
\[
    F^\times/F^{\times q} \cong \IZ/q\IZ \times \mu'_F/(\mu'_F)^q.
\]
From $\zeta_q \in \mu'_F = \mu_{p^f - 1}$ (Lemma \ref{lem:cycExt}), we conclude that $q \mid (p^f - 1)$, so the second factor is isomorphic to $\IZ/q\IZ$ and the claim follows.

Now consider the case $q = p$. Then $\IQ_p(\zeta_p) \subseteq F$, and let $e \coloneqq e(F/\IQ_p(\zeta_p))$ be the ramification index relative to $\IQ_p(\zeta_p)$. Hence,
\begin{equation} \label{eq:prod_id}
   n = e(F/\IQ_p) f(F/\IQ_p) = e(p - 1)f. 
\end{equation}
Reasoning as in the first case, we obtain
\[
    F^\times/F^{\times p} \cong \IZ/p\IZ \times U^{(1)}/\bigl(U^{(1)}\bigr)^{p},
\]
so it suffices to show that $U^{(1)}/\bigl(U^{(1)}\bigr)^{p}$ has $\IF_p$-dimension $n + 1$. To calculate the dimension, we again work modulo $U^{(ep + 1)}$. First of all, write
\begin{equation*}
U^{(1)}/\bigl(U^{(1)}\bigr)^{p} \cong \frac{U^{(1)}/U^{(ep + 1)}}{\displaystyle (U^{(1)})^p/U^{(ep + 1)}}
\end{equation*}
using Lemma \ref{lem:U_Kpowers_general}. By Corollary~\ref{cor:representation} and \eqref{eq:prod_id}, we have
\begin{equation*}
    \bigl|U^{(1)}/U^{(ep + 1)}\bigr| = |Fv|^{ep} = p^{fep} = p^{n + \frac{n}{p - 1}}.
\end{equation*}
Hence, it remains to prove
\begin{equation} \label{eq:desideratum}
    \bigl|(U^{(1)})^p/U^{(ep + 1)}\bigr| = p^{\frac{n}{p - 1} - 1}.
\end{equation}

\textit{Claim. The $p$-th power map $x \longmapsto x^p$ induces a well-defined epimorphism
\[
    \Phi : U^{(1)}/U^{(e + 1)} \epi (U^{(1)})^p/U^{(ep + 1)},
\]
with image of cardinality $p^{n/(p - 1) - 1}$.}

\begin{claimproof*}{}
Let $\pi$ be a uniformiser in $F$. Using $vp = v(\pi^{e(p - 1)})$, observe that
\[
    (1 + b\pi^{e + 1})^p \equiv 1 \pmod{\pi^{ep + 1}} \quad \text{for all $b \in \Oo_F$,}
\]
which implies $(U^{(e + 1)})^p \subseteq U^{(ep + 1)}$. Thus, $\Phi$ is well-defined and surjective. It suffices to show that $\Phi$ has kernel of order $p$. To this end, let $1 + a_1 \pi + \ldots + a_e \pi^e$ represent an element in $U^{(1)}/U^{(e + 1)}$, where each $a_i \in \Oo_F$ comes from a set of representatives for $Fv$ (Corollary~\ref{cor:representation}). Without loss of generality, assume that $a_i = 0$ represents $0 \in Fv$. Consider
\[
    (1 + a_1 \pi + \ldots + a_e \pi^e)^p \equiv 1 \quad \bigl(\mathrm{mod}\ U^{(ep + 1)}\bigr),
\]
or equivalently,
\begin{equation} \label{eq:diff_expression}
    (1 + a_1 \pi + \ldots + a_e \pi^e)^p - 1 \equiv 0 \pmod{\pi^{ep + 1}}.
\end{equation}
We claim that $a_i = 0$ for all $i \le e - 1$. If not, let $i$ be minimal such that $a_i \ne 0$. Expanding the expression in \eqref{eq:diff_expression}, we see that if $i < e$, then
\[
    v(a_i^p \pi^{pi}) = pi < e(p - 1) + i = v(pa_i\pi^i),
\]
so it follows that $a_i^p \pi^{pi}$ is the unique term of minimal value, a contradiction.
Therefore, \eqref{eq:diff_expression} simplifies to
\[
    (1 + a_e \pi^e)^p - 1 \equiv a_e^p \pi^{ep} + p a_e \pi^e \equiv 0 \pmod{\pi^{ep + 1}},
\]
which is equivalent to $\overline{a_e} \in Fv$ being a solution of
\begin{equation} \label{eq:sol}
    X^p + \overline{\frac{p}{\pi^{e(p - 1)}}} X = 0.
\end{equation}
Using Lemma~\ref{lem:p/(1-zeta)^p}, one easily checks that
\[
    \{0\} \cup \left\{\overline{\frac{\zeta_p - 1}{\pi^e}\zeta_{p - 1}^i} : i = 0, \ldots, p - 2\right\}
\]
is the solution set for (\ref{eq:sol}). This proves that $\Phi$ is $p$-to-1, implying \eqref{eq:desideratum}.
\end{claimproof*}
\end{proof}

\begin{rem} \label{rem:general_dim_formula}
As the proof of the Main Theorem will require a somewhat stronger version of the Degree Lemma for $q = p$, we examine which properties of~$\IQ_p$ were used in the arguments leading up to Lemma~\ref{lem:dimOfPowerQuotient}. We find that the following assumptions are needed: $(K,v)$ is a henselian valued field of characteristic 0 satisfying
\begin{enumerate}[(i)]
    \item $\zeta_p \in K$;
    \item the value group has minimal positive element $1$ and $v(p) = e(p - 1) < \infty$;
    \item the residue field is $\IF_{p^f}$ with $f < \infty$.
\end{enumerate}
If $(K,v)$ satisfies these conditions, then the following general degree formula holds:
\[
    \dim_{\IF_p}(K^\times/K^{\times p}) = e(p - 1)f + 2
\]
Let us briefly note that an alternative proof of this generalisation could proceed by reduction to the $p$-adic case: in the Standard Decomposition of $v_K$, the (0,\:\!0)-place does not contribute to the dimension of $K^\times/K^{\times p}$, and both the core field $Kv_0$ and its completion (which is a $p$-adic field) have the same $p$-rank by Hensel's Lemma.
\end{rem}

\begin{cor}     \label{cor:finitelyManyExtensionsQ_p}
$\IQ_p$ has only finitely many extensions of degree $n$ for any positive integer $n$. 
\end{cor}
\begin{proof}
Assume $\IQ_p$ has infinitely many extensions of degree $n$. Passing to normal hulls, this
gives infinitely many Galois extensions of $\IQ_p$ of degree $m$ for some $m \leq n!$. If $L$ is the finite extension of $\IQ_p$ obtained by adjoining $\zeta_q$ for each prime divisor $q$ of $m$, then it has infinitely many Galois extensions of degree $k$ for some $k \mid m$. By the Degree Lemma~\ref{lem:dimOfPowerQuotient} and Kummer theory, any finite extension of $L$ admits only finitely many $C_q$-extensions. As any Galois extension of $L$ of degree $k$ is obtained by iterated $C_q$-extensions (any such extension is solvable by the ramification filtration), there can only be finitely many of them: contradiction.
\end{proof}

We say that $G_{\IQ_p}$ is \emph{small}, i.e., has finitely many open subgroups of index $n$, for any fixed $n$.

\begin{rem}
The alternative (standard) approach is to use continuity of roots (Fact~\ref{fact:cont_roots}), Krasner's Lemma~\ref{lem:krasner}, and the compactness of the space $\IZ_p^n$ of integral polynomials of degree $n$ to deduce that there are finitely many totally ramified extensions of degree $n$. The general statement then follows by considering inertia subfields.
\end{rem}

We will later need the following converse statement about fields that have small absolute Galois group.

\begin{lem} \label{lem:small_Gal->finite_dim}
Let $K$ be a field and $q$ be any prime. If $K$ has small absolute Galois group and $\charK K \ne q$, then $\dim_{\IF_q} (K^{\times}/K^{\times q}) < \infty$.
\end{lem}
\begin{proof}
Note that $L \coloneqq K(\zeta_q)$ is a Galois extension of $K$ of degree $d \mid (q - 1)$. The norm map induces an $\IF_q$-vector space homomorphism
\[
    \overline{N_{L/K}} : L^\times/L^{\times q} \longrightarrow K^\times/K^{\times q}
\]
that is surjective, since for any $x \in K^\times$, we have
\[
    \langle N_{L/K}(x) \rangle = \langle x^d \rangle \equiv \langle x \rangle \pmod{K^{\times q}}.
\]
Hence,
\[
    \dim_{\IF_q} (K^\times/K^{\times q}) \le \dim_{\IF_q} (L^\times/L^{\times q}) < \infty,
\]
by Kummer theory, using the fact that $L$ admits only finitely many $C_q$-extensions (since $G_K$ is small).
\end{proof}

\subsection{Fields with finite \texorpdfstring{$p$}{p}-rank}
\label{sec:Pop}

The finiteness of $K^\times/K^{\times p}$ is a natural condition in the context of Kummer theory, and, as we have shown above, holds for $\IQ_p$. In \cite[Kor. 2.7]{Pop88}, Pop discovered that this simple condition implies strong structural properties whenever $K$ comes with a valuation. See also \cite[12.5]{Efrat06}.

\begin{lem}[Pop's Lemma] \label{lem:Pop}
Let $(K,v)$ be a valued field of mixed characteristic $(0,p)$. Assume that $K^\times/K^{\times p}$ is finite. Then:
\begin{enumerate}[(i)]
    \item $Kv$ is perfect.
    \item If $v$ is of rank 1, then $vK$ is either discrete or $p$-divisible.
    \item If $vK$ is discrete, then $Kv$ is finite.
\end{enumerate}
\end{lem}
\begin{proof}
We have two identities,
\begin{align*}
    \Oo_v^{\times p} & = \Oo_v^\times \cap K^{\times p} \\
    (1 + \Mm_v)^p & = \Oo_v^{\times p} \cap (1 + \Mm_v),
\end{align*}
where the first one is immediate; for the second one, note that for any $x \in \Oo_v^\times$,
\begin{equation} \label{eq:x^p=1 mod M}
    x^p \equiv 1 \pmod{\Mm_v} \Longrightarrow x \equiv 1 \pmod{\Mm_v}.
\end{equation}
Consequently, the two standard short exact sequences of \eqref{eq:SES} induce two additional short exact sequences by the nine lemma:
\begin{gather*}
    1 \longrightarrow \Oo_v^\times/\Oo_v^{\times p} \longrightarrow K^\times/K^{\times p} \longrightarrow vK/pvK \longrightarrow 1 \\
    1 \longrightarrow (1 + \Mm_v)/(1 + \Mm_v)^p \longrightarrow \Oo_v^\times/\Oo_v^{\times p} \longrightarrow (Kv)^\times/(Kv)^{\times p} \longrightarrow 1
\end{gather*}
The exactness of these sequences implies that all quantities in
\[
    |(Kv)^\times/(Kv)^{\times p}| \cdot |(1 + \Mm_v)/(1 + \Mm_v)^p| = |\Oo_v^\times/\Oo_v^{\times p}| \le |K^\times/K^{\times p}|
\]
are finite.

To prove (i), it suffices to consider the case that $Kv$ is infinite. In that case, $(Kv)^p$ must be infinite as well.
Now fix any $x \in (Kv)^\times$. Observe that $\{x + a^p : a \in (Kv)^{\times}\}$ is an infinite set, while $K^\times/K^{\times p}$ is finite. By the pigeonhole principle, there are $a, b, c, d \in (Kv)^{\times}$, $a \ne c$, $b \ne d$, such that $(x + a^p)b^p = (x + c^p)d^p$. Solving for $x$, we obtain
\[
    x = \frac{(cd - ab)^p}{(b - d)^p} \in (Kv)^{\times p}.
\]
Hence, $Kv$ is perfect.

For (ii), recall that we know $(1 + \Mm_v)/(1 + \Mm_v)^p$ is finite. We let $\{1 + a_i\}_{1 \le i \le n} \subseteq 1 + \Mm_v$ be a complete set of representatives for $(1 + \Mm_v)/(1 + \Mm_v)^p$ and define
\[
    \delta \coloneqq \min\{va_1, \ldots, va_n, vp\} > 0.
\]
Consider an arbitrary element $x \in \Mm_v$. For some $i \in \{1, \ldots, n\}$ and $y \in \Mm_v$, we may write
\[
    1 + x = (1 + a_i)(1 + y)^p
\]
and therefore
\begin{equation} \label{eq:pop_id}
    vx = v\bigl(a_i(1 + y)^p + [(1 + y)^p - 1 - y^p] + y^p\bigr) \ge \min\{va_i,vp,v(y^p)\}.
\end{equation}
If $vx < \delta$, then \eqref{eq:pop_id} implies $v(y^p) < \delta$. In particular, we attain equality $vx = v(y^p) \in pvK$ in this case. Our consideration now splits into two cases: if $\delta$ is minimal positive in $vK$, then $vK = \delta \IZ$ is discrete. Otherwise, the interval $(0,\delta) \subset vK$ is $p$-divisible and non-empty. Since $vK$ is of rank 1, the interval generates $vK$, and hence $vK$ must be $p$-divisible itself.

Finally, for (iii), we assume $K$ has a uniformiser. It is easy to see that $(1 + \Mm_v)^p \subseteq 1 + \Mm_v^2$. In particular,
\[
    |(1 + \Mm_v)/(1 + \Mm_v^2)| \le |(1 + \Mm_v)/(1 + \Mm_v)^p| < \infty.
\]
Adapting Fact~\ref{fact:TowerLawU}, we see that the left-hand side equals $|Kv|$, so $Kv$ must be finite.
\end{proof}

\subsection{Distinguishing characteristics of fields}
\label{sec:cohomology disguise}

In this section, we answer the question: \emph{how does $G_{\IQ_p}$ know that $\IQ_p$ has characteristic $\ne p$?}

Galois cohomology provides an immediate answer if we consider cohomological $p$-dimensions:
\[
    \cd_p G_{\IQ_p} = 2, \text{ while } \cd_p G_K \le 1
\]
for any field $K$ of characteristic $p$, see \cite[II.~\S4, II.~\S2.2]{Serre97}. Since our goal is to give a self-contained answer, we present an elementary criterion---motivated by Galois cohomology---for distinguishing fields of characteristic zero from those of positive characteristic. This will make explicit the elementary ideas concealed by the cohomological formalism.

\begin{idea} \label{idea:char}
For a field $K$, consider the property of whether every $C_p$-extension of $K$ embeds into a $C_{p^2}$-extension of $K$. We will show that this is true in general if $K$ has characteristic $p$, and false for some $p$-adic field $K$.
\end{idea}

Cyclic Galois extensions are well-understood by Hilbert 90, Kummer theory, and Artin-Schreier theory (in positive characteristic). For convenience, we state these classical results.

\begin{thm}[Hilbert 90] \label{thm:H90}
Let $L/K$ be a finite cyclic Galois extension and $\sigma$ a generator of $\Gal(L/K)$. Then for any $a \in L$:
\begin{alignat*}{4}
    N_{L/K}(a) = 1 & \Longleftrightarrow \exists b \in L^{\times}\ && a = \sigma(b) \cdot b^{-1} && \quad \text{\emph{(``Multiplicative H90'')}} \\
    \Tr_{L/K}(a) = 0 & \Longleftrightarrow \exists b \in L\ && a = \sigma(b) - b && \quad \text{\emph{(``Additive H90'')}}
\end{alignat*}
\end{thm}

One can deduce the following characterisations of $C_p$-extensions:

\begin{thm}[Kummer theory]
\label{thm:Kummer theory}
Let $K$ be a field with $\charK K \ne p$ and $\zeta_p \in K$. Then there exists a bijection
\begin{align*}
    \{\text{one-dim.\,subspaces } \langle c \rangle K^{\times p} \le K^\times/K^{\times p}\} & \longleftrightarrow \{\text{$L/K$ Galois extensions of deg.\;$p$}\} \\
    \langle c \rangle K^{\times p} & \longmapsto K(\sqrt[p]{c}) \\
    (K^\times \cap L^{\times p})/K^{\times p} & \longmapsfrom L.
\end{align*}
\end{thm}

\begin{thm}[Artin-Schreier theory]
\label{thm:Artin-Schrier theory}
Let $K$ be a field of positive characteristic $p$. Then there exists a bijection
\begin{align*}
    \{\text{one-dim.\;subspaces } \langle c \rangle  + \wp(K) \le K/\wp(K)\} & \longleftrightarrow \{\text{$L/K$ Galois extensions of deg.\;$p$}\} \\
    \langle c \rangle + \wp(K) & \longmapsto K(\alpha),
\end{align*}
where $\wp(X) = X^p - X$ is the Artin-Schreier map and $\alpha$ is a root of the Artin-Schreier polynomial $\wp(X) - c$.
\end{thm}

The next three lemmas realise Idea~\ref{idea:char}.

\begin{lem} \label{lem:C_pInC_p^2Charp}
    Let $K$ be a field of positive characteristic $p$. Then any $C_p$-extension $L/K$ is contained in a $C_{p^2}$-extension $M/K$.
\end{lem}
\begin{diagram}[cramped]
    M \ar[d, no head, dotted] \ar[dd, no head, dotted, bend left=40, "C_{p^2}"]\\
    L \ar[d, no head, "C_p"'] \\
    K
\end{diagram}
\begin{proof}
    By Artin-Schreier theory, $L$ is generated by a root $x$ of an Artin-Schreier polynomial $X^p - X - a$. Dividing by $(-x^pa)$, we obtain
    \[
        -a^{-1} + a^{-1}(x^{-1})^{p-1} + (x^{-1})^p=0,
    \]
    which implies $\Tr_{L/K}(x^{-1}) = -a^{-1}$. We also have $-x^{p - 1} = -1 - a x^{-1}$, yielding
    \begin{align}
        \label{eq:Tr(b)=1}
        \Tr(-x^{p-1}) = \Tr(-1 - a x^{-1}) = -\Tr(1) - a\Tr(x^{-1}) = -p + 1 = 1. 
    \end{align}
    Define $b \coloneqq -x^{p - 1}$. In characteristic $p$, $\Tr(b^p) = \Tr(b)^p$ holds, and thus
    \[
        \Tr(b^p - b) = \Tr(b^p) -  \Tr(b) = \Tr(b)^p -  \Tr(b) = 1 - 1=0.
    \]
    Hence, we can apply (additive) Hilbert 90 to find $c \in L$ such that
    \begin{equation} \label{eq:H90_conj}
        b^p - b = \sigma(c) - c
    \end{equation}
    for some generator $\sigma$ of $\Gal(L/K)$.
    We claim that the Artin-Schreier extension $L(y)/L$, where $y$ is a root of $X^p - X - c$, is the desired $C_{p^2}$-extension of $K$.
    Indeed, $M \coloneqq L(y)$ over $K$ has degree $p^2$ and is Galois by (\ref{eq:H90_conj}). So we are left to prove $\Gal(M/K) \not\cong C_p \times C_p$. Assume the contrary. By the Galois correspondence, there exists a $C_p$-extension $L' \ne L$ over $K$ contained in $M$. By Artin-Schreier theory, $L' = K(z)$ where $z^p - z - d = 0$ for some $d \in K$. By construction, we have $L(z) = M$. Again, Artin-Schreier theory tells us that $c$ and $d$ generate the same $\IF_p$-subspace in $L/\wp(L)$, so $c = id + e^p - e$ for some $i \in \{1, \ldots, p - 1\}$ and $e \in L$. Therefore,
    \[
        b^p-b = \sigma(c) -c = (\sigma(id) + \sigma(e)^p - \sigma(e)) - (id + e^p - e) = (\sigma(e) - e)^p - (\sigma(e)-e).
    \]  
    In particular, $\sigma(e) - e = b + j$ for some $j \in \{0, \ldots, p - 1\}$. However,
    \[
        \Tr(\sigma(e) - e) = 0 \ne 1 = 1 + pj = \Tr(b) + \Tr(j) = \Tr(b + j),
    \]
    by \eqref{eq:Tr(b)=1}, which yields a contradiction.
\end{proof}

As indicated by Idea \ref{idea:char}, there is a $p$-adic field that admits a $C_p$-extension that does not embed into any $C_{p^2}$-extension. We will give an explicit example in Lemma \ref{lem:C_pNotInC_p^2}. To facilitate the construction, we will first give an equivalent criterion in the presence of a primitive $p$-th root of unity.
\begin{lem}
    \label{lem:C_pCriterion}
    Assume $K$ is a field with $\charK K \ne p$ and $\zeta_p \in K$. Let $L/K$ be a $C_p$-extension. Then $L$ can be embedded into a $C_{p^2}$-extension $M/K$ if and only if $\zeta_p \in \im N_{L/K}$.
\end{lem}
\begin{proof}
    $\Rightarrow$. Assume that $L$ can be embedded into a $C_{p^2}$-extension $M/K$. Then $M/L$ is Galois with $\Gal(M/L) \cong C_p$. Therefore, $M = L(\sqrt[p]{d})$ for some $d \in L^\times \setminus L^{\times p}$ by Kummer theory. Let $\sigma$ be a generator of $\Gal(L/K)$. As $L(\sqrt[p]{d}) = L(\sqrt[p]{\sigma(d)})$, Kummer theory implies $\sigma(d) \in \langle d \rangle L^{\times p}$. Replacing $\sigma$ by an appropriate power if necessary, we may assume $\sigma(d) \in d L^{\times p}$ without loss of generality.
    Let $e \in L^\times$ be such that $\sigma(d)d^{-1} = e^p$. Then
    \[
        (N_{L/K}(e))^p = N_{L/K}(e^p) = N_{L/K}\left(\frac{\sigma(d)}{d}\right) = \frac{\sigma(d)}{d}\frac{\sigma(\sigma(d))}{\sigma(d)} \cdots \frac{\sigma^p(d)}{\sigma^{p-1}(d)} = 1.
    \]
    It suffices to show that $N_{L/K}(e) \ne 1$. If instead $N_{L/K}(e) = 1$, then we could find $f \in L^\times$ such that $e = \sigma(f)f^{-1}$ by Hilbert 90. Thus,
    \[
        \sigma(d)d^{-1} = e^p = \left(\sigma(f)f^{-1}\right)^p \Longrightarrow \sigma\left(df^{-p}\right) = df^{-p},
    \]
    which implies $df^{-p} \in K$. But then $K(\sqrt[p]{df^{-p}}) = K(\sqrt[p]{d}/f) \ne L$ would be yet another $C_p$-extension of $K$ contained in $M$, which is impossible by the Galois correspondence for $M/K$.
    
    $\Leftarrow$. The line of reasoning is very similar, but with inverted order. Assume that $N_{L/K}(e) = \zeta_p$. Then $N_{L/K}(e^p) = (N_{L/K}(e))^p = 1$, so there exists some $d \in L^\times$ such that $\sigma(d)d^{-1} = e^p$. We claim that $d \notin L^{\times p}$. For otherwise, we could write $d = f^p$, so that
    \[
        e^p = \sigma(d)d^{-1} = \left(\sigma(f)f^{-1}\right)^p \Longrightarrow e = \sigma(f)f^{-1} \zeta_p^i \phantom{a}
    \]
    for some $i \in \{0, \ldots, p - 1\}$, which implies
    \[
        \zeta_p = N_{L/K}(e) =  N_{L/K}\left(\sigma(f)f^{-1}\right)N_{L/K}(\zeta_p^i) = 1,
    \]
    a contradiction.
    Thus, $M \coloneqq L(\sqrt[p]{d})$ is indeed a $C_p$-extension of $L$. Moreover, using $\sigma(d) = de^p$, we see that $L(\sqrt[p]{d}) = L(\sqrt[p]{\sigma(d)})$, so $M/K$ is a Galois extension of degree $p^2$ (as $M$ contains all conjugates of $\sqrt[p]{d}$). Note that $\Gal(M/K)$ is either isomorphic to $C_{p^2}$ or $C_p \times C_p$. We are left to exclude the latter case. So assume for a contradiction that $\Gal(M/K) \cong C_p \times C_p$. On account of the Galois correspondence, there exists a $C_p$-extension $L' \ne L$ of $K$ inside $M$. By Kummer theory, $L' = K(\sqrt[p]{c})$ for some $c \in K$ and $L(\sqrt[p]{c}) = M$. Yet again, this means that $c$ and $d$ generate the same subgroup in $L^\times/L^{\times p}$. Thus, $d = c^i f^p$ with $i \in \{1, \ldots, p - 1\}$ and $f \in L^\times$. But then
    \[
        e^p = \sigma(d)d^{-1} = \frac{\sigma(c^i f^p)}{c^i f^p} = \left(\sigma(f)f^{-1}\right)^p,
    \]
    contradicting our earlier observation that this equality cannot hold.
\end{proof}

\begin{lem} \label{lem:C_pNotInC_p^2}
Let $p \ne 2$ be a prime. Consider $K = \IQ_p(\zeta_p)$ with uniformiser $\pi = \zeta_p - 1$. Then, $L = K(\sqrt[p]{(1 - p)\pi})$ is a $C_p$-extension of $K$ that is not contained in any $C_{p^2}$-extension. \\
Moreover, $L = \IQ_2(\sqrt{-2})$ is a quadratic extension of $\IQ_2$ not contained in any $C_4$-extension.
\end{lem}
\begin{proof}
We write $c \coloneqq 1 - p \in \Oo_K^\times$ and let $\Pi \coloneqq \sqrt[p]{c\pi}$ be a root of the Eisenstein polynomial $X^p - c\pi$. A fortiori, this is a uniformiser of $L$. By Lemma~\ref{lem:C_pCriterion}, we only need to verify that $\zeta_p \notin \im N_{L/K}$. Since $K^{\times p} \subseteq \im N_{L/K}$, it suffices to show that the induced map
\[
    \overline{N_{L/K}} : L^\times/L^{\times p} \longrightarrow K^\times/K^{\times p}
\]
does not contain $\zeta_pK^{\times p}$ in its image. We will determine the image by direct calculation.

From Lemma~\ref{lem:generators}, we know that
\[
    R = \bigl\{\Pi, 1 + \Pi, 1 + \Pi^2, \ldots, 1 + \Pi^{p^2}\bigr\}
\]
is a set of representatives generating $L^{\times}/L^{\times p}$.
Consider an element $x_k = 1 + \Pi^k$ for $p \nmid k$.
Note that $x_k$ has minimal polynomial $(X - 1)^p - (c\pi)^k$, and thus
\[
    N_{L/K}(x_k) = 1 + (c\pi)^k.
\]
For $k \ge p + 1$, the norms $N_{L/K}(x_k) = 1 + (c\pi)^k$ already lie in $K^{\times p}$ by Lemma~\ref{lem:U_Kpowers_general}. Therefore, $\im \overline{N_{L/K}}$ is generated by the set of representatives
\[
    S = \{c\pi, 1 + c\pi, \ldots, 1 + (c\pi)^{p - 1}\}.
\]
Hence, $\im \overline{N_{L/K}} \subseteq K^\times/K^{\times p}$ is a subgroup of index $p$ (by the second part of Lemma~\ref{lem:generators}).

Assume towards a contradiction that $\zeta_p = 1 + \pi \in \im N_{L/K}$. By Lemma~\ref{lem:p/(1-zeta)^p},
\[
    1 + c\pi = 1 + \pi -p\pi \equiv 1 + \pi + \pi^p \equiv (1 + \pi) (1 + c^p\pi^p) \pmod{\pi^{p + 1}}
\]
and hence, by Remark~\ref{rem:mod_rule},
\[
    \frac{1 + c\pi}{1 + \pi} \equiv 1 + (c\pi)^p \pmod{U_K^{(p + 1)}}.
\]
Note that, by Lemma~\ref{lem:U_Kpowers_general}, $U_K^{(p + 1)} \subseteq K^{\times p}$. But this implies that $\overline{1 + (c\pi)^p}$ lies in the image of $\overline{N_{L/K}}$. This is impossible, since $S \cup \{1 + (c\pi)^p\}$ represents a basis of $K^{\times}/K^{\times p}$.

For $K = \IQ_2$ and $L = \IQ_2(\sqrt{-2})$, we have $\im \overline{N_{L/K}} = \langle \overline{2}, \overline{3} \rangle$. Therefore, $-1 \notin \im N_{L/K}$, for otherwise $\IQ_2^\times/\IQ_2^{\times 2} = \langle \overline{2}, \overline{3}, \overline{5} \rangle = \langle \overline{2}, \overline{3}, \overline{-1} \rangle \subseteq \overline{N_{L/K}}$.
\end{proof}

\subsection{Galois characterisation of henselianity} \label{sec:Galois_char_henselianity}

The main result of this section provides a method for recovering certain henselian valuations---those we shall call \emph{tamely branching}---from absolute Galois groups. We begin by stating the result. After giving some historical context, we motivate the definition of tamely branching valuations and then proceed to the proof.

\begin{defi} \label{def:tamely_branch}
    Let $(K,v)$ be a valued field and $p$ any prime. We say $v$ is \emph{tamely branching at $p$} if
    \begin{enumerate}[(i)]
        \item $\charK Kv \ne p$; \label{item:tam_branch1}
        \item $vK \ne p vK$; \label{item:tam_branch2}
        \item $p^2 \mid |G_{Kv}|$, whenever $(vK : p vK) = p$. \label{item:tam_branch3}
    \end{enumerate}
\end{defi}

\begin{thm}[Galois characterisation of henselianity] \label{thm:Galois code henselianity}
A field $K$ admits a henselian valuation, tamely branching at $p$, if and only if $G_K$ has a Sylow $p$-subgroup $P \not\cong C_2, \IZ_p, \IZ_2 \rtimes C_2$ that contains a non-trivial normal abelian closed subgroup $N$.
\end{thm}
Recovering valuations lies at the very heart of extracting arithmetic information from Galois-theoretic data. Both Neukirch’s Theorem \ref{thm:Neukirch} and Pop’s Theorem \ref{thm:Pop} crucially rely on this principle. For example, Neukirch’s Theorem builds on the following result \cite[Satz 1]{Neukirch68}, which can be viewed as an early precursor to the full Galois-theoretic characterisation of henselian valued fields.

\begin{thm}[Neukirch 1968]
Let $K \subseteq \IQ^{\alg}$ be a field such that
\begin{enumerate}[(i)]
    \item $G_K$ is pro-solvable and
    \item the Brauer group $\operatorname{Br}(K)$ is not a $p$-group for any prime $p$ and not a $\{2,3\}$-group (i.e., a group where each element has order $2^n 3^m$).
\end{enumerate}
Then $K$ admits a henselian valuation.
\end{thm}

Pop extended Neukirch's result \cite[Satz 1.13]{Pop88}, keeping the assumption of pro-solvability of $G_K$, but replacing condition (ii) with a more general one. Incidentally, in \cite{Koenigsmann95}, this generalised result was used to show that fields $K$ with $G_K \cong G_{\IQ_p}$ admit a henselian valuation, leading to a complete proof of our Main Theorem. Nevertheless, the full Galois characterisation of henselianity, as stated above, had not yet been established.

Further progress was made in \cite{Engler98}, where the authors proved a $p$-henselian analogue of Theorem \ref{thm:Galois code henselianity}, extending earlier results in \cite{Engler94, Efrat952} for the case $p = 2$. A criterion for henselian valuations was later formulated in \cite{Koenigsmann01}, although it required the additional assumption that the Sylow $p$-subgroups of $G_K$ are solvable. This restriction was subsequently removed in \cite{Koenigsmann03} via the``henseling down'' techniques discussed in Section \ref{sec:henseling down}. The result and its proof were also presented in Engler and Prestel's exposition \cite[Sec. 5.4]{Engler05}, which our treatment here generally follows.
We adapt some of their arguments for clarity.

Before turning to the proof of Theorem~\ref{thm:Galois code henselianity}, we will motivate Definition~\ref{def:tamely_branch} by indicating why we can only expect to recover tamely branching henselian valuations from the absolute Galois group.
\begin{rem}
\begin{enumerate}[(a)]
    \item Let $k$ be a field of characteristic 0 that does not admit a non-trivial henselian valuation, e.g., $k = \IQ$. Further, let $(K,v)$ be a henselian valued field with divisible value group $vK$ and residue field $Kv = k$, e.g., the field of Puiseux series over $k$ (see Example~\ref{ex:puiseux}).
    Then $G_K \cong G_k$ by Facts~\ref{fact:inertia group} and \ref{fact:ramification group}. By our assumption on $k$, we cannot expect to recover $v$ from the isomorphism class of $G_K$. Therefore, we require $vK$ to be non-$p$-divisible for at least one prime $p$, motivating \ref{item:tam_branch2} in Definition~\ref{def:tamely_branch}.

    \item Consider $K = \IC\laurent{t}$ with the $t$-adic valuation. We saw in Example~\ref{ex:Galois calculations} that $G_K \cong G_{\IF_q}$ for any prime $q$. Therefore, we cannot expect to recover the $t$-adic valuation from the isomorphism class of $G_K$. In general, if $\charK Kv \ne p$, $(vK : pvK) = p$, and $p \nmid G_{Kv}$, then $G_K$ has Sylow $p$-subgroup $\IZ_p$. However, this is the absolute Galois group of $\bigcup_{p \nmid n} \IF_{q^n}$, which admits only the trivial valuation, hence motivating \ref{item:tam_branch3}.
    
    \item If $p > 2$, then $p \mid G_{Kv}$ implies $p^2 \mid G_{Kv}$ by the Artin-Schreier Theorem~\ref{thm:Artin-Schreier}. There is yet another issue if $p = 2$ and $2 \mid G_{Kv}$, but $4 \nmid G_{Kv}$.
    The field $K = \IR\laurent{t}$, endowed with the $t$-adic valuation $v$, is henselian with $\charK Kv \ne 2$, $(vK : 2vK) = 2$, and $G_{Kv} \cong C_2$.
    By Theorem~\ref{thm:Splitting of D/R}, its absolute Galois group is $\widehat{\IZ} \rtimes C_2$ (which has Sylow 2-subgroup $\IZ_2 \rtimes C_2$). However, it is possible to construct a field $L$ that does not admit a non-trivial henselian valuation with $G_L \cong \IZ_2 \rtimes C_2$ as follows.
    Let $L$ be the intersection of two distinct real closures of $\IQ$ in $\IQ^{\alg}$, and let $\sigma$ and $\tau$ be the involutions in $G_{\IQ}$ corresponding to the two real closures; together they topologically generate $G_L$. By Artin-Schreier, they do not commute and $\sigma\tau$ generates an infinite cyclic group $Z$, where $\sigma$ acts on $Z$ by conjugation via $\sigma\tau \mapsto \tau\sigma = (\sigma\tau)^{-1}$. Hence, $G_L = \langle\sigma, \tau\rangle$ is isomorphic to $\overline{Z} \rtimes \langle \sigma \rangle$. If we choose the two real closures such that $\sigma$ and $\tau$ lie in the same Sylow $2$-subgroup of $G_{\IQ}$, then $G_L$ will be of the form $\IZ_2 \rtimes C_2$.
\end{enumerate}
\end{rem}

Furthermore, note that condition \ref{item:tam_branch1} on the characteristic of the residue field allows for a natural choice of the group $N$ in Theorem~\ref{thm:Galois code henselianity} (see below).

For the remainder of this section, we use the following notation. Given a field $K$ and a closed subgroup $H \leq G_K$, we write $\Fix H$ for the (absolute) fixed field of $H$ inside $K^{\mathrm{sep}}$.

\begin{lem} \label{lem:Galois_hens-->}
    Let $K$ be a field with a henselian valuation $v$, tamely branching at $p$, and $w$ its unique extension to $L \coloneqq \Fix P$, where $P$ is Sylow $p$-subgroup of $G_K$. Then, the absolute inertia group $I_L$  is a non-trivial normal abelian closed subgroup of $P$. Moreover, $P \not\cong  C_2, \IZ_p, \IZ_2 \rtimes C_2$.
\end{lem}
\begin{proof}
    By assumption and  Fact~\ref{fact:inertia group}, $D_L = G_L = P$ and $I_L$ is a normal closed subgroup of $P$. Since $\charK Kv \ne p$, the absolute ramification subgroup $R_L$ is trivial by Fact~\ref{fact:ramification group}. Using Theorem~\ref{thm:Splitting of D/R}, we compute
    \begin{align}
        I_L & \cong \IZ_p^{\dim_{\IF_p}w L/p w L} \label{eq:I_L} \\
        P = D_L & \cong I_L \rtimes G_{Lw}. \label{eq:P}
    \end{align}
    As $P$ is a Sylow $p$-subgroup of $G_K$, the extension $L/K$ is prime-to-$p$, and hence $p \nmid [Lw : Kv]$ and $(wL : p wL) = (vK : p vK)$. Therefore, $w$ is tamely branching at $p$ as well, so $I_L$ is indeed non-trivial abelian by \eqref{eq:I_L} and $P \not\cong  C_2, \IZ_p, \IZ_2 \rtimes C_2$ by \eqref{eq:P}.
\end{proof}

This lemma proves the ``only if'' part in Theorem~\ref{thm:Galois code henselianity}; we will now focus on the other direction. We prove it in two steps: First, we establish the theorem in the ``classical'' case where $G_K \cong \IZ_p \rtimes \IZ_p$ (the name stems from the fact that any Sylow $p$-subgroup of $G_{\IQ_q}$ is of the shape $\IZ_p \rtimes \IZ_p$ for $p \ne q$, see Proposition~\ref{prop:Syl G_Q_p}). In a second step, we reduce the general case to the ``classical'' one.

\begin{prop}[Classical case]
    \label{prop:classical case}
    Let $K$ be a field with $G_K \cong \IZ_p \rtimes \IZ_p$. Assume additionally that $\sqrt{-1} \in K$ if $p = 2$. Then $K$ admits a henselian valuation tamely branching at $p$.
\end{prop}

Some of the ideas in the proof of the classical case are inspired by cohomological considerations; as such, they will look similar to those in Section~\ref{sec:cohomology disguise}. Otherwise, our main technical tool is the notion of rigid elements, as introduced in Section~\ref{sec:rigid elements} (Theorem~\ref{thm:creation_p-rigid} will yield the desired valuation on $K$).

Before we establish the classical case, we prove the following basic lemma, which allows us to pass to finite extensions of $K$ without changing the shape of the Galois group.
\begin{lem}
    \label{lem:finite extensions classical case}
    Let $K$ be a field with $G_K \cong \IZ_p \rtimes \IZ_p$. Then for every finite extension $L/K$, also $G_L \cong \IZ_p \rtimes \IZ_p$ (with possibly different action in the semi-direct product).
\end{lem}
\begin{proof}
    Let
    \[
    \begin{tikzcd}
        1 \arrow[r] & \IZ_p \arrow[r, "\iota"] & G_K \arrow[r, "\pi"] & \IZ_p \arrow[r] & 1
    \end{tikzcd}
    \]
    be the short exact sequence associated to the semidirect product for $G_K$. As $L/K$ is a finite extension, $G_L$ is a closed subgroup of finite index in $G_K$. Hence, $\iota^{-1}(G_L)$ and $\pi(G_L)$ will also be closed of finite index in their respective copies of $\IZ_p$. Therefore, $\iota^{-1}(G_L)$ and $\pi(G_L)$ are themselves isomorphic to $\IZ_p$. The associated short sequence with maps $\iota|_{\iota^{-1}(G_L)}$ and $\pi|_{G_L}$ remains exact and split, so the claim follows.
\end{proof}

Recall that an element $x \in K^\times \setminus K^{\times p}$ is called $p$-rigid with respect to $K^{\times p}$ if
\begin{equation}
    \label{eq:p-rigidity}
    K^{\times p} + xK^{\times p} \subseteq \bigcup_{i = 0}^{p - 1} x^iK^{\times p}.
\end{equation}
If all such $x$ are $p$-rigid and $(K^\times : K^{\times p}) > p$, then Theorem~\ref{thm:creation_p-rigid} allows us to create a valuation for the choice of parameters $S = T_S = K^{\times p}$. The key idea of the proof is that the right-hand side of \eqref{eq:p-rigidity} is given by the image of the norm map for $K(\sqrt[p]{x})/K$.

We organise the proof into six steps.

\begin{enumerate}[wide, labelwidth=!, labelindent=1.5em]
    \item[\normalfont\ref{hensel_step:1}] Show $\charK K \ne p$ and $\zeta_p \in K$. Deduce $\dim_{\IF_p}K^\times/K^{\times p} = 2$.
    \item[\normalfont\ref{hensel_step:2}] Show that every $x \in K \setminus K^{\times p}$ is $p$-rigid via Kummer theory.
    \item[\normalfont\ref{hensel_step:3}] Create a valuation $v$ satisfying $vK \ne pvK$ using Theorem~\ref{thm:creation_p-rigid}.
    \item[\normalfont\ref{hensel_step:4}] Show that $\charK Kv \ne p$.
    \item[\normalfont\ref{hensel_step:5}] Show that $v$ is henselian.
    \item[\normalfont\ref{hensel_step:6}] Show that $v$ is tamely branching at $p$.
\end{enumerate}

\begin{proof}[{Proof of Proposition~\ref{prop:classical case}}]
We follow the six steps outlined above.

\begin{enumerate}[wide, labelwidth=!, align=left, labelindent=0pt, label=\textbf{{Step} \arabic*.}]
    \item \label{hensel_step:1} Note that $K$ admits precisely two linearly disjoint $C_p$-extensions because $\IZ_p \rtimes \IZ_p$ has rank two.
    Suppose for now that $\charK K = p$. Then, by Artin-Schreier theory (Theorem~\ref{thm:Artin-Schrier theory}),
    \[
        \dim_{\IF_p} K/\wp(K) = 2.
    \]
    Let $a\wp(K)$ and $b\wp(K)$ be a basis for $K/\wp(K)$. The splitting fields of $X^p - X - a$ and $X^p - X - b$ are then linearly disjoint. If $x$ is a root of $X^p - X - a$, then $\alpha \coloneqq -x^{p-1}$ satisfies $\Tr(\alpha)=1$ (as in \eqref{eq:Tr(b)=1} in the proof of Lemma~\ref{lem:C_pInC_p^2Charp}), where $\Tr = \Tr_{K(x)/K}$.
    
    We claim that the elements $a\alpha$, $b\alpha$, and $b$ are $\IF_p$-linearly independent in $K(x)/\wp(K(x))$. Indeed, let $i_1, i_2, i_3 \in \IF_p$ be such that
    \[
        i_1 a \alpha + i_2 b \alpha + i_3b = y^p - y
    \]
    for some $y \in K(x)$. Taking the trace on both sides of the equation, we obtain
    \[
        i_1 a + i_2 b + p(i_3 b) = i_1 a + i_2 b = \Tr(y^p) - \Tr(y) = \Tr(y)^p - \Tr(y),
    \]
    which implies $i_1 = i_2 = 0$ by our choice of $a$ and $b$. Hence, $y^p - y - i_3b = 0$. By our choice of $b$, this implies $i_3 = 0$, and therefore, $\dim_{\IF_p} K(x)/\wp(K(x)) \ge 3$. Recall that Lemma~\ref{lem:finite extensions classical case} tells us that $G_{K(x)} \cong \IZ_p \rtimes \IZ_p$, which contradicts our initial observation applied to $K(x)$. Therefore, $\charK K \ne p$ and $\zeta_p \in K^{\sepc}$.
    
    Note that $[K(\zeta_p) : K] \le p - 1$ and the assumption that $G_K$ is a pro-$p$ group imply $\zeta_p \in K$. In particular, $\dim_{\IF_p} K^\times/K^{\times p} = 2$ by Kummer theory (Theorem~\ref{thm:Kummer theory}).
    
    \item \label{hensel_step:2}
    Let $x \in K^\times \setminus K^{\times p}$ and set $L \coloneqq K(\sqrt[p]{x})$. Choose two cosets $xK^{\times p}$, $yK^{\times p}$ that form an $\IF_p$-basis for $K^\times/K^{\times p}$. By Kummer theory, $\sqrt[p]{x}L^{\times p}$ and $yL^{\times p}$ are linearly independent in $L^\times/L^{\times p}$ and hence form a basis by Lemma~\ref{lem:finite extensions classical case}. Thus, we may write
    \[
        L^{\times} = \bigcup_{0 \le i, j < p} (\sqrt[p]{x})^i y^j L^{\times p}.
    \]
    For any $a, b \in K^\times$, there exist $i, j \in \{0, \ldots, p - 1\}$ and $z \in L^\times$ such that
    \[
        a + b\sqrt[p]{x} = (\sqrt[p]{x})^i y^j z^p.
    \]
    Note that $a + b\sqrt[p]{x}$ has minimal polynomial $(X - a)^p - b^px$. Applying the norm map to both sides yields
    \[
        (-1)^{p - 1} N_{L/K}(a + b\sqrt[p]{x}) = (-1)^{p - 1}a^p + xb^p = x^i y^{jp} N_{L/K}(z)^p \in x^i K^{\times p}.
    \]
    This proves the inclusion \eqref{eq:p-rigidity}: $x$ is $p$-rigid. (Note that we assumed $\sqrt{-1} \in K$ if $p = 2$.)
    
    \item \label{hensel_step:3} By Step 2, any $x \in K^\times \setminus K^{\times p}$ is $p$-rigid. Moreover, $(K^\times : K^{\times p}) = p^2 > p$. By the creation of valuations from $p$-rigid elements (Theorem~\ref{thm:creation_p-rigid}) for parameters $S = T_S = K^{\times p}$, we obtain a non-trivial valuation $v$ with $vK \ne p vK$. Additionally, we replace $v$ by the coarsening with respect to the maximal $p$-divisible convex subgroup. Then, $v$ will not contain any non-trivial $p$-divisible convex subgroups and $v$ remains non-$p$-divisible.

    \item \label{hensel_step:4} We now prove $\charK Kv \ne p$. Assume towards a contradiction that $\charK Kv = p$.
    Then, $\Conv(vp)$ is a non-trivial convex subgroup. By our choice of $v$, we know that $\Conv(vp)$ and hence $(0,vp] \subset vK$ is not $p$-divisible. So we may choose an element $x \in K$ such that
    \[
        0 < vx \le vp \and vx \notin p vK.
    \]
    In particular, $x \in K^\times \setminus K^{\times p}$ is $p$-rigid, so there exist $i \in \{0, \ldots, p - 1\}$ and $y \in K^\times$ such that $1 + x = x^i y^p$. Note that $v(1 + x) = 0$ implies
    \[
        i = 0,\ vy = 0{,} \text{ and } y \equiv 1 \pmod{\Mm_v},
    \]
    the latter as in \eqref{eq:x^p=1 mod M}. If we write $y = 1 + z$ with $z \in \Mm_v$, then
    \[
        vx = v((1 + z)^p - 1) \ge \min\{vp + vz, v(z^p)\}.
    \]
    But $vx < vp + vz$ implies $vx = v(z^p)$, and hence $vx \in pvK$, contrary to our choice of $x$.
    
    \item \label{hensel_step:5} Assume that $v$ is not henselian. By Lemma~\ref{lem:p-henselianity criterion}, there exists a Galois extension $L/K$ of degree $p$ with pairwise inequivalent prolongations $v_1, \ldots, v_g$  of $v$ to $L$ with $g > 1$. By the fundamental equality (Corollary~\ref{cor:fundamental equality}), $p = d\;\!efg$ and therefore
    \[
        d = e = f = 1 \and g = p,
    \]
    where $v_1, \ldots, v_p$ are thus immediate, incomparable prolongations (cf. Fact~\ref{fact:compatibility}). By Kummer theory, $L = K(\sqrt[p]{x})$ for some $x \in \Oo_v^{\times}$ (otherwise, we would have $e = p$). Set $x_0 \coloneqq \sqrt[p]{x} \notin L^{\times p}$.
    
    Applying the weak approximation theorem (Fact~\ref{fact:weak_approx}), we may find for each $i \in \{1, \ldots, p\}$,
    \begin{equation} \label{eq:x_i-from-weak}
        x_i \in \Mm_{v_i} \cap \bigcap_{j \ne i} \Oo_{v_j}^\times.
    \end{equation}
    As all $v_i$ are immediate, we have $v_iL = vK$ for each $i$.
    We may assume without loss of generality, that $v_ix_i \notin pvK$, as follows: if $v_ix_i \in pvK$, use the fact that $\Conv(v_ix_i)$ is non-trivial to find $y \in K$ with $vy \in (0,v_ix_i)$ and $vy \notin pvK$ (otherwise, as in Step 4, $(0,v_ix_i] \subset vK$ would be $p$-divisible, which cannot happen). Replacing $x_i$ by $x_i + y$, gives the desired element satisfying \eqref{eq:x_i-from-weak} and $v_ix_i \notin pvK$.
    
    We claim that $L(\sqrt[p]{x_0},\sqrt[p]{x_1}, \ldots, \sqrt[p]{x_p})/L$ has Galois group $C_p^{p + 1}$. By Kummer theory, it suffices to show that
    \[
        x_0^{\vphantom{i_p}i_0} x_1^{\vphantom{i_p}i_1} \cdots x_p^{i_p} \in L^{\times p} \text{ with $i_0, \ldots, i_p \in \{0, \ldots, p - 1\}$} \Longrightarrow i_0 = \ldots = i_p = 0.
    \]
    Evaluating this product by $v_j$ yields $i_j = 0$ for every $j = 1, \ldots, p$. Recall that $x_0 \notin L^{\times p}$, so $i_0 = 0$ as well.
    By Step 1 and Lemma~\ref{lem:finite extensions classical case}, we know that $L$ admits at most two linearly disjoint $C_p$-extensions, so our construction is impossible. We conclude that $v$ must be henselian.

    \item \label{hensel_step:6} Assume $(vK : pvK) = p$ and let $x \in K$ be such that $vx \notin p vK$. It remains to verify property \ref{item:tam_branch3}, that is, $p^2 \mid |G_{Kv}|$. If $G_{Kv} = 1$, then $\Oo_v^{\times} = \Oo_{v}^{\times p}$, since $(Kv)^{\times} = (Kv)^{\times p}$ and $1 + \Mm_v = (1 + \Mm_v)^p$ by henselianity. This would imply that there is only one $C_p$-extension, namely $K(\sqrt[p]{x})$. Hence $G_{Kv} \ne 1$ (and we are done by the Artin-Schreier Theorem~\ref{thm:Artin-Schreier} in the case $p \ne 2$).
    
    Let $L \coloneqq K(\sqrt[p]{y})$, where $y \in \Oo_v^{\times}$ is such that $yv \notin (Kv)^{\times p}$. In particular, $L/K$ is unramified of degree $f = p$. As before, we cannot have $G_{Lv} = 1$, because $L$ admits more than one $C_p$-extension. This shows that $p^2 \mid G_{Kv}$. \qedhere
\end{enumerate}
\end{proof}

From this, we are able to deduce the general case. Our main tool will be ``Henseling Down'' (Theorem~\ref{thm:henseling down}). In order to apply this theorem to a tamely branching valuation $v$, we require $\Oo_v \in H(K)$ and moreover, that no proper coarsening of $v$ has real closed residue field. Conveniently, we may assume this extra condition without loss of generality:

\begin{lem}[{\cite[Lem. 5.4.6]{Engler05}}]
    \label{lem:tamely branching in H(K)}
    Let $(K,v)$ be a henselian valued field, $v$ tamely branching at $p$. Then there exists a valuation ring in $H(K)$ tamely branching at $p$ as well. 
    If $p = 2$, we can find a henselian valuation on $K$, tamely branching at $2$, such that furthermore, no proper coarsening has real closed residue field.
\end{lem}
\begin{proof}
    It suffices to consider the case $\Oo_v \in H_2(K)$, i.e., when $Kv$ is separably closed. We claim that the canonical henselian valuation $v_K$ is tamely branching at $p$. By definition, $v_K$ is coarser than $v$, so consider the decomposition:
    \[
        \begin{tikzcd}
            K \ar[rr,bend left=30,"v"] \ar[r, "v_K"]& Kv_K \ar[r, "\overline{v}"] & Kv
        \end{tikzcd}
    \]
    We see that $\charK(Kv_K) \ne p$. Moreover, $\overline{v}(Kv_K)$ is divisible as $Kv_K$ is separably closed. Thus,
    \[
        (v_KK : p v_KK) = (vK : p vK),
    \]
    and therefore, $v_KK$ cannot be $p$-divisible.
    Note that since $Kv$ is separably closed, $p^2 \nmid |G_{Kv}|$ and hence, $(v_KK : p v_KK) = (vK : p vK) \ne p$.
    
    For $p = 2$, consider the case that some coarsening of $v$ has real closed residue field. Let $w$ denote the coarsest such coarsening, which exists by Fact~\ref{fact:coarsest_real_closed}. Then, $\Oo_w \in H(K)$ automatically satisfies the required property, and it remains to show that $w$ is tamely branching at $2$. We proceed as before, with $w$ taking the place of $v_K$ and real closedness of $Kw$ playing the role of separable closedness of $Kv_K$.
    Since $Kw$ is real closed, it follows that $\charK(Kw) = 0$ and that the value group $\overline{v}(Kw)$ is divisible (for any $x \in Kw$, we have that $x$ or $-x$ has a $q$th root for any prime $q$).
    
    Again, $(wK : 2 wK) = (vK : 2 vK)$ holds as in the previous argument. The assumption that $v$ is tamely branching at $2$ implies $(vK : 2 vK) \geq 4$ (since $|G_{Kv}| = 2$, cf. Fact~\ref{fact:coarsest_real_closed}). It follows that $w$ is also tamely branching at $2$.
\end{proof}

Similarly, we need to ensure that the process of henseling down also preserves the property of being tamely branching at a prime $p$.
\begin{lem}
    \label{lem:tamely branching down algebraic}
    Let $(L,w)$ be a valued field, $w$ tamely branching at $p$. If $L/K$ is an algebraic extension of fields, then $v = w|_K$ is tamely branching at $p$ as well.
\end{lem}
\begin{proof}
Recall that $wL/vK$ is a torsion group and $Lw/Kv$ is an algebraic extension by Fact~\ref{fact:finite_orders}. In particular, $p^2 \mid G_{Lw}$ implies $p^2 \mid G_{Kv}$.
The rest of the claim follows immediately from the following general fact about abelian groups: if $\Delta \le \Gamma$ are $p$-torsion-free abelian groups such that $\Gamma/\Delta$ is a torsion group, then
\[
    \dim_{\IF_p}(\Delta/p\Delta) \ge \dim_{\IF_p}(\Gamma/p\Gamma).
\]
See \cite[Ex. 5.5.2]{Engler05} for a proof of this inequality.
\end{proof}

\begin{cor} \label{cor:Galois_hens<--}
    Let $K$ be a field and $p$ any prime. Assume $G_K$ has a Sylow $p$-subgroup $P \not\cong C_2, \IZ_p, \IZ_2 \rtimes C_2$ that contains a non-trivial normal abelian closed subgroup $N$. Then $K$ admits a henselian valuation, tamely branching at $p$.
\end{cor}
\begin{proof}
    By Theorem~\ref{thm:Artin-Schreier}, $N$ is torsion-free unless $p = 2$ and $\sqrt{-1} \notin \Fix N$. We may replace $K$ by $K(\sqrt{-1})$ in this exceptional case by Henseling Down (Theorem~\ref{thm:henseling down}(ii)) and Lemma~\ref{lem:tamely branching down algebraic}, ensuring $P \not\cong \IZ_2$. The fact that $N$ remains non-trivial is due to Lemma~\ref{lem:no_normal_C2}. In particular, we may now view $N$ as a torsion-free $\IZ_p$-module.

    \textbf{Case 1.} Let us first consider the easier case that $N$ is not topologically generated by a single element (i.e., $N \not\cong \IZ_p$), so there exists $N_0 \vartriangleleft N$ with $N_0 \cong \IZ_p \times \IZ_p$. By the ``classical case'' (Proposition~\ref{prop:classical case}), there exists a henselian valuation $w$ on $M \coloneqq \Fix N_0$, tamely branching at $p$, which we may assume to lie in $H(M)$ by Lemma~\ref{lem:tamely branching in H(K)}. Now, we may ``hensel down'' as indicated by the left diagram.
    \begin{center}
        \begin{tikzcd}[column sep=9em, cramped]
            K^{\sepc} \ar[d, dash] & K^{\sepc} \ar[d, dash] \\
            M \rlap{$\,= \Fix N_0$} \ar[d, dash, "\text{normal}"'] & F \rlap{$\,= \Fix N$} \ar[d, dash] \ar[dd, dash, bend right=55, "\text{normal}"', pos=0.48] \\
            \Fix N \ar[d, dash, "\text{normal}"'] & M \rlap{$\,= \Fix \langle N, \sigma \rangle$} \ar[d, dash] \\
            L \rlap{$\;= \Fix P$} \ar[d, dash, "\substack{\text{max.}\\\text{prime-to-$p$}}"'] & L \rlap{$\;= \Fix P$} \ar[d, dash, "\substack{\text{max.}\\\text{prime-to-$p$}}"'] \\
            K & K \\[-1em]
            \textsc{Case 1.} & \textsc{Case 2.}\\[-1.5em]
        \end{tikzcd}
    \end{center}
    
    More precisely, two applications of Henseling Down (Theorem~\ref{thm:henseling down}(i) and Lemma~\ref{lem:tamely branching down algebraic}) show that $w|_L$ is in $H(L)$ and tamely branching at $p$, where $L \coloneqq \Fix P$. By Lemma~\ref{lem:tamely branching in H(K)}, again, we may additionally assume that $L$ admits a henselian valuation $w'$, tamely branching at $p$, such that no proper coarsening has real closed residue field if $p = 2$. Therefore, $w'|_K$ is the desired valuation by Henseling Down, Theorem~\ref{thm:henseling down}(iii).

    \textbf{Case 2.} Let us now consider the case that $N \cong \IZ_p$. As $P \not\cong \IZ_p$, we may find some $\sigma \in P \setminus N$ such that $\langle \sigma \rangle \cong \IZ_p$ and $\langle N,\sigma \rangle \cong \IZ_p \rtimes \IZ_p$. By Proposition~\ref{prop:classical case}, there exists a valuation $w$ on $M \coloneqq \Fix\langle N,\sigma \rangle$, tamely branching at $p$, which we may assume to lie in $H(M)$. For $p = 2$, we additionally require that $w$ does not have any proper coarsenings with real closed residue field. Let $u$ denote the unique prolongation of $w$ to $F \coloneqq \Fix N$.
    If $\Oo_u \in H(F)$, we see that $u|_K$ is the desired valuation by Henseling Down (as indicated by the right diagram).
    
    It remains to consider the case that $u, v_F \in H_2(F)$ with $\Oo_u \subsetneq \Oo_{v_F}$. By definition, $Fv_F$ is separably closed, $\charK Fv_F \ne p$ (as $Fu$ is a residue field of $Fv_F$), and $v_FF$ is not $p$-divisible (otherwise $N = G_F$ would be trivial). Further, $\Oo_u \subsetneq \Oo_{v_F}$ implies $\Oo_w \subsetneq \Oo_{v_F|_M}$ (Fact~\ref{fact:compatibility}), therefore, $\Oo_{v_F|_M} \in H_1(M)$ since $\Oo_w \in H(M)$. By Henseling Down, $u' \coloneqq v_F|_L$ lies in $H(L)$ and is tamely branching at $p$ (since $\charK Fv_F \ne p$, $v_FF \ne p v_FF$, and $Mv_F|_M$ is not separably closed and $\sqrt{-1} \in Lu'$).
    Finally, applying Henseling Down to $L/K$ shows that $v_F|_K = u'|_K$ is henselian and tamely branching at $p$.
\end{proof}

This completes the proof of the Galois characterisation of henselianity (Theorem~\ref{thm:Galois code henselianity}).

%% file: part4_transfer.tex
\section{Transfer between fields of different characteristic}
\label{chap:transfer}

We pursue two goals in this chapter. First, we give two separate proofs of the \emph{Transfer Lemma}, which will be a key component in the proof of the Main Theorem.
\begin{transferlem*}
Let $(K,v)$ be a henselian valued field of mixed characteristic $(0,p)$.
Assume that
\vspace{-0.3em}
\begin{enumerate}[(i)]
    \item $G_K$ is small;
    \item $vK = \Conv(vp)$ and $vK$ is $p$-divisible.
\end{enumerate}
\vspace{-0.3em}
Then, there exists a field $F$ of positive characteristic $p$ such that $G_K \cong G_F$.
\end{transferlem*}
Secondly, we aim to provide a systematic treatment of the Saturation-Decomposition Method---a non-standard (i.e., model-theoretic) transfer method for valued fields.
We give some general background on transfer methods in Section~\ref{sec:two_transfer}.

The original, model-theoretic strategy for proving the Transfer Lemma was a precursor to the Saturation-Decomposition Method for non-discrete valuations. In his Habilitation, the second author asked for a purely algebraic proof of the Transfer Lemma, noting that the lemma
\begin{quote}
    ``is one of those rare examples of `applied model theory', where the only proof known so far uses model theory in an essential way.'' \cite[p.\;24]{Koenigsmann04}
\end{quote}
In \cite{Koenigsmann04}, all arguments were based on the smallness of the absolute Galois group $G_K$, which obscured the connection to the truncated valuation ring $\Oo_v/p\Oo_v$, the algebraic object at the heart of transfer from mixed characteristic to equal positive characteristic.

We give a new proof of the Transfer Lemma based on the \emph{tilting construction} as presented by Scholze \cite{Scholze12}. We should note that apart from its definition and most basic properties, we only use the Fontaine-Wintenberger Theorem~\ref{thm:FW}, for which we will provide the necessary background in Section~\ref{sec:perfectoid}.

Motivated by the theory of perfectoid fields, Jahnke and Kartas developed and greatly generalised the Saturation-Decomposition Method in \cite{Jahnke-Kartas}, further extending it to include non-standard treatments of perfectoid fields, the Fontaine-Wintenberger Theorem, and almost mathematics.

With this in mind, what we present in Section~\ref{sec:sat-decomp_method} as the Saturation-Decomposition Method is, in fact, a hybrid between \cite{Koenigsmann04} and \cite{Jahnke-Kartas}. Apart from changing emphasis, we prove an optimal version of Jahnke and Kartas' ``Taming Theorem'' \cite[1.5]{Jahnke-Kartas}---we show that it holds precisely for the class of roughly deeply ramified (rdr) fields in the sense of Gabber-Ramero/Kuhlmann-Rzepka (Theorem~\ref{thm:taming}). The Almost Purity Theorem for rdr fields (Corollary~\ref{cor:KR-almost-purity}) and the Transfer Lemma (Corollary~\ref{cor:transfer_lemma2}) follow as immediate corollaries.

In Section~\ref{sec:JK}, we briefly explain what extra ingredients are needed to recover the theorems of Jahnke-Kartas and how their work exemplifies the above-mentioned method. This will complete our picture of transfer, as sketched in the thematic diagram on p.\;\pageref{diagram:inf_ram}.
We hope that our presentation shows that the arithmetic and model-theoretic proofs of the Transfer Lemma are two sides of the same coin, and that model theory brings more than just a new language to the table.

\subsection{Two methods of transfer}
\label{sec:two_transfer}

The purpose of the Transfer Lemma is to exhibit an isomorphism of absolute Galois groups between fields of characteristic 0 and characteristic $p$. More generally, \emph{transfer} is achieved when arithmetic properties or invariants can be translated between fields of different characteristics. Usually, both fields in question belong to a common family, e.g., local/global fields.
In many natural settings, transfer can be observed.

There are at least two rigorous settings for transfer---both of which are asymptotic in nature---as is indicated by their respective slogans ``$p \rightarrow \infty$'' and ``$e \rightarrow \infty$'', see \cite{Scholze14}.

\subsubsection{\texorpdfstring{$p \rightarrow \infty$}{p→∞}} \label{subsec:p->infty}

Consider the local fields $\IQ_p$ and $\IF_p\laurent{t}$ with $p$ varying over the set of primes.
Ax and Kochen \cite{Ax-Kochen1,Ax-Kochen2,Ax-Kochen3} established the following rigorous transfer principle: any first-order property $\varphi$ (in the language of rings) holds in $\IQ_p$ if and only if it holds in $\IF_p\laurent{t}$ for all $p$ sufficiently large. In short, there is a bound $N(\varphi) \in \IN$ such that
\begin{equation*}
    \IQ_p \models \varphi \Longleftrightarrow \IF_p\laurent{t} \models \varphi \quad \text{for all $p \ge N(\varphi)$}.
\end{equation*}
Indeed, this transfer principle was one of the main goals of their series of papers. Ax and Kochen were motivated by a concrete question of Artin: does any homogeneous polynomial over $\IQ_p$ of degree $d \in \IN$ in $n > d^2$ variables have a non-trivial solution? Any field for which this statement holds is said to have the \emph{$C_2$-property}. Artin expected a positive answer since a result of Lang says that $\IF_p\laurent{t}$ is $C_2$ for all primes $p$. The Ax-Kochen transfer described above implies that the $p$-adic fields $\IQ_p$ are ``asymptotically'' $C_2$, that is to say, for any $d \in \IN$, there is a bound $N(d) \in \IN$ such that for all $p \ge N(d)$, all $\IQ_p$ satisfies the $C_2$-property for homogeneous polynomials of degree $\le d$ (this is a first-order property for any given $d$). On the other hand, it is known that $\IQ_p$ is not $C_2$ for any prime $p$.

Results obtained from this type of transfer are necessarily \emph{asymptotic} in nature. Looking at the proof of Ax-Kochen transfer, this is to be expected, as it uses the existence of non-principal ultrafilters in a crucial way. Essentially, Ax-Kochen transfer is a formal consequence of the elementary equivalence of the ultraproducts
\begin{equation} \label{eq:ultra}
    \Qq \coloneqq \prod_{p \in \IP} \IQ_p \ultra \Uu \equiv \prod_{p \in \IP} \IF_p\laurent{t} \ultra \Uu \eqqcolon \Ss
\end{equation}
for any non-principal ultrafilter $\Uu$ on the set of primes $\IP$. This is because ultraproducts allow one to interpolate first-order properties: a first-order property $\varphi$ holds for almost all fields in the families $\{\IQ_p\}_{p \in \IP}$ or $\{\IF_p\laurent{t}\}_{p \in \IP}$ if and only if $\varphi$ holds in the corresponding ultraproduct valued field $\Qq$ or $\Ss$ (Łoś's theorem).

The fact that $\IQ_p$ and $\IF_p\laurent{t}$ have the same value group and residue field implies that the same is true for $\Qq$ and $\Ss$---their value groups and residue fields are isomorphic to
\[
    \prod_{p \in \IP} \IZ \ultra \Uu \and \prod_{p \in \IP} \IF_p \ultra \Uu,
\]
respectively.
The elementary equivalence of $\Qq$ and $\Ss$ in (\ref{eq:ultra}) is then a special case of the following theorem that lies at the heart of the model theory of valued fields:
\begin{thm}[Ax-Kochen-Ershov] \label{thm:AKE}
Let $(K,v)$ and $(K',v')$ be two henselian valued fields of equicharacteristic $(0{,}0)$. Then:
\[
    (K,v) \equiv (K',v') \Longleftrightarrow vK \equiv v'K' \and Kv \equiv K'v'.
\]
\end{thm}

\subsubsection{\texorpdfstring{$e \rightarrow \infty$}{e→∞}} \label{subsec:e->infty}

In contrast to the previous method, let $p$ now be a fixed prime.
Krasner \cite{Krasner57}, Kazhdan \cite{Kazhdan86}, and Deligne \cite{Deligne84} observed that by increasing the absolute ramification index $e = v(p)$ of a local field $F$, it more and more resembles a positive characteristic local field $\IF_q\laurent{t}$.
The simplest instance of this phenomenon is, perhaps, the following observation of Kazhdan's:
\begin{fact}
Let $F/\IQ_p$ be a totally ramified extension of degree $e = e(F/\IQ_p)$. If $\pi$ is a uniformiser in $F$, then
\[
    \Oo_F/(\pi^e) = \Oo_F/(p) \cong \IF_p[t]/(t^e).
\]
If $e$ is a power of $p$, the right-hand side can be written isomorphically as $\IF_p[t^{1/e}]/(t)$.
\end{fact}
For a simple proof of this fact, see \cite[Lem. 6.1.3]{Kartas22}. We lose the dependence on $e$ by passing to a direct limit. For example, if we adjoin a compatible system of $p$-power roots
\[
    p^{1/p^{\infty}} = \bigl\{p^{1/p^n}\bigr\}_{n \ge 1}
\]
to $\IQ_p$, i.e., setting $K = \IQ_p(p^{1/p^\infty})$, this yields
\[
    \Oo_K/(p) \cong \varinjlim \Oo_{\IQ_p(p^{1/p^n})}/(p) \cong \varinjlim \IF_p\bigl[t^{1/p^n}\bigr]\ultra(t) \cong \IF_p[t]^{1/p^\infty}/(t).
\]
Examining the right-hand side, we see that the Frobenius map $x \mapsto x^p$ is surjective. By taking an inverse limit with the Frobenius map acting as the transition maps, one recovers an integral domain that is the valuation ring of an algebraic extension of $\IF_p\laurent{t}$.
\[
    \IF_p\llbracket t \rrbracket^{1/p^\infty} \cong \varprojlim_{x \mapsto x^p} \Bigl(\IF_p[t]^{1/p^\infty}/(t)\Bigr)
\]
This is the valuation ring of the $t$-adic valuation on the perfect hull $\IF_p\laurent{t}^{1/p^\infty}$ of $\IF_p\laurent{t}$. The latter is what is called the \emph{tilt} of $K$.
Fontaine and Wintenberger furthermore showed that for this construction, one gets a canonical isomorphism of absolute Galois groups between $K$ and its tilt:
\begin{equation} \label{eq:FW}
    G_{\IQ_p(p^{1/p^{\infty}})} \cong G_{\IF_p\laurent{t}^{1/p^\infty}}.
\end{equation}
In \cite[Thm. 1.3]{Scholze12}, Scholze vastly generalised this isomorphism to an equivalence of categories over perfectoid fields, algebras, and spaces and their respective tilts. We will only use the Fontaine-Wintenberger Theorem for perfectoid fields, which we explain in the next section.

In a completely distinct fashion, Koenigsmann \cite{Koenigsmann04} constructed an absolutely unramified $(0,p)$-place $\varphi : K \longrightarrow F$ for certain valued fields $K$ with small absolute Galois group and $p$-divisible value group, that are sufficiently\footnote{For example, it suffices to consider non-principal ultrapowers of valued fields---these are $\aleph_1$-saturated.} saturated (in the sense of model theory). By ramification theory, this means that $\varphi$ induces an isomorphism
\[
    G_K \cong G_F
\]
of absolute Galois groups of fields of characteristic 0 and $p$. This was a new instance of transfer for valued field with ``infinite'' absolute ramification index $e = v(p)$. By this, we mean that the interval $(0,vp] \subset vK$ has infinitely many points because $vK$ is $p$-divisible.

As we have mentioned before, the Saturation-Decomposition Method, within the formalism of Jahnke-Kartas (Sections~\ref{sec:sat-decomp_method} and~\ref{sec:JK}), will link these two interpretations of ``$e \rightarrow \infty$''.
To summarise, this chapter roughly follows the outline of the following thematic diagram.

\begin{diagram}[column sep=huge, row sep=large] \label{diagram:inf_ram}
    & \text{Perfectoid fields} \arrow[ddl, "\shortstack{\scriptsize Scholze \\ \scriptsize 2012}"', no head] \arrow[ddr, "\shortstack{\scriptsize Jahnke-Kartas \\ \scriptsize 2025}", no head] & \\
    & \shortstack{(concrete) infinitely\\ ramified valued fields} \arrow[dl, "\shortstack{\scriptsize Fontaine-\\ \scriptsize Wintenberger\\ \scriptsize 1979}", no head] \arrow[dr, "\shortstack{\scriptsize Koenigsmann \\ \scriptsize 2004}"', no head] & \\
    \text{Abs. Galois groups} && \text{Model theory}
\end{diagram}

\subsection{Perfectoid fields}
\label{sec:perfectoid}

In this section, we consider valued fields $(K,v)$ that are complete and of rank 1 (this is the familiar setup of algebraic number theory). This is essential because the constructions involved are analytic in nature. We further assume that $(K,v)$ has mixed characteristic $(0,p)$. In particular, the residue map
\[
    \res : \Oo_v \longrightarrow Kv, \quad \alpha \longmapsto \overline{\alpha},
\]
reduces elements that live in characteristic 0 to elements that live in characteristic $p$. If we wish to exhibit a structural connection between $K$ and $Kv$, it would be pertinent if the residue map had a section. Certainly, it cannot be additive (since the characteristics differ).
However, if $p$ is a uniformiser, there is a unique multiplicative section $\tau$ of the residue map called the \emph{Teichmüller lift}. Its construction is based on the following basic fact:
\begin{fact}
Let $R$ be a commutative ring. Then for any $\alpha, \beta \in R$ and $n \in \IN$,
\[
    \alpha \equiv \beta \mod{(p)} \Longrightarrow \alpha^{p^n} \equiv \beta^{p^n} \mod{\left(p^{n + 1}\right)}.
\]
\end{fact}
Assume that $p$ is a uniformiser. Then the above fact provides a source of canonicity for lifts: if $\alpha$ and $\beta$ are different representatives of the same residue class in $Kv$, their $p^n$-th powers will lie closer and closer together as $n$ grows. Following this idea, we define
\[
    \tau : Kv \longrightarrow \Oo_v, \quad a_0  \longmapsto \lim_{n \rightarrow \infty}\alpha_n^{p^n}, \quad \text{where } \alpha_n \in \res^{-1}(a_n),\ a_n \coloneqq a_0^{1/p^n} = \Phi^{-n}(a_0),
\]
and $\Phi$ denotes the Frobenius on $Kv$.
Crucially, $K$ needs to be complete for the limit to be well-defined. Using the language of inverse limits, we can rewrite this in a more suggestive fashion. The Teichmüller lift $\tau$ is given as the composition
\begin{equation} \label{eq:sharp}
    \begin{tikzcd}[row sep=0em]
        Kv \ar[r,"\cong","\iota"'] & \displaystyle\varprojlim_{\Phi} Kv = \varprojlim \left(\Oo_v/p\Oo_v \stackrel{\Phi}{\longleftarrow} \Oo_v/p\Oo_v \stackrel{\Phi}{\longleftarrow} \cdots \right) \ar[r,"\sharp"] & \Oo_v \\
        a_0 \ar[r, maps to] &  (a_n)_{n \ge 0} = {\bigl(a_0, a_0^{1/p}, a_0^{1/p^2}, \ldots\bigr)} \ar[r, maps to] & \displaystyle\lim_{n \rightarrow \infty} \alpha_n^{p^n}
    \end{tikzcd}
\end{equation}
of the canonical isomorphism $\iota$ (that sends $a_0$ to the unique tuple of iterated $p$-th roots) and a map denoted by $\sharp$, which is called the \emph{sharp map}. If we leave aside $\iota$ and weaken our assumption on $\Phi$, requiring only surjectivity rather than bijectivity, we are naturally led to the following new objects and definitions:

\begin{defi}[{\cite[Def. 3.1]{Scholze12}}]
Let $(K,v)$ be a complete valued field of rank 1 and residue characteristic $p$. We call $K$ \emph{perfectoid} if the value group is non-discrete and the truncated valuation ring $\Oo_v/p\Oo_v$ is semi-perfect, i.e., the Frobenius on it is surjective. We write
\[
    \Oo_v^{\flat} \coloneqq \displaystyle\varprojlim_{\Phi} \Oo_v/p\Oo_v = \varprojlim \left(\Oo_v/p\Oo_v \stackrel{\Phi}{\longleftarrow} \Oo_v/p\Oo_v \stackrel{\Phi}{\longleftarrow} \cdots \right)
\]
for the \emph{tilt} of the valuation ring and $\sharp : \Oo_v^\flat \longrightarrow \Oo_v$ for the sharp map defined above.
\end{defi}

\begin{rem}
\begin{enumerate}
    \item Perfectoid fields have a $p$-divisible value group (we will see the short argument presented as Lemma~\ref{lem:semi-perf->p-div}) and perfect residue fields since $\Oo_v/\Mm_v \cong (\Oo_v/p\Oo_v)_{\Red}$, see Theorem~\ref{thm:stand_decomp}\ref{item:O induced}.
    \item By definition, $\Oo_v^\flat$ has a ring structure, whereas $\sharp$ is just a multiplicative map.
\end{enumerate}
\end{rem}

\begin{example}
\begin{enumerate}[(1)]
    \item Perfectoid fields of characteristic $p$ are precisely perfect fields endowed with a complete rank 1 valuation. A basic example is the completion of $\IF_p\laurent{t}^{1/p^\infty}$.
    \item Typical examples of perfectoid fields of characteristic $0$ are complete valued fields that arise from adjoining a compatible system of $p$-power roots to $\IQ_p$. For instance, the completions of $\IQ_p(p^{1/p^{\infty}})$ and $\IQ_p(\zeta_{p^{\infty}})$,
    \[
        \widehat{\IQ_p(p^{1/p^{\infty}})} \and \widehat{\IQ_p(\zeta_{p^{\infty}})},
    \]
    are perfectoid. Here, $\zeta_{p^{\infty}} = \{\zeta_{p^n}\}_{n \ge 1}$ denotes a compatible system of $p^n$-th primitive roots of unity.
\end{enumerate}
\end{example}

We summarise the basic properties of the tilt, which, on their own, are not very hard to prove (see the original paper \cite[Lem. 3.4]{Scholze12} or Bhatt's lecture notes \cite[Chap. 3.2]{Bhatt17} for a more detailed treatment).
\begin{fact} \label{fact:perfectoid_props}
Let $K$ be a perfectoid field of mixed characteristic $(0,p)$. Then
\begin{enumerate}[(i)]
    \item $\Oo_v^{\flat}$ is a valuation ring. The corresponding valuation on $K^{\flat} \coloneqq \Frac(\Oo_v^{\flat})$ is given by the composition
    \[
        v^{\flat} : K^{\flat} \stackrel{\sharp}{\longrightarrow} K \stackrel{v}{\longrightarrow} \Gamma_v \cup \{\infty\},
    \]
    where $\sharp$ is extended multiplicatively to the quotient field. Moreover, any element $x^{\sharp}$ in the image of\; $\sharp$ admits a compatible system of $p^n$-th roots $(x^{1/p^n})^{\sharp}$.
    \item There exists $t \in \Oo_v^{\flat}$ with $v(t^{\sharp}) = vp$ such that $\sharp$ induces an isomorphism
    \[
        \overline{\sharp} : \Oo_v^{\flat}/t\Oo_v^{\flat} \cong \Oo_v/p\Oo_v.
    \]
    \item $(K^{\flat}, v^{\flat})$ is a perfectoid field of characteristic $p$, and it has the same value group and residue field as $K$.
\end{enumerate}
\end{fact}

The main Galois-theoretic theorem about perfectoid fields due to Kedlaya-Liu and Scholze, generalising (\ref{eq:FW}), is the following:
\begin{thm}[Generalised Fontaine-Wintenberger] \label{thm:FW}
For any perfectoid field $K$, we have
\[
    G_K \cong G_{K^{\flat}}.
\]
\end{thm}

\subsubsection{Direct proof of the Transfer Lemma}

The following lemma---the proof of which will be genuinely valuation-theoretic---is a useful source for perfectoid fields. The Transfer Lemma follows as an immediate corollary.

\begin{lem} \label{lem:finite->perfectoid}
Let $(K,v)$ be a valued field of mixed characteristic $(0,p)$ such that
\begin{enumerate}[(i)]
    \item $\dim_{\IF_p} K^{\times}/(K^{\times})^p < \infty$;
    \item the value group $vK$ is $p$-divisible and of rank 1;
    \item the residue field $Kv$ is perfect.
\end{enumerate}
Then, $\Oo_v/p\Oo_v$ is semi-perfect. In particular, the completion of $(K,v)$ is perfectoid.
\end{lem}
\begin{proof}
Let $(\widehat{K},\hat v)$ be the completion of $(K,v)$. The second part of the conclusion follows immediately from
\begin{align*}
    \hat v\widehat{K} & \cong vK, \\
    \widehat{K}\hat v & \cong Kv, \\
    \Oo_v/p\Oo_v & \cong \Oo_{\hat v}/p\Oo_{\hat v}.
\end{align*}
To prove that $\Oo_v/p\Oo_v$ is semi-perfect, fix any $x \in \Oo_v$. We will define a sequence of increasingly better approximations to a $p$-th root of $x$.

Let $\{a_1(K^{\times})^p, \ldots, a_n(K^{\times})^p\}$ be a set of representatives for $K^{\times}/(K^{\times})^p$. Using the condition that $vK$ is $p$-divisible, we can assume without loss of generality that $a_1, \ldots, a_n \in \Oo_v^{\times}$. We may further assume $a_1, \ldots, a_n \in 1 + \Mm_v$, again without loss of generality, since $Kv$ is perfect. Set $\gamma \in vK$ to be the value
\[
    \gamma \coloneqq \min_{1 \le i \le n} v(a_i - 1) > 0.
\]
Using that $vK$ is archimedean, we choose $m \in \IN$ such that $m \gamma \ge vp$.

We can now construct our approximations\footnote{We note that an essentially identical approximation argument is given in the proof of \cite[Prop.~4.4]{Kuhlmann-Rzepka}.} to a $p$-th root of $x$. Choose $y_1, y_2, \ldots, y_m \in \Oo_v$ and $z_1, z_2, \ldots, z_m \in \{a_1 - 1, \ldots, a_n - 1\}$ inductively such that
\begin{align*}
    x & = y_1^p (1 + z_1) \\
    z_i & = y_{i + 1}^p (1 + z_{i + 1})
\end{align*}
for $i = 1, \ldots, m - 1$. Therefore, we obtain the sum decomposition
\[
    x = y_1^p (1 + y_2^p (1 + \ldots(1 + z_m)\ldots)) = y_1^p + y_1^p y_2^p + \ldots + y_1^p y_2^p \cdots y_m^p + y_1^p y_2^p \cdots y_m^p z_m.
\]
By construction,
\[
    v(y_1^p y_2^p \cdots y_m^p z_m) = v(x) + v(z_1) + v(z_2) + \ldots + v(z_m) \ge m \gamma \ge vp,
\]
so we have $y_1^p y_2^p \cdots y_m^p z_m \in p\Oo_v$. The Frobenius map (which is additive on $\Oo_v/p\Oo_v$) thus sends
\[
    y_1 + y_1 y_2 + \ldots + y_1 y_2 \cdots y_m
\]
to our chosen $x$ modulo $p$.
\end{proof}

\begin{cor}[Transfer Lemma] \label{cor:transfer_lemma1}
Let $(K,v)$ be a henselian valued field of mixed characteristic $(0,p)$. Assume that
\begin{enumerate}[(i)]
    \item $G_K$ is small;
    \item $vK = \Conv(vp)$ and $vK$ is $p$-divisible.
\end{enumerate}
Then, there exists a field $F$ of positive characteristic $p$ such that $G_K \cong G_F$.
\end{cor}
\begin{proof}
Smallness of $G_K$ implies that $\dim_{\IF_p} K^{\times}/(K^{\times})^p < \infty$ by Lemma~\ref{lem:small_Gal->finite_dim}.
Consider the standard coarsening $v_p$, which, by our assumptions on $vK$, has $p$-divisible value group $v_pK$ of rank 1. By Pop's Lemma~\ref{lem:Pop}(i), $Kv_p$ is perfect. Therefore, $(K,v_p)$ satisfies all the conditions of the preceding lemma, making $(\widehat{K},\widehat{v_p})$ a perfectoid field. Hence
\[
    G_{\vphantom{\widehat{K}}K} \cong G_{\widehat{K}} \cong G_{\widehat{K}^{\flat}}
\]
by Proposition~\ref{prop:abs_gal_completion} and the Generalised Fontaine-Wintenberger Theorem. So we may choose $F = \widehat{K}^{\flat}$ as our characteristic $p$ field.
\end{proof}

For the purpose of proving our Main Theorem, this is all we need. In particular, no saturated models were involved in the construction of $F$ --- instead, we provided a purely arithmetic description via completions and tilts.

\subsection{The Saturation-Decomposition Method} \label{sec:sat-decomp_method}

For the rest of the chapter, we will assume familiarity with the basic concepts of model theory (such as theories, types, saturation, and ultraproducts).

To the best of our knowledge, the Saturation-Decomposition Method originates in Kochen's individual account \cite{Kochen75} of his work with Ax. The idea of the method is roughly the following:

\begin{slogan*}
Given a valued field $(K,v)$, study its structure in two steps:
\begin{enumerate}[(i)]
    \item \textbf{Saturate} by passing to an extension with strong closure properties whilst preserving all elementary properties of the valued field (i.e., pass to a saturated elementary extension).
    \item \textbf{Decompose} the saturated extension into places that can each be analysed individually.
\end{enumerate}
\end{slogan*}

\subsubsection{The discrete case}

We start with a classic example. The theorem below characterises the first-order theory of $(\IQ_p,v_p)$. This serves a dual purpose: it is the simplest example of the Saturation-Decomposition Method at work, and, at the same time, it characterises fields elementarily equivalent to $\IQ_p$ via four natural axioms. It is precisely these four axioms that we will have to check in the proof of the Main Theorem.
\begin{thm} \label{thm:char_Qp}
A valued field $(K,v)$ is elementary equivalent to $(\IQ_p,v_p)$ if and only if
\begin{enumerate}[(i)]
    \item $(K,v)$ is henselian of mixed characteristic $(0,p)$;
    \item $vp$ is minimal positive;
    \item $vK \equiv \IZ$;
    \item $Kv = \IF_p$.
\end{enumerate}
\end{thm}
\begin{proof}[Proof]
Properties (i)--(iv) clearly hold in $(\IQ_p,v_p)$ and it is not difficult to verify that (i)--(iv) can be expressed by a set of first-order sentences. Any valued field elementarily equivalent to $(\IQ_p,v_p)$ will thus satisfy (i)--(iv) as well.

It now suffices to show that (i)--(iv) form a complete theory. To this end, let $(K^*,v^*)$ and $(K',v')$ be two $\aleph_1$-saturated models of (i)--(iv). Our aim is to show $(K^*,v^*) \equiv (K',v')$. First, decompose $v^*$ and $v'$ (Remark~\ref{rem:stand_decomp_mixed}):
\newcommand{\eqdotted}{\mathrel{\ooalign{$\cong$\cr\hfil\textcolor{white}{\rule[.18ex]{.1em}{1.4ex}\hspace{.2em}\rule[.18ex]{.1em}{1.4ex}}\hfil\cr}}}
\begin{center}
\begin{tikzcd}[remember picture, cramped]
    K^* \ar[r,"v^*_0"] & K^*v^*_0 \ar[r,"\overline{v^*}"] & \IF_p \\
    K' \ar[r,"v'_0"] & K'v'_0 \ar[r,"\overline{v'}"] & \IF_p
\end{tikzcd}
\begin{tikzpicture}[overlay,remember picture]
    \path (\tikzcdmatrixname-1-2) to node[midway,sloped]{$\eqdotted$}
    (\tikzcdmatrixname-2-2);
    \path (\tikzcdmatrixname-1-3) to node[midway,sloped]{$=$}
    (\tikzcdmatrixname-2-3);
\end{tikzpicture}
\end{center}
Properties (ii) and (iii) imply that $\overline{v^*}$ and $\overline{v'}$ have value groups isomorphic to $\IZ$, whereas the value groups of $v^*_0$ and $v'_0$ are non-trivial and divisible (cf. Theorem~\ref{thm:stand_decomp}\ref{item:induced value groups} and Fact~\ref{fact:Z-groups}). In particular, the completeness of the theory of non-trivial divisible ordered abelian groups implies $v^*_0K^* \equiv v'_0K'$. Moreover, saturation implies that $(K^*v^*_0,\overline{v^*})$ and $(K'v'_0,\overline{v'})$ are complete. By a well-known structure theorem (\cite[Chap. II, \S 5.3]{Serre79}), $\IQ_p$ is the only complete valued field with value group $\IZ$, minimal positive value $vp$, and residue field $\IF_p$, up to isomorphism. Therefore, $K^*v^*_0 \cong K'v'_0$. By the Ax-Kochen-Ershov Theorem~\ref{thm:AKE}, we conclude
\[
    (K^*,v^*_0) \equiv (K',v'_0).
\]
While $v^*_0$ and $v'_0$ are not the valuations we started with, note that $\Oo_{v^*}$ and $\Oo_{v'}$ are quantifier-free definable via J. Robinson's formula
\begin{equation} \label{eq:JR-formula}
    \varphi(x) \coloneqq \begin{cases}
        \exists y\, (1 + 2 x^3 = y^3) & \text{if $p = 2$} \\
        \exists y\, (1 + p x^2 = y^2) & \text{if $p \ne 2$},
    \end{cases}
\end{equation}
so $K^* \equiv K'$ (as rings) implies $(K^*,v^*) \equiv (K',v')$ by Fact~\ref{fact:val_lang}.
\end{proof}
\begin{rem}
It might be tempting to consider just one saturated model $(K^*,v^*)$ and compare its theory to $(\IQ_p,v_p)$ directly. This does not work, however, since the (0,\:\!0)-place in the Standard Decomposition is trivial and thus not elementarily equivalent to $v^*_0$. Working instead with the field of Puiseux series over $\IQ_p$ as a ``canonical model'' of (i)--(iv) would do the trick.
\end{rem}

\begin{rem} \label{rem:Witt}
This method is the standard approach to reducing problems in mixed characteristic $(0,p)$ to equicharacteristic (0,\:\!0) in model theory and remains widely used. For the above strategy, it is crucial that we are able to determine the core field via structure theorems for Witt rings or, more generally, Cohen rings (for the case the residue field is imperfect, see \cite{Anscombe-Jahnke}). Note that in the above proof of Theorem~\ref{thm:char_Qp}, we used the structure theorem for the Witt ring over $\IF_p$. These tools only become available if the core field is $\IZ$-valued, or equivalently, if the original valued field is finitely ramified (the interval $(0,vp] \subset vK$ contains finitely many points).
The strongest model-theoretic results one could hope for were established in \cite{ADF24}, which handles this general case (henselian, finitely ramified valued fields).
\end{rem}

\subsubsection{The non-discrete case}

The case when $(0,vp] \subset vK$ has infinitely many points is handled quite differently. Here, it is not the core field that plays the central role, but rather the positive characteristic residue field $Kv_p$ (from the Standard Decomposition~\ref{rem:stand_decomp_mixed}).
The key assumption we need to handle this case is the semi-perfectness of $\Oo_v/p\Oo_v$.

We now present a model-theoretic criterion for semi-perfectness. The proof will use a preservation lemma along the Standard Decomposition.

\begin{lem} \label{lem:semi-perf->p-div}
Let $(K,v)$ be a valued field of mixed characteristic $(0,p)$. If $\Oo_v/p\Oo_v$ is semi-perfect and $vp$ is not minimal positive, then $\Conv(vp)$ is $p$-divisible.
\end{lem}
\begin{proof}
Use the following trick (cf. \cite[Lem. 3.2]{Scholze12}): whenever
\[
    x \equiv y^p \pmod{p\Oo_v}, \quad 0 \le vx < vp,
\]
we must have $vx = p \cdot vy$. Hence, the interval $[0,vp) \subset vK$ is $p$-divisible. Moreover, $vp \in pvK$, which can be seen by writing $vp = \gamma + (vp - \gamma)$ for some $\gamma \in (0,vp)$.
But $[0,vp]$ generates all of $\Conv(vp)$, which must then be $p$-divisible as well.
\end{proof}

\begin{lem} \label{lem:semi_perf_up and down}
Let $(K,v)$ be a valued field of mixed characteristic $(0,p)$. Assume that $vp$ is not minimal positive in its archimedean class (equivalently, $v_p p$ is not minimal positive in $v_p K$). Let $v = \dbloverline{v} \circ \overline{v_p} \circ v_0$ be decomposed as in the Standard Decomposition. Then the following are equivalent:
\begin{enumerate}[(i)]
    \item $\Oo_v/p\Oo_v$ is semi-perfect;
    \item $\Oo_{\overline{v}}/p\Oo_{\overline{v}}$ is semi-perfect;
    \item $\Oo_{v_p}/p\Oo_{v_p}$ is semi-perfect;
    \item $\Oo_{\overline{v_p}}/p\Oo_{\overline{v_p}}$ is semi-perfect.
\end{enumerate}
\end{lem}
\begin{proof}
We may view (iii) $\Leftrightarrow$ (iv) as a special case of (i) $\Leftrightarrow$ (ii).

(i) $\Rightarrow$ (ii). Note that $\Oo_v$ projects onto $\Oo_{\overline{v}} = \Oo_v/\Mm_{v_0}$.

(i) $\Rightarrow$ (iii). Let $x \in \Oo_{v_p} \supseteq \Oo_v$. There is nothing to show if $x \in \Oo_v$, so assume $x \notin \Oo_v$. This implies $x^{-1} \in \Oo_v$; in particular, $x \in \Oo_{v_p}^{\times}$ and $x^{1 - p} \in \Oo_v$. Hence, we may find $y_0, z_0 \in \Oo_v$ with $x^{1 - p} = y^p + pz$. But then $x = y'^p + pz'$ for $y' = xy \in \Oo_{v_p}$ and $z' = x^p z \in \Oo_{v_p}$.

(ii) $\Rightarrow$ (i). Let $x \in \Oo_v \subseteq \Oo_{v_0}$. There is nothing to show if $x \in p\Oo_v$, so assume $x \notin p\Oo_v \supset \Mm_{v_0}$. This implies $x \in \Oo_{v_0}^\times$, so we can find $y, z \in \Oo_v \cap \Oo_{v_0}^\times$ such that $\overline{x} = \overline{y}^p + p\overline{z}$ in $\Oo_{\overline{v}} = \Oo_v/\Mm_{v_0}$. 
However, since $\Mm_{v_0} \subset p\Oo_v$, we must have $x - y^p \in p\Oo_v$.

(iii) $\Rightarrow$ (i). Write $\Gamma = vK$ and recall that $\Gamma_p$ is the maximal convex subgroup not containing $vp$. By assumption and Lemma~\ref{lem:semi-perf->p-div}, the value group $v_pK = \Gamma/\Gamma_p$ is $p$-divisible.
Let $x \in \Oo_v \subseteq \Oo_{v_p}$ and find $y_0, z_0 \in \Oo_{v_p}$ such that $x = y_0^p + pz_0$. A posteriori, we observe that $y_0 \in \Oo_v$ because $pz_0 \in \Oo_v$ by Theorem~\ref{thm:stand_decomp}\ref{item:O varpi}. If $z_0 \in \Oo_v$, we are done; otherwise $z_0 \in \Oo_{v_p}\setminus\Oo_v$, i.e., $vz_0 < 0$ and $vz_0 \in \Gamma_p$. We would like to use assumption (iii) a second time, but this will not be successful if we apply it to $pz_0$ directly. Instead, let $c \in \Oo_v$ be any constant of value $v_pc = \frac{1}{p^2}v_pp \in \Gamma/\Gamma_p$, so that $v(pc^{-p}) \gg -vz_0$ and hence $pc^{-p}z_0 \in \Oo_v$.
In a second pass, we may write $pc^{-p}z_0 = y_1^p + pz_1$, again with $y_1 \in \Oo_v$ and $z_1 \in \Oo_{v_p}$. Therefore,
\[
    x = y_0^p + pz_0 = y_0^p + (cy_1)^p + pc^pz_1.
\]
Finally, note that $y_0^p + (cy_1)^p \equiv (y_0 + cy_1)^p$ modulo $p\Oo_v$ and $c^pz_1 \in \Oo_v$.
\end{proof}

\begin{rem}
The somewhat unusual condition about the value $vp$ is essential (it is not enough to require that $vp$ is not minimal positive in $vK$), as the following example shows (suggested to us by Margarete Ketelsen). Consider the fraction field of the Witt vector ring $K = \Frac(W(k))$ over the field $k = \IF_p(t)^{1/p^{\infty}}$, the perfect hull of $\IF_p(t)$. Endow $K$ with the composite valuation $v = v_t \circ w$, where $w$ is the valuation corresponding to the Witt ring and $v_t$ is the $t$-adic valuation on $k$; in particular, $w = v_p$. Even though $vp$ is not minimal positive in $vK$, $v_pp = wp$ is minimal positive in $v_pK = \IZ$. Moreover, $\Oo_{v_p}/p\Oo_{v_p} \cong k$ is perfect. However, $\Oo_v/p\Oo_v$ is not semi-perfect: the element $t^{-1}p$ does not admit a $p$-th root modulo $p$---any such root would have value $\frac 1 p v(t^{-1}p)$, which reduces to $\frac 1 p v_pp$ in $v_pK$. This would contradict the minimality of $v_pp$.
\end{rem}

\begin{thm}[Criterion for semi-perfectness] \label{thm:semi-perf_char}
Let $(K,v)$ be a valued field of mixed characteristic $(0,p)$. Then the following are equivalent:
\begin{enumerate}[(i)]
    \item $\Oo_v/p\Oo_v$ is semi-perfect and $vp$ is not minimal positive in $vK$.
    \item $K^*v^*_p$ is perfect and $\overline{v^*}(K^*v^*_0)$ is $p$-divisible for any (possibly trivial) elementary extension $(K^*,v^*) \succcurlyeq (K,v)$.
    \item $K^*v^*_p$ is perfect and $\overline{v^*_p}(K^*v^*_0)$ is $p$-divisible for some $\aleph_1$-saturated elementary extension $(K^*,v^*) \succcurlyeq (K,v)$.
\end{enumerate}
\end{thm}
\begin{proof}
For improved readability, we abbreviate $(F,w) \coloneqq (K^*v^*_0,\overline{v^*_p})$, which is, we recall, the ``middle'' rank~1 place in the Standard Decomposition.
\begin{diagram}
    K^* \ar[r,"v^*_0"] & F = K^*v_0^* \ar[r,"w = \overline{v^*_p}"] & K^*v_p^* \ar[r] & K^*v^*
\end{diagram}

(i) $\Rightarrow$ (ii). Both properties in (i) are first-order and are, hence, preserved in elementary extensions. We may write the residue field $K^*v_p^* = Fw$ as
\[
    Fw = \Frac ((\Oo_{v^*}/p\Oo_{v^*})_{\Red})
\]
by Theorem~\ref{thm:stand_decomp}\ref{item:O induced}. Consequently, if the Frobenius map on $\Oo_{v^*}/p\Oo_{v^*}$ is surjective, then so is the Frobenius on $Fw$.
The value group $\overline{v^*}F$ is given by $\Conv(vp)$, cf. Theorem~\ref{thm:stand_decomp}\ref{item:induced value groups}, and hence $p$-divisible by Lemma~\ref{lem:semi-perf->p-div}.

(ii) $\Rightarrow$ (iii). Note that $\overline{v^*_p}(K^*v^*_0) = wF$ must be $p$-divisible as a quotient of $\Conv(vp)$, which itself is $p$-divisible.

(iii) $\Rightarrow$ (i).  It suffices to show that $\Oo_{v^*}/p\Oo_{v^*}$ is semi-perfect and $v^*p$ is not minimal positive, since (i) is a first-order property.
If $wF$ is $p$-divisible, then in particular, $v^*_p p$ is not minimal positive in $v^*_p K^*$. Therefore, by applying Lemma~\ref{lem:semi_perf_up and down}, it suffices to show that $\Oo_w/p\Oo_w$ is semi-perfect.

The remaining argument is a transfinite version of the construction of an approximate $p$-th root (similar to the proof of Lemma~\ref{lem:finite->perfectoid}).

Let $a \in \Oo_w$. By transfinite recursion, we define a sequence $\{a_{\alpha}\}_{\alpha < \omega_1}$ in $\Oo_w$ satisfying
\begin{equation} \label{eq:transfinite}
    w(a_{\beta}^p - a) > w(a_{\alpha}^p - a) \ge 0 \quad \text{if $w(a_{\alpha}^p - a) < wp$}
\end{equation}
for all ordinals $\alpha < \beta < \omega_1$, as follows:
\begin{enumerate}
    \setcounter{enumi}{-1}
    \item Set $a_0 = 0$.
    \item For any $a_{\alpha}$ previously defined, set $a_{\alpha + 1} = a_{\alpha}$ if $w(a_{\alpha}^p - a) \ge wp$. Otherwise, using the assumption that $wF$ is $p$-divisible and $Fw$ is perfect, choose $b \in \Oo_w$ such that
    \begin{equation} \label{eq:def_b}
        w(b^p) = w(a_{\alpha}^p - a) \and
        \frac{b^p}{a_{\alpha}^p - a} \equiv 1 \pmod{\Mm_w}.
    \end{equation}
    Set $a_{\alpha + 1} = a_{\alpha} - b$. We obtain
    \begin{gather*}
        a_{\alpha + 1}^p - a = (a_{\alpha} - b)^p - a \equiv a_{\alpha}^p - a - b^p \pmod{p\Oo_w} \\
        w(a_{\alpha}^p - a - b^p) > w(a_{\alpha}^p - a),
    \end{gather*}
    the latter implied by (\ref{eq:def_b}). Together, these yield
    \[
        w(a_{\alpha + 1}^p - a) \ge \min\{w(a_{\alpha}^p - a - b^p), wp\} > w(a_{\alpha}^p - a).
    \]
    \item Let $\lambda < \omega_1$ be a limit ordinal. As before, set $a_{\lambda} = a_{\alpha}$ if $w(a_{\alpha}^p - a) \ge wp$ for some $\alpha < \lambda$. Otherwise, first lift $a$ and each $a_{\alpha}$ from $\Oo_w = \Oo_{v^*_p}/\Mm_{v^*_0}$ to elements $c, c_{\alpha}$ in $\Oo_{v^*_p} \subseteq K^*$. It then follows from $\Mm_{v_p^*} \subseteq \Mm_{v^*}$ and (\ref{eq:transfinite}) that 
    \[
        v^*(c_{\beta}^p - c) > v^*(c_{\alpha}^p - c) \quad \text{for all $\alpha < \beta < \lambda$.} 
    \]
    Therefore, the set of formulas
    \[
        p(x) = \{v^*(x^p - c) > v^*(c_{\alpha}^p - c)\}_{\alpha < \lambda}
    \]
    defines a 1-type over constants $\{c\} \cup \{c_{\alpha}\}_{\alpha < \lambda}$. Using saturation of $(K^*,v^*)$, we may find a realisation $d \in K^*$. In particular, $d \in \Oo_{v^*_p}$, and so we may set $a_{\lambda} = d\Mm_{v^*_0} \in \Oo_w$.
\end{enumerate}
The sequence defined above satisfies (\ref{eq:transfinite}) by construction. If $w(a_{\alpha}^p - a) < wp$ for all $\alpha < \omega_1$,  we would obtain an uncountable, strictly increasing sequence in the interval $[0,wp]$. But this is impossible because any interval in an archimedean ordered abelian group is second-countable. The only other option is that $w(a_{\alpha}^p - a) \ge wp$ for some countable ordinal $\alpha$, which yields a $p$-th root of $a$ modulo $p$, as required.
\end{proof}

In a second step, we characterise the property that $\Gamma/\Gamma_0$ is divisible. To this end, define:

\begin{defi}
Let $\Gamma$ be an ordered abelian group and $\gamma \in \Gamma_{>0}$ a fixed constant. If $\Gamma$ satisfies the axiom scheme
\[
    \forall x > n\gamma\, \exists y \exists z\, (x = ny + z \wedge 0 \le z < n\gamma)
\]
for all integers $n > 1$, we say that $\Gamma$ is \emph{regular above $\gamma$}.
\end{defi}

\begin{rem}
This is compatible with the usual notion of \emph{regularity} for ordered abelian groups, as it is used by Robinson and Zakon \cite{Robinson-Zakon}. Indeed, an ordered abelian group $\Gamma$ is regular if and only if $\Gamma$ is regular above $\gamma$ for all $\gamma \in \Gamma_{>0}$.
\end{rem}

\begin{prop}[Characterisation of regularity above $\gamma$] \label{prop:reg_char}
Let $\Gamma$ be an ordered abelian group and $\gamma \in \Gamma_{>0}$. The following are equivalent:
\begin{enumerate}[(i)]
    \item $\Gamma$ is regular above $\gamma$;
    \item $\Gamma/\Conv(\gamma)$ is divisible;
    \item $\Gamma^*/\Conv(\gamma)$ is divisible for any elementary extension $(\Gamma^*,\gamma) \succcurlyeq (\Gamma,\gamma)$.
\end{enumerate}
\end{prop}
\begin{proof}
It suffices to show that (i) $\Leftrightarrow$ (ii), since being regular above $\gamma$ is a first-order property.

(i) $\Rightarrow$ (ii). Write $\Delta \coloneqq \Conv(\gamma)$. Let $x \in \Gamma\setminus\Delta$ be positive and $n \ge 2$. By regularity, we may find $y \in \Gamma$ and $z \in \Delta$ such that $x = ny + z$, so $\overline{x} = n \overline{y} \in \Gamma/\Delta$.

(ii) $\Rightarrow$ (i). Let $n \ge 2$ and $x > n \gamma$ be given. By divisibility of $\Gamma/\Delta$, we can find $y' \in \Gamma$ such that $z' \coloneq x - ny' \in \Delta$. Choose $k \in \IZ$ such that $k\cdot n \gamma \le z' <(k+1)\cdot n \gamma$ and define $z \coloneq z' -k\cdot n\gamma$, where $0 \le z < n\gamma$. Setting $y = y' + k\gamma$, we obtain the desired expression $x = ny + z$.
\end{proof}

\begin{fact}[{\cite[p.~25]{Koenigsmann04}, \cite[Thm.~1.13]{AK16}}] \label{fact:sat->defectless}
Let $K$ be a field endowed with a non-trivial valuation $v$. Assume that $v$ decomposes as $v = v'' \circ w \circ v'$ with $w$ of rank 1.
If $(K,v)$ is $\aleph_1$-saturated, then $w$ has value group $\IZ$ or $\IR$. Moreover, $w$ is henselian defectless (cf. Definition~\ref{def:abs_ram}).
\end{fact}
\begin{proof}
Let $E$ be the residue field of $v'$ on which $w$ is defined. Without loss of generality, we may assume $wE$ is a subgroup of $\IR$ and $1 \in wE$. Choose $\varpi \in \Oo_{v'}^{\times}$ such that $w\overline{\varpi} = 1$.

\textit{Claim 1. $wE$ is isomorphic to $\IZ$ or $\IR$.}

\begin{claimproof}{1}
If $wE$ is discrete, then $wE \cong \IZ$. Otherwise, $wE$ is dense. For any $\gamma \in \IR$, the set
\[
    p(x) = \left\{s \cdot v\varpi < v(x) < t\cdot v\varpi \right\}_{s,\:\! t\;\! \in\;\! \IQ,\, \gamma\;\! \in\:\! (s,t)}
\]
defines a 1-type in $(K,v)$ by density. Any realisation $x$ of $p(x)$ lies in $\Oo_{v'}^{\times}$ and satisfies $w\overline{x} = \gamma$. A realisation exists by saturation and thus $\gamma \in wE$.
\end{claimproof}

By a similar argument, $(E,w)$ is complete of rank 1 and hence henselian \cite[Chap.~II, (4.6), (6.7)]{Neukirch99}.

We present the proof of defectlessness only in the case that $E$ is perfect (this is the only case we need later for our applications; see Remark~\ref{rem:defectless_generalities} below for the general case).

\textit{Claim 2. $(E,w)$ is algebraically maximal.}

\begin{claimproof}{2}
Let $E(\alpha)/E$ be an immediate finite simple extension and $\widetilde{w}$ the unique prolongation of $w$ to $E^{\sepc}$.
Assume towards a contradiction that $\alpha$ has minimal polynomial $f(X) \in E[X]$ with roots $\alpha = \alpha_1, \alpha_2, \ldots, \alpha_n \notin E$.
By the Conjugacy Theorem~\ref{thm:Conjugacy},
\[
    \widetilde{w}(c - \alpha_1) = \ldots = \widetilde{w}(c - \alpha_n) \quad \text{for any $c \in E$,}
\]
and therefore, $w(f(c)) = n \cdot \widetilde{w}(c - \alpha)$.
Observe that $w(f(E))$ is bounded, as otherwise, for any $c \in E$ with $\widetilde{w}(c - \alpha) = \frac 1 n w(f(c))$ sufficiently large, $\alpha \in E(c) = E$ by Krasner's Lemma~\ref{lem:krasner}. Moreover, we can show that the maximum is attained in $wE \subseteq \IR$.
Let $\gamma \coloneqq \sup w(f(E)) \in \IR$ and let $g(X) \in \Oo_{v'}[X]$ be a lift of the polynomial $f(X)$. Consider the 1-type
\[
    q(x) \coloneqq \left\{r \cdot v\varpi < v(g(x)) \right\}_{r\:\! \in\:\! (-\infty,\gamma)\;\! \cap\;\! \IQ}
\]
Any realisation $x \in K$ satisfies $x \in \Oo_{v'}$ and $w(f(c_0)) = \gamma$ for $c_0 \coloneqq \overline{x} \in E$.
But since $E(\alpha)/E$ is immediate, we can find $a, b \in E$ such that
\[
    \widetilde{w}(c_0 - \alpha) = w(a) \and \widetilde{w}\left(\frac{c_0 - \alpha}{a} - b\right) > 0.
\]
Hence, taking $c_1 \coloneqq c_0 - ab \in E$ yields
\[
    \widetilde{w}(c_1 - \alpha) > \widetilde{w}(c_0 - \alpha),
\]
and therefore,
\[
    w(f(c_1)) > w(f(c_0)) = \gamma,
\]
which contradicts our choice of $\gamma$. Thus, the extension $E(\alpha)/E$ is trivial.
\end{claimproof}

\textit{Claim 3. Any finite extension $F/E$ is algebraically maximal.}

\begin{claimproof}{3}
To this end, let $L$ be an unramified extension of $(K,v')$ with residue field $F$. Repeat the proof of Claim 2 by replacing $K$ and $E$ with $L$ and $F$, noting that $L$ must be saturated as a finite extension of $K$.
\end{claimproof}

\textit{Claim 4. Any finite extension $F/E$ is defectless.}

\begin{claimproof*}{4}
It suffices to consider Galois extensions $F/E$.

If $wE = \IR$, we must have $e(F/E) = 1$. Let $E_1/E$ be the inertia subfield of $F/E$, so that $[E_1 : E] = f(F/E)$ and $F/E_1$ is an immediate extension. By Claim 3, we have $F = E_1$, and the conclusion follows.

If $wE \cong \IZ$, consider instead the tower of field extensions (cf. \cite[Cor. 2.10]{Kuhlmann10})
\[
    E = E_0 \subseteq E_1 \subseteq E_2 \subseteq \ldots \subseteq E_r = F,
\]
where $E_1$ is the ramification subfield of $F/E$ and each $E_{i + 1}/E_i$ is a $C_p$-extension for $i = 1, \ldots, r - 1$, obtained by applying the Galois correspondence to a composition series of the $p$-group $R(F/E) = \Gal(F/E_1)$. By Claim 3, none of the extensions $E_{i + 1}/E_i$ can be immediate, so they are all defectless instead. Finally, this means that $F/E$ can be written as a tower of defectless extensions and hence must be defectless itself.
\end{claimproof*}
\end{proof}

\begin{rem} \label{rem:defectless_generalities}
The standard strategy for proving this fact is the following: it is a formal consequence of $\aleph_1$-saturation that $(E,w)$ is spherically complete, i.e., any decreasing filtration of non-empty open balls has non-empty intersection (cf. \cite[Lem. 2.3.7]{Jahnke-Kartas}). In general, spherical completeness implies defectlessness (this is, in particular, true without the assumption that $E$ is perfect, see \cite[Thm. 31.14./31.21.]{Warner89}). This implication is related to the more widely known fact that complete rank 1 valued fields are defectless (cf. the remark following \cite[Chap. II, (6.8)]{Neukirch99}).
Note that in our proof, which is a simplified implementation of this strategy, we only used finitely many constants for our types (except when arguing that $w$ is henselian). Hence, we can make do with the weaker assumption of $\aleph_0$-saturation to show that $w$ is defectless (at least whenever $E$ is perfect and $w$ is henselian).
\end{rem}

As a consequence of the Saturation-Decomposition Method, we obtain the following application: a sharpened version of the Taming Theorem of Jahnke-Kartas \cite[Thm.\;6.2.3(I)]{Jahnke-Kartas}. In particular, the conclusion about absolute Galois groups can be seen as a non-standard, partial version of the almost purity theorem.

\begin{thm}[Taming Theorem] \label{thm:taming}
Let $(K,v)$ be an $\aleph_1$-saturated valued field of mixed characteristic $(0,p)$. Consider the valuations $v_0$ and $\overline{v_p}$ in the decomposition
\begin{diagram}
    v : K \ar[r,"v_0"] & Kv_0 \ar[r,"\overline{v_p}"] & Kv_p \ar[r,"\dbloverline{v}"] & Kv.
\end{diagram}
Assume $v_0$ is henselian. Then the following statements hold:
\begin{enumerate}[(i)]
    \item $v_0$ is tame.
    \item $v_0$ is absolutely unramified if and only if $vK$ is regular above $vp$.
    \item The following are equivalent:
    \begin{enumerate}[(a)]
        \item $\overline{v_p}$ is tame.
        \item $\overline{v_p}$ is absolutely unramified.
        \item $\Oo_v/p\Oo_v$ is semi-perfect and $vp$ is not minimal positive in $vK$. \label{cond:rdr}
        \item $\overline{v_p}$ has $p$-divisible value group and perfect residue field.
    \end{enumerate}
\end{enumerate}
In particular, if the equivalent conditions for (ii) and (iii) are satisfied, then there exist canonical isomorphisms
\[
    G_K \cong G_{Kv_0} \cong G_{Kv_p}.
\]
\end{thm}
\begin{proof}
By ramification theory (Fact~\ref{fact:ramification group}), $v_0$ is tame outright and $v_0$ is absolutely unramified if and only if $v_0K$ is divisible. By Proposition~\ref{prop:reg_char}, this is precisely the case when $vK$ is regular above $vp$.

(b) $\Rightarrow$ (a). Immediate.

(a) $\Rightarrow$ (d). If $\overline{v_p}$ is tame, then $Kv_p$ is perfect and $\overline{v_p}(Kv_0)$ is $p$-divisible by ramification theory.

(d) $\Rightarrow$ (b). By Fact~\ref{fact:sat->defectless}, we know that $\overline{v_p}$ is henselian defectless and has value group $\IR$. The absolute ramification subgroup of $\overline{v_p}$ is trivial: if not, condition (d) implies that $(Kv_0)^{\sepc}$ is a non-trivial immediate extension of its ramification subfield; one can then find a finite defect extension of $Kv_0$ (cf. Definition~\ref{cor/def:defect}). Hence, $\overline{v_p}$ is tame, and therefore absolutely unramified.

(c) $\Leftrightarrow$ (d). This follows immediately from Theorem~\ref{thm:semi-perf_char}, our criterion for semi-perfectness.

If $v_0$ and $\overline{v_p}$ are indeed absolutely unramified, then their absolute inertia groups are trivial and the absolute Galois groups of $K$, $Kv_0$, and $Kv_p$ are canonically isomorphic by Fact~\ref{fact:inertia group}.
\end{proof}

\begin{rem}
The situation trivialises in the case that $vp$ is minimal positive in $vK$, because then $\dbloverline{v}$ is trivial and $\Oo_v/p\Oo_v \cong Kv$. If $Kv$ is perfect, then $Kv_0 \cong \Frac(W(Kv))$ by the usual Witt structure theorem we already used in the proof of Theorem~\ref{thm:char_Qp}, see also Remark~\ref{rem:Witt}.
Moreover, we note that the above theorem is also true for $\aleph_0$-saturated valued fields $(K,v)$. However, this does not follow directly from our treatment.
\end{rem}

Following Kuhlmann-Rzepka, we give a name to condition~\ref{cond:rdr} in the theorem above:
\begin{defi}
Let $(K,v)$ be a valued field of mixed characteristic $(0,p)$. We call it \emph{roughly deeply ramified} if $\Oo_v/p\Oo_v$ is semi-perfect and $vp$ is not minimal positive in $vK$.
\end{defi}

In \cite{Kuhlmann-Rzepka}, this is the weakest notion of deep ramification. Using the sharp version of the Taming Theorem above, we recover a result of Kuhlmann-Rzepka on deeply ramified fields. The proof below is strikingly short and follows the non-standard strategy of Jahnke-Kartas, who employ it to show that perfectoid fields (a subclass of the class of deeply ramified fields) are closed under finite extensions.

\begin{cor}[Almost Purity, {\cite[Thm. 1.5]{Kuhlmann-Rzepka}}] \label{cor:KR-almost-purity}
Let $(K,v)$ be a roughly deeply ramified valued field of mixed characteristic $(0,p)$. Let $L/K$ be an algebraic extension and $w$ a prolongation of $v$ to $L$. Then $(L,w)$ is roughly deeply ramified.
\end{cor}
\begin{proof}
It suffices to consider the case that $L/K$ is finite, since the valuation ring of $L$ is the union of the valuation rings of all finite subextensions. Consider the ultrapowers $K^*$ and $L^*$ of $K$ and $L$, with respect to the same non-principal ultrafilter over $\IN$. These ultrapowers are automatically $\aleph_1$-saturated \cite[Thm. 6.1.1]{CK90}. Moreover, $L^*/K^*$ is a finite extension of degree equal to $[L : K]$. The Standard Decomposition now reads:
\begin{diagram}
    \llap{$w^* : \;$} L^* \ar[r,"w^*_0"] \ar[d, no head] & L^*w^*_0 \ar[r,"\overline{w^*_p}"] \ar[d, no head] & L^*w^*_p \ar[r,"\dbloverline{w^*}"] \ar[d, no head] & L^*w^* \ar[d, no head] \\
    \llap{$v^* : \;$} K^* \ar[r,"v^*_0"] & K^*v^*_0 \ar[r,"\overline{v^*_p}"] & K^*v^*_p \ar[r,"\dbloverline{v^*}"] & K^*v^*
\end{diagram}
As we have noted before, being roughly deeply ramified is a first-order property, so it holds for $(K^*,v^*)$. By the Taming Theorem~\ref{thm:taming}, the valuation $\overline{v^*_p}$ has $p$-divisible value group and perfect residue field. Consequently, $\overline{w^*_p}$ will have $p$-divisible value group and perfect residue field as well (cf. Fact~\ref{fact:finite_orders}). Therefore, yet again by the Taming Theorem, the valued field $(L^*,w^*)$ is roughly deeply ramified, and so must be $(L,w)$.
\end{proof}
In contrast, the proof in \cite{Kuhlmann-Rzepka} is purely algebraic and proceeds by an analysis of defect.

\subsubsection{Non-standard proof of the Transfer Lemma}

With the Taming Theorem at hand, we can now give a second proof of the Transfer Lemma. This relies only on one additional fact: small Galois groups are encoded in the theory of fields. The Transfer Lemma then follows as an immediate corollary.

\begin{lem}[Klingen, {\cite{Klingen74}}] \label{lem:small}
Let $K$ be a field with small absolute Galois group. Then for any field $F$,
\[
     K \equiv F \Longrightarrow G_K \cong G_F.
\]
\end{lem}
\begin{proof}[Proof (Sketch)]
Let $n \ge 1$ be a natural number. By definition, $K$ admits only finitely many extensions of degree at most $n$. Let $K_n$ denote the compositum of all of these, and set $G_n \coloneqq \Gal(K_n/K)$ (which is finite too). Then $K$ satisfies the first-order sentences $\varphi_n$ expressing ``there exists an isomorphism between $G_n$ and the Galois group of the compositum of all extensions of degree $\le n$, restricting to an isomorphism between $G_{n - 1}$ and the Galois group of the compositum of all extensions of degree $\le n - 1$''. Any field $F \equiv K$ satisfies all $\varphi_n$. From this, an isomorphism of inverse systems can be recovered. Therefore,
\[
    G_K = \varprojlim_n G_n \cong G_F. \qedhere
\]
\end{proof}

\begin{cor} \label{cor:transfer_lemma2}
Let $(K,v)$ be a henselian valued field of mixed characteristic $(0,p)$. Assume further that
\begin{enumerate}[(i)]
    \item $G_K$ is small;
    \item $vK$ is regular above $vp$ and $p$-divisible.
\end{enumerate}
Then, there exists a field $F$ of characteristic $p$ satisfying $G_K \cong G_F$. In particular, the case $vK = \Conv(vp)$ recovers the Transfer Lemma.
\end{cor}
\begin{proof}
By the preceding lemma, we can pass to an $\aleph_1$-saturated elementary extension $(K^*,v^*)$, preserving (ii) and without changing the (small) absolute Galois group.
Recall Lemma~\ref{lem:small_Gal->finite_dim} showed that $(K^*)^{\times}/(K^*)^{\times p}$ is finite, implying $K^*v_p^*$ is perfect by Pop's Lemma~\ref{lem:Pop}(i). Lastly, since $v^*$ has $p$-divisible value group, so does $\overline{v^*_p}$.
Hence, by the Taming Theorem~\ref{thm:taming}, we obtain a chain of canonical isomorphisms
\begin{equation*}
    G_K \cong G_{K^*} \cong G_{K^*v_0^*} \cong G_{K^*v_p^*}.
\end{equation*}
So we may choose $F = K^*v_p^*$ as our characteristic $p$ field.
\end{proof}

\subsection{The work of Jahnke-Kartas}
\label{sec:JK}

We have now sufficiently developed the Saturation-Decomposition Method in the non-discrete case in order to give an exposition of the ideas of Jahnke and Kartas. Their main technical result is:

\begin{thm}[{\cite[Thm. 1.7.3]{Jahnke-Kartas}}] \label{thm:JK_main}
Let $(K,v) \subseteq (K',v')$ be two henselian valued fields with regular dense value groups and residue fields of characteristic $p$. Suppose that one may choose $\varpi \in \Mm_v \setminus\{0\}$ such that
\begin{enumerate}[(i)]
    \item $\Oo_v/\varpi\Oo_v$ and $\Oo_{v'}/\varpi\Oo_{v'}$ are semi-perfect, and
    \item $\Oo_v[\varpi^{-1}]$ and $\Oo_{v'}[\varpi^{-1}]$ are algebraically maximal coarsenings.
\end{enumerate}
Then
\[
    (K,v) \preccurlyeq (K',v') \Longleftrightarrow \Oo_v/\varpi\Oo_v \preccurlyeq \Oo_{v'}/\varpi\Oo_{v'}.
\]
\end{thm}

Let us sketch a proof of Theorem~\ref{thm:JK_main}. Two additional ingredients are needed.

\begin{ingr*}
\begin{enumerate}[(I)]
    \item Add a constant symbol for $\varpi$ to our language. Then, the class of valued field $(K,v)$ described by Theorem~\ref{thm:JK_main} is elementary \cite[Prop. 4.1.4]{Jahnke-Kartas}. On its own, the property that $\Oo_v[\varpi^{-1}]$ is algebraically maximal is not first-order axiomatisable---it is crucial that we additionally know that $\Oo_v/p\Oo_v$ is semi-perfect. \label{ingr:alg_max_axiom}
    \item What enables us to use methods from logic are Kuhlmann's results on the model theory of tame valued fields \cite{Kuhlmann16}. In particular, we will use the following Ax-Kochen/Ershov principle:
    
    \textit{Let $(K,v) \subseteq (K',v')$ be tame valued fields. Then:}
    \[
        (K,v) \preccurlyeq (K',v') \Longleftrightarrow vK \preccurlyeq v'K' \quad \text{\textit{and}} \quad Kv \preccurlyeq K'v'.
    \]
    Moreover, this AKE principle is \emph{resplendent}, that is to say, it still holds if we add extra structure on the residue field sort. For instance, we may add a predicate for a valuation ring to the residue field sort of $(K,v)$.
    \label{ingr:MT_tame}
\end{enumerate}
\end{ingr*}

\begin{proof}[{Proof of Theorem \ref{thm:JK_main} (Sketch).}]
We can interpret $\Oo_v/\varpi\Oo_v$ in $(K,v)$, so the implication from left to right follows immediately. Conversely, assume $\Oo_v/\varpi\Oo_v \preccurlyeq \Oo_{v'}/\varpi\Oo_{v'}$.

The first part is saturating. Let $(K^*,v^*)$ and $(K'^*,v'^*)$ be ultrapowers of $(K,v)$ and $(K',v')$, with respect to the same non-principal ultrafilter on $\IN$. By Łoś's Theorem,
\begin{align*}
    (K,v) \preccurlyeq (K',v') & \Longleftrightarrow (K^*,v^*) \preccurlyeq (K'^*,v'^*) \\
    \Oo_v/\varpi\Oo_v \preccurlyeq \Oo_{v'}/\varpi\Oo_{v'} & \Longleftrightarrow \Oo_{v^*}/\varpi\Oo_{v^*} \preccurlyeq \Oo_{v'^*}/\varpi\Oo_{v'^*}.
\end{align*}
Crucially, Ingredient \ref{ingr:alg_max_axiom} says that the assumptions on $\Oo_v/\varpi\Oo_v$ and $\Oo_v[\varpi^{-1}]$ are preserved in the ultrapower. Since ultrapowers are $\aleph_1$-saturated, we may thus assume, without loss of generality, that we are working with an instance of the theorem where the given valued fields $(K,v)$ and $(K',v')$ are $\aleph_1$-saturated.

The second part of the proof is  decomposition: Theorem~\ref{thm:stand_decomp} yields
\[
    \begin{tikzcd}
        K \ar[r,"v_0"] & Kv_0 \ar[r,"\overline{v_{\varpi}}"] & Kv_{\varpi} \ar[r,"\dbloverline{v}"] & Kv,
    \end{tikzcd}
\]
where $v_{\varpi} = \overline{v_{\varpi}} \circ v_0$ has divisible value group by regularity of $vK$. Thus, $v_{\varpi}$ is tame by the proof of the Taming Theorem~\ref{thm:taming}; the same holds for $v'_{\varpi}$. By saturation,\footnote{There are several ways to make this precise. One is to use the Keisler-Shelah Theorem \cite[6.1.15]{CK90} to choose an ultrafilter at the start of the proof that makes $\Oo_{v^*}/\varpi\Oo_{v^*}$ and $\Oo_{v'^*}/\varpi\Oo_{v'^*}$ (plainly) isomorphic. This simplifies things considerably at the cost of a model-theoretic black box. A different approach is to observe that the nilradical of $\Oo_v/\varpi\Oo_v$ is uniformly $\Ll_{\omega_1,\omega}$-definable. This suffices, because $\Oo_{v^*}/\varpi\Oo_{v^*} \preccurlyeq_{\omega_1,\omega} \Oo_{v'^*}/\varpi\Oo_{v'^*}$ is a formal consequence of the structures being $\aleph_1$-saturated. For more details, see \cite{Gitin26}.}
\[
    \Oo_v/\varpi\Oo_v \preccurlyeq \Oo_{v'}/\varpi\Oo_{v'} \Longrightarrow \Oo_{\dbloverline{v}} = (\Oo_v/\varpi\Oo_v)_{\Red} \preccurlyeq (\Oo_{v'}/\varpi\Oo_{v'})_{\Red} = \Oo_{\dbloverline{v}'},
\]
and consequently, $(Kv_{\varpi}, \Oo_{\dbloverline{v}}) \preccurlyeq (K'v'_{\varpi}, \Oo_{\dbloverline{v}'})$.
Ingredient \ref{ingr:MT_tame} provides the resplendent tame AKE principle. It follows that
\begin{equation} \label{eq:tame_two_places}
    (K, v_{\varpi}, \Oo_{\dbloverline{v}}) \preccurlyeq (K', v'_{\varpi}, \Oo_{\dbloverline{v}'}).
\end{equation}
We can recover our original valuation $v = \dbloverline{v} \circ v_{\varpi}$ definably---as per Section~\ref{sec:comp_places}---via
\[
    x \in \Oo_v \Longleftrightarrow xv_{\varpi} = \res(x) \in \Oo_{\dbloverline{v}}.
\]
Hence, (\ref{eq:tame_two_places}) implies $(K,v) \preccurlyeq (K',v')$ by Fact~\ref{fact:val_lang}.
\end{proof}

We describe two main consequences. The first one is a new class of examples for the AKE philosophy. Prior to Jahnke and Kartas' work, such examples were inaccessible due to our notoriously limited understanding of the first-order theory of positive characteristic fields (see \cite{Anscombe24}) and of the defect.
\begin{cor}
Let $\IF_p(t)^h$ be the henselisation of $\IF_p(t)$ inside $\IF_p\laurent{t}$. Then
\[
    (\IF_p(t)^h)^{1/p^\infty} \preccurlyeq \IF_p\laurent{t}^{1/p^\infty} \preccurlyeq \widehat{\IF_p\laurent{t}^{1/p^\infty}}.
\]
\end{cor}
\begin{proof}
Simply pick $\varpi = t$ and observe that, by a density argument, all three valued fields have the same truncated valuation ring $\Oo/\varpi\Oo = \IF_p[t]^{1/p^\infty}/(t)$ and trivial coarsenings $\Oo[t^{-1}]$. The value groups are isomorphic to $\frac{1}{p^{\infty}}\IZ$ and therefore regular dense. Now apply Theorem~\ref{thm:JK_main}.
\end{proof}

As a second application, one recovers the Generalised Fontaine-Wintenberger Theorem~\ref{thm:FW}.
\begin{thm}[Model-theoretic Fontaine-Wintenberger, {\cite[Thm.\;6.2.3(II)]{Jahnke-Kartas}}]
\

Let $(K,v)$ be a perfectoid field of mixed characteristic $(0,p)$. Let $(K^*,v^*) = (K,v)^{\IN}/\Uu$ be an ultrapower of $(K,v)$ with respect to a non-principal ultrafilter $\Uu$ on $\IN$.
\begin{enumerate}[(i)]
    \item The valuation ring $\Oo_v^\flat$ embeds into $\Oo_{v^*}$ multiplicatively via
    \begin{equation} \label{eq:natural}
        \natural : \Oo_v^\flat \longrightarrow \Oo_{v^*}, \quad x \longmapsto \ulim_{n \in \IN} (x^{\sharp})^{1/p^n}
    \end{equation}
    and induces an elementary embedding $\tau : (K^\flat,v^\flat) \mono (K^*v^*_p,\dbloverline{v^*})$.
    \item The canonical isomorphism of Theorem \ref{thm:taming},
    \[
        G_{K^*} \cong G_{K^*v_p^*} \text{ restricts to }\ G_K \cong G_{K^\flat},
    \]
    the Fontaine-Wintenberger isomorphism.
\end{enumerate}
\end{thm}
Let us briefly explain the main ideas of the beautiful proof.

The musical maps $\sharp$ and $\natural$ defined in (\ref{eq:sharp}) and (\ref{eq:natural}) are both given by limits of a certain kind. Though $\sharp$ is only multiplicative and additive modulo $p$, by passing to an ultrapower, the induced embedding into the canonical residue field $K^*v_p^*$ becomes a genuine valued field embedding, see \cite[Lem. 6.1.2]{Jahnke-Kartas}.

Perhaps an even more subtle observation is elementarity. Let $\varpi^{\flat}$ be such $\varpi = (\varpi^{\flat})^\sharp$ has value $v\varpi = vp$ (Fact~\ref{fact:perfectoid_props}) and write $\pi \coloneqq (\varpi^{\flat})^{\natural} \in \Oo_{v^*}$. Then the composition
\begin{equation*}
    \begin{tikzcd}[row sep=0em]
        \llap{$\overline{\natural} :\;$} \Oo_{v^\flat}/(\varpi^{\flat}) \ar[r,"\cong","\overline{\sharp}"'] & \Oo_v/(\varpi) \ar[r,"\iota"',"\preccurlyeq"] & \Oo_{v^*}/(\varpi) \ar[r,"\cong","{\;\Phi^*}"'] & \Oo_{v^*}/(\pi) \\
        x/(\varpi^{\flat}) \ar[r, maps to] & x^{\sharp}/(\varpi) \ar[r, maps to] & \displaystyle\ulim_{n \in \IN} x^{\sharp} \ultra (\varpi) \ar[r, maps to] & \displaystyle\ulim_{n \in \IN} (x^{\sharp})^{1/p^n} \ultra (\pi)
    \end{tikzcd}
\end{equation*}
is equal to the map $\overline{\natural}$ induced by $\natural$ on $\Oo_{v^\flat}/(\varpi^{\flat})$. Here, $\iota$ is induced by the diagonal embedding into the ultrapower (elementary by Łoś) and $\Phi^*$ is the inverse of a non-standard Frobenius isomorphism. Loosely speaking, $\overline{\natural}$ is an elementary embedding because it is a diagonal embedding, twisted twice by ``cancelling'' Frobenius isomorphisms \cite[Thm. 6.1.3]{Jahnke-Kartas}.
One then shows that the embedding $\tau$ is itself elementary by carefully checking that the assumptions of Theorem \ref{thm:JK_main} are satisfied. The main difficulty in this step is in showing that $\Oo_{\dbloverline{v^*}}[\dbloverline{\pi}^{-1}]$ is algebraically maximal, cf. \cite[Lem. 4.2.5]{Jahnke-Kartas}.

The final conclusion in (ii) uses the special property of perfectoid fields that any finite extension generated by a polynomial of discriminant of value $<vp$ can be generated by polynomials of arbitrarily small discriminant (essentially, by taking $p$-th roots of the coefficients iteratively). As a consequence, $\natural$ interpolates between extensions of $K^{\flat}$, generated by polynomials of (potentially) arbitrarily small discriminant, and finite extensions of $K^*$ coming from $K$, generated by polynomials of (actually) infinitesimal discriminant, i.e., of value $\ll vp$, making the latter unramified with respect to $v^*_p$, cf. \cite[Cor. 6.2.2]{Jahnke-Kartas}.

We have described a proof of the Fontaine-Wintenberger Theorem via the Saturation-Decom\-po\-sition Method. At the same time, recall that we used this theorem to give our first proof of the Transfer Lemma. Methodically, we have thus come full circle in our diagram on p.\;\pageref{diagram:inf_ram}.

\begin{rem}[Limitations of the method]
If we follow the strategy of the proof of Theorem~\ref{thm:JK_main}, the results will be limited by the strength of the model-theoretic input that enters as Ingredient~\ref{ingr:MT_tame}. Indeed, Jahnke-Kartas show that Theorem~\ref{thm:JK_main} still holds if we replace the regularity assumption on value groups by $\Gamma_v \preccurlyeq \Gamma_{v'}$, see \cite[Thm. 1.7.3]{Jahnke-Kartas}.
This is done by adapting Kuhlmann's model theory of tame valued fields to \emph{roughly tame} valued fields (those $(K,v)$ for which the core field $(Kv_0,\overline{v})$ is tame). It turns out that this weakening of assumptions is harmless, essentially because the (0,\:\!0)-place in the Standard Decomposition is always tame by the Lemma of Ostrowski \ref{cor/def:defect} (cf. \cite[3.3.7]{Jahnke-Kartas}). In \cite{JvdS26}, Jahnke and van der Schaaf generalise Theorem~\ref{thm:JK_main} even further using AKE principles for separably tame
valued fields due to Kuhlmann-Pal \cite{Kuhlmann-Pal} and Anscombe \cite{Anscombe25}.
\end{rem}

%% file: part5_main.tex
\section{Proof of Main Theorem}
\label{chap:main}

We are finally ready to prove Theorem~\ref{thm:main theorem}. One direction is straightforward:

\begin{prop}
    Let $K \equiv \IQ_p$ in the language of rings. Then $G_{K} \cong G_{\IQ_p}$.
\end{prop}
\begin{proof}
Recall that $\IQ_p$ has small absolute Galois group (Corollary~\ref{cor:finitelyManyExtensionsQ_p}). Small absolute Galois groups are encoded in the theory of a field, so the conclusion follows (see Lemma~\ref{lem:small}).
\end{proof}

For the other direction, we now assume $G_K \cong G_{\IQ_p}$. Recall that by Theorem~\ref{thm:char_Qp}, to prove $K \equiv \IQ_p$, it suffices to find a valuation $v$ on $K$ satisfying the following properties:
\begin{enumerate}[(i)]
    \item $(K,v)$ is henselian of mixed characteristic $(0,p)$;
    \item $vK \equiv \IZ$;
    \item $v(p)$ is minimal positive;
    \item $Kv=\IF_p$.
\end{enumerate}

Additionally, note that (ii) admits the following well-known algebraic description:
\begin{fact}[{\cite[Chap. 3]{Marker02}}] \label{fact:Z-groups}
    For an ordered abelian group $\Gamma$, the following are equivalent:
    \begin{enumerate}[(i)]
        \item $\Gamma \equiv \IZ$;
        \item $\Gamma$ has a convex subgroup $\Delta \cong \IZ$ such that $\Gamma/\Delta$ is divisible;
        \item $\Gamma$ has a minimal positive element and $[\Gamma : n\Gamma] = n$ for all $n \in \IN$.
    \end{enumerate}
\end{fact}

We will show that the canonical henselian valuation $v_K$, introduced in Section~\ref{sec:can_hens_val}, satisfies all the required properties.
The strategy is organised into three main steps:
\begin{itemize}[wide, labelwidth=!, labelindent=0pt]
    \item[\ref{main_step:1}.] Show that the canonical henselian valuation $v_K$ has mixed characteristic $(0,p)$.
    \begin{itemize}[labelwidth=!, labelindent=5em, leftmargin=*]
        \item[\ref{main_step:1.1}.] Find henselian valuations on $K$ using the Galois characterisation of henselianity (Theorem~\ref{thm:Galois code henselianity}). Infer that the canonical henselian valuation $v_K$ is non-trivial and satisfies $Kv_K \ne (Kv_K)^{\sepc}$ and $(v_K K : q v_K K) = q$ for all $q \ne p$.
        \item[\ref{main_step:1.2}.] Deduce that $\charK K = 0$ using Idea~\ref{idea:char} (Lemma~\ref{lem:C_pInC_p^2Charp} and Lemma~\ref{lem:C_pNotInC_p^2}).
        \item[\ref{main_step:1.3}.] Prove that the prime-to-$p$ roots of unity of $K$ and $\IQ_p$ coincide using Lemma~\ref{lem:prime-to-p roots}.
        \item[\ref{main_step:1.4}.] Show that $\charK Kv_K \ne q$ for all primes $q \ne p$ via cyclotomic extensions.
        \item[\ref{main_step:1.5}.] Deduce that $\charK Kv_K = p$ using the previous steps and ramification-theoretic considerations (Section~\ref{sec:ram_theory}).
    \end{itemize}
    
    \item[\ref{main_step:2}.] Apply the Standard Decomposition (Section~\ref{sec:stand_decomp})
    \begin{diagram}
        v_K : K \ar[r,"v_0","{(0,\:\!0)}"'] & Kv_0 \ar[r,"\overline{v_p}","{(0,\,p)}"'] & Kv_p \ar[r,"\dbloverline{v}","{(p,\,p)}"'] & Kv_K
    \end{diagram}
    to decompose $v_K$ into three places that can be treated separately. Show that $(Kv_0, \overline{v_p})$ has value group $\IZ$ and finite residue field.
    \begin{itemize}[labelwidth=!, labelindent=5em, leftmargin=*]
        \item[\ref{main_step:2.1}.] Show $G_{Kv_0}$ cannot be realised as an absolute Galois group in characteristic $p$. Use the Transfer Lemma~\ref{cor:transfer_lemma1} to conclude that $\overline{v_p}(Kv_0)$ is not $p$-divisible.
        \item[\ref{main_step:2.2}.]
        Deduce $\overline{v_p}(Kv_0) \cong \IZ$ and $Kv_p$ is finite using Pop's Lemma~\ref{lem:Pop}.
    \end{itemize}
    \item[\ref{main_step:3}.] Conclude that $(K,v_K)$ is $p$-adically closed.
    \begin{itemize}[labelwidth=!, labelindent=5em, leftmargin=*]
        \item[\ref{main_step:3.1}.] Observe that $Kv_K = \IF_p$.
        \item[\ref{main_step:3.2}.] Prove that $v_KK \equiv \IZ$ using the elementary characterisation of $\IZ$ (Fact~\ref{fact:Z-groups}).
        \item[\ref{main_step:3.3}.] Prove that $v_K(p)$ is minimal positive using a counting argument (Lemma~\ref{lem:dimOfPowerQuotient} and Remark~\ref{rem:general_dim_formula}).
    \end{itemize}
\end{itemize}

\begin{itemize}[wide, labelwidth=!, labelindent=0pt]
\stepitem{1}{main_step:1}
To say anything about the canonical henselian valuation $v_K$, we first need to find at least one sufficiently non-trivial henselian valuation on $K$.

\stepitem{1.1}{main_step:1.1}
For primes $q \ne p$, any Sylow $q$-subgroups of $G_K \cong G_{\IQ_p}$ is isomorphic to a semi-direct product $\IZ_q \rtimes \IZ_q$ (Proposition~\ref{prop:Syl G_Q_p}). The Galois characterisation of henselianity (Theorem~\ref{thm:Galois code henselianity}) gives rise to henselian valuations $v_q$ on $K$, tamely branching at $q$. In particular, $\charK K v_q \ne q$ and $v_q K \ne qv_q K$, i.e.,
\[
    r_q \coloneqq \dim_{\IF_q}(v_q K/q v_q K) \ge 1.
\]
For any $q \ne p$, let $K_q$ be the fixed field of some Sylow $q$-subgroup of $G_K$. By henselianity, every $v_q$ has a unique prolongation $w_q$ to $K_q$. The $q$-rank $r_q$ does not change when we pass from $K$ to $K_q$. By Fact~\ref{fact:ramification group}, the ramification subgroup $R_{K_q}$ is trivial (since $\charK K_q w_q \ne q$), and
\[
    \IZ_q^{r_q} \cong I_{K_q} \vartriangleleft G_{K_q} \cong \IZ_q \rtimes \IZ_q
\]
by Theorem~\ref{thm:Splitting of D/R}. Using the explicit description in Remark~\ref{rem:Z_q-action}, one can check that the only non-trivial closed normal subgroups of $\IZ_q \rtimes \IZ_q$ are of the shape $q^s\IZ_q \times \{1\}$ and $q^{s_1}\IZ_q \rtimes q^{s_2}\IZ_q$. Since the latter are non-abelian, we conclude $r_q = 1$, $I_{K_q} \ne G_{K_q}$, and $K_q w_q$ is not separably closed.

We can further see that $G_{K_qw_q} \cong \IZ_q$ as follows. If
$I_{K_q} \cong q^s\IZ_q \times \{1\}$, then
\[
    G_{K_qw_q} \cong G_{K_q}/I_{K_q} \cong C_{q^s} \rtimes \IZ_q.
\]
We are done if $s = 0$, so assume for a contradiction that $s \ge 1$. By Artin-Schreier, $C_{q^s} \rtimes \IZ_q$ can only occur as an absolute Galois group of a field if $q = 2$ and $s = 1$. Moreover, as $\Aut(C_2) =1$, $C_2 \rtimes \IZ_2$ is a direct product in this case, contradicting Lemma~\ref{lem:no_normal_C2}. We conclude $I_{K_q} \cong \IZ_q \cong G_{K_qw_q}$ in all cases.

While we will continue working with $(K,v_q)$ and $(K_q,w_q)$ until \ref{main_step:1.4}, let us already make some observations about the canonical henselian valuation $v_K$. As the fields $K_q w_q$ are not separably closed, we have $\Oo_{v_K} \subseteq \Oo_{v_q}$ for all $q \ne p$ by Lemma~\ref{lem:tree}. In particular, $v_q K$ is a quotient of $v_K K$ by a convex subgroup, so $v_K K$ cannot be $q$-divisible. By repeating the above reasoning for $v_K$, we deduce that
\[
    r_q = \dim_{\IF_q}(v_K K/q v_K K) = 1,
\]
so $(v_K K: q v_K K) = q$ and $Kv_K$ is not separably closed.

\stepitem{1.2}{main_step:1.2}
From $\charK K v_q \ne q$, it follows that $\charK K \ne q$ for all primes $q \ne p$. Therefore, to show $\charK K = 0$, it only remains to exclude the case $\charK K = p$.

By Lemma~\ref{lem:C_pNotInC_p^2}, the field $\IQ_p$ satisfies the following property: it has a $C_{p - 1}$-extension (namely, $\IQ_p(\zeta_p)/\IQ_p$), which itself admits a $C_p$-extension that is not embeddable into a $C_{p^2}$-extension. As $G_K \cong G_{\IQ_p}$, the same property holds for $K$. However, this cannot happen in characteristic~$p$ by Lemma~\ref{lem:C_pInC_p^2Charp}.

\stepitem{1.3}{main_step:1.3}
Determining the prime-to-$p$ roots of unity of $K$ (from Galois theory alone) is a somewhat subtle affair. We would like to use Lemma~\ref{lem:prime-to-p roots}, which gives a Galois characterisation of the roots of unity of certain valued fields $(K,v)$. Condition \ref{item:unitroots_assum3} in this lemma requires knowledge of the Galois group of the maximal pro-$q$ Galois extension of $Kv$, which we would typically construct (say, over $\IF_p$) by looking at the appropriate cyclotomic extension. Such a direct argument is not available to us due to circularity (we still do not know the prime-to-$p$ roots of unity of $K$ yet). So instead, we will use an auxiliary construction.

Consider the field
\[
    \IQ_{p,q} \coloneqq \IQ_p\bigl(\zeta_{q^{\infty}},\!\!\sqrt[q^{\infty}]{p}\bigr) = \bigcup_{n \ge 1} \IQ_p\bigl(\zeta_{q^n},\!\!\sqrt[q^n]{p}\bigr).
\]
We compute its Galois group as
\[
   \Gal(\IQ_{p,q}/\IQ_p) = \Gal(\IQ_{p,q}/\IQ_p(\zeta_{q^{\infty}})) \rtimes \Gal(\IQ_p(\zeta_{q^{\infty}})/\IQ_p) \cong \IZ_q \rtimes (\IZ_q \times C),
\]
where $C = \Gal(\IQ_p(\zeta_q)/\IQ_p)$ is a cyclic group of order dividing $q - 1$ (the precise order depends on $p$ and $q$) and the semi-direct product is again given by the Frobenius action as in Remark~\ref{rem:Z_q-action}. Hence, there exists a Galois extension $L/K$ with
\[
    \Gal(L/K) \cong \IZ_q \rtimes (\IZ_q \times C).
\]
Let $u_q$ be the unique prolongation of $v_q$ to $L$. By considering possible closed normal abelian subgroups isomorphic to $I_{u_q}(L/K)$ as in
\ref{main_step:1.1}, we obtain
\[
    \Gal(Lu_q/Kv_q) \cong C_{q^s} \rtimes (\IZ_q \times C).
\]
Therefore, $Kv_q$ admits a Galois extension with Galois group $\IZ_q$, which must be maximal pro-$q$ since $G_{K_qw_q} \cong \IZ_q$ (this also implies $s = 0$). Similarly, one shows that the same holds true for the residue field $Kv_K$.

These considerations show that $(K,v_q)$ satisfies the conditions of Lemma~\ref{lem:prime-to-p roots}, and hence,
\begin{align*}
    \zeta_{q^r} \in \IQ_p & \iff \text{$\IQ_p$ admits a $C_{q^r} \times C_{q^r}$-extension} \\
    & \iff \text{$K$ admits a $C_{q^r} \times C_{q^r}$-extension} \iff \zeta_{q^r} \in K,
\end{align*}
for any prime $q \ne p$ and $r \ge 1$. In particular, the set of prime-to-$p$ roots of unity in $K$ is $\mu_{p - 1}$ for $p \ne 2$ and $\mu_2$ for $p = 2$ by Corollary~\ref{cor:rootsOfUnity}.

\stepitem{1.4}{main_step:1.4}
Assume that $\charK K v_K = q$ for some prime $q \ne p$. As $v_K$ is a refinement of $v_q$, there exists an induced valuation $\overline{v_K}$ on $Kv_q$ with residue field $ K v_K$ and of mixed characteristic $(\charK Kv_q, q)$.
\[
    \begin{tikzcd}
        K \ar[rr,bend left=30,"v_K"] \ar[r, "v_q"]& Kv_q \ar[r, "\overline{v_K}"] & Kv_K
    \end{tikzcd}
\]
Clearly, $\charK Kv_q \ne q$ implies $\charK K v_q = 0$. We now derive a contradiction by considering the cyclotomic extension $K(\zeta_{q^\infty})/K$. By Lemma~\ref{lem:inf_cyclotom}, such a cyclotomic extension can be finite only in one of two exceptional cases, namely $K(\zeta_{q^{\infty}}) = K(\zeta_q)$ or $K(\zeta_{2^{\infty}}) = K(\zeta_4)$. Neither of these occur: indeed, this follows by applying Lemma~\ref{lem:prime-to-p roots} to $K(\zeta_q)$ and $K(\zeta_4)$; see also Remark~\ref{rem:prime-to-p roots finite extensions}. Hence $K(\zeta_{q^{\infty}})/K$ is an infinite cyclotomic extension. By henselianity, this extension is unramified with respect to $v_q$. Therefore, by Lemma~\ref{lem:inf_cyclotom},
\[
    \Gal(K(\zeta_{q^{\infty}})/K) \cong \Gal(Kv_q(\zeta_{q^{\infty}})/Kv_q) \cong \IZ_q \times C.
\]
Thus, $Kv_q(\zeta_{q^{\infty}})$ must be $q$-closed (recall that the maximal pro-$q$ Galois extension of $Kv_q$ has Galois group $\IZ_q$).

From $Kv_q(\zeta_{q^{\infty}})$ being $q$-closed and $\charK Kv_K = q$, it follows that the residue field of $Kv_q(\zeta_{q^{\infty}})$ with respect to the unique prolongation of $\overline{v_K}$ is equal to $Kv_K$ and must be $q$-closed as well. In positive characteristic $q$, cyclotomic extensions by $\zeta_{q^{q^n} - 1}$ have degree dividing $q^n$. It follows that $\zeta_{q^{q^n} - 1} \in Kv_K$ for all $n \in \IN$.
However, note that we have
\[
    \gcd\biggl(\frac{q^{q^{n + 1}} - 1}{q^{q^n} - 1},\, q^{q^n} - 1\biggr) = 1,
\]
which is an instance of the well-known identity $\gcd((a^m - 1)/(a - 1), a - 1) = \gcd(m, a - 1)$. In conjunction with the above, this implies that there are infinitely many primes $q'$ with $\zeta_{q'} \in Kv_K$.
This contradicts \ref{main_step:1.3}, where we showed that $Kv_K$ contains a $q'$-th primitive root of unity for primes $q' \ne p, q$ if and only if $q' \mid (p - 1)$. We conclude $\charK Kv_K \ne q$ for all primes $q \ne p$.

\stepitem{1.5}{main_step:1.5}
To exclude the remaining case, assume that $\charK Kv_K = 0$. Here, we will eventually arrive at a contradiction by considering the ramification data for explicit extensions of $(K,v_K)$ implied by this assumption.\footnote{For this purpose, Pop uses a different type of argument. In the proof of \cite[(1.4)]{Pop88}, he shows that the Sylow $p$-subgroups of any non-trivial normal subgroup of $G_{\IQ_p}$ are not pro-cyclic. However, if $\charK Kv_K = 0$, then considering the action in the semi-direct product of Theorem~\ref{thm:Splitting of D/R} shows that $G_K$ has a normal pro-cyclic subgroup of the form $\prod_{r_q \ne 1} \IZ_q$, which has a pro-cyclic Sylow $p$-subgroup.}

(1) We first show that $v_KK$ is $p$-divisible. Assume towards a contradiction that
\[
    r_p \coloneqq \dim_{\IF_p}(v_K K/p v_K K) \ge 1,
\]
and let $K_p$ be the fixed field of some Sylow $p$-subgroup of $G_K$.
By ramification theory (Facts~\ref{fact:inertia group}, \ref{fact:ramification group}, \ref{thm:Splitting of D/R}), the absolute ramification subgroup of $K_p$ with respect to the unique prolongation of $v_K$ to $K^{\sepc}$ is trivial, so $I_{K_p} \cong \IZ_p^{r_p}$ is a non-trivial closed normal abelian subgroup of $G_{K_p}$. Moreover, $G_{K_p} \not\cong \IZ_p, \IZ_2 \rtimes C_2$; for instance, $K_p$ admits $C_p^3$-extensions by the Degree Lemma~\ref{lem:dimOfPowerQuotient}.
Hence, the absolute Galois groups $G_K \cong G_{\IQ_p}$ satisfy the hypotheses of the Galois characterisation of henselianity (Theorem~\ref{thm:Galois code henselianity}).
In particular, both $K$ and $\IQ_p$ must admit a henselian valuation tamely branching at $p$.
However, as we have seen in Example~\ref{ex:valuations on Qp}, the only henselian valuations on $\IQ_p$ are the trivial valuation and the standard $p$-adic valuation (which has residue characteristic $p$), neither of which is tamely branching at $p$. Therefore, $v_K K$ is $p$-divisible.

(2) Next, we claim that $K$ admits a unique $C_{p - 1} \times C_{p - 1}$-extension $L$. The existence follows from the fact that $\IQ_p$ admits a $C_{p - 1} \times C_{p - 1}$-extension, namely $\IQ_p(\zeta_{(p-1)^2}, \zeta_p)/\IQ_p$. To show uniqueness, observe that any such extension is the compositum of its $C_{q^r} \times C_{q^r}$-subextensions, where $q^r$ ranges over the prime powers in the factorisation of $p - 1$. At the same time, each $C_{q^r} \times C_{q^r}$-extension of $K$ is unique, for otherwise, the Galois group of the compositum of two distinct $C_{q^r} \times C_{q^r}$-extensions would have $q$-rank at least 3. For primes $q \mid (p - 1)$ this is incompatible with $G_{K_q} \cong \IZ_q \rtimes \IZ_q$ having $q$-rank 2.

(3) We claim that $K$ admits a unique $C_{p - 1}$-extension unramified at $v_K$, which can moreover be characterised uniquely in Galois-theoretic terms.
From \ref{main_step:1.3}, we know that
\[
    \Gal(Kv_K(q)/Kv_K) \cong \IZ_q.
\]
In particular, $Kv_K$ admits a unique $C_{q^r}$-extension for each prime power $q^r | (p - 1)$. Arguing again via composita, we infer that $Kv_K$ admits a unique $C_{p - 1}$-extension that induces a unique unramified $C_{p - 1}$-extension $L_1/K$. Using our description of the prime-to-$p$ roots in \ref{main_step:1.3} and the proof of Lemma~\ref{lem:inf_cyclotom}, one can explicitly write $L_1 = K(\zeta_{(p - 1)^2})$. By Lemma~\ref{lem:prime-to-p roots} and Remark~\ref{rem:prime-to-p roots finite extensions}, this gives rise to a Galois-theoretic description of $L_1$: it is the unique $C_{p - 1}$-extension that itself admits a $C_{q^{2r}} \times C_{q^{2r}}$-extension for every prime power $q^r \mid (p - 1)$.
Moreover, we can identify totally ramified extensions: any $C_{p - 1}$-extension of $K$ that is linearly disjoint from $L_1$ over $K$ must be totally ramified by the uniqueness of $L_1$.

(4.1) We now derive the contradiction for $p \ne 2$. Consider the $C_p$-extension of $\IQ_p(\zeta_p)$ obtained by adjoining $\sqrt[p]{p}$. Over $\IQ_p$, it has a non-abelian Galois group
\[
    \Gal(\IQ_p(\zeta_p,\!\sqrt[p]{p})/\IQ_p) \cong C_p \rtimes C_{p - 1}.
\]
Being totally ramified, it is linearly disjoint from $\IQ_p(\zeta_{(p - 1)^2})$ over $\IQ_p$. Translating over to $K$, we find a Galois extension $L_3/K$ with Galois group $C_p \rtimes C_{p - 1}$ and $C_{p - 1}$-subextension $L_2/K$, such that $L_1$ and $L_3$ are linearly disjoint over $K$ (with $L_1$ uniquely determined by the above Galois property).
\begin{diagram}[cramped, column sep=1.75em]
    & L \ar[dl, no head] \ar[dr, no head] & L_3 \ar[d, "C_p", no head, pos=0.55] \\
    L_1 \ar[dr, "C_{p - 1}"', no head] & & L_2 \ar[dl, "C_{p - 1}", no head] \\
    & K &
\end{diagram}
As noted before, $L_2/K$ is totally ramified. Since $v_KK$ is $p$-divisible, the $C_p$-extension $L_3/L_2$ must be unramified. But then, $L_3$ is the compositum of $L_2$ and the inertia subfield of $L_3/K$. Therefore, $\Gal(L_3/K) \cong C_{p - 1} \times C_p$ is abelian, contrary to our assumption.

(4.2) Finally, consider the case $p = 2$. We need to modify the above constructions slightly to work around the triviality of the $C_{p - 1} \times C_{p - 1}$-extension. Let us instead start with the $C_2$-extension $\IQ_2(\zeta_3)/\IQ_2$. It admits a unique $C_3 \times C_3$-extension given by $\IQ_2(\zeta_9, \!\sqrt[3]{2})/\IQ_2(\zeta_3)$, where $\IQ_2(\zeta_9)/\IQ_2(\zeta_3)$ is the unique unramified $C_3$-extension and $\IQ_2(\zeta_3, \!\sqrt[3]{2})/\IQ_2(\zeta_3)$ is a totally ramified $C_3$-extension. By Kummer Theory and the Degree Lemma~\ref{lem:dimOfPowerQuotient}, the field $\IQ_2(\zeta_3, \!\sqrt[3]{2})$ admits a $C_2^8$-extension.
\begin{diagram}[cramped, column sep=small]
    & & L_3 \ar[d, "C_2^8", no head] \\
    L_1 \ar[dr, "C_3"', no head] & & L_2 \ar[dl, "C_3", no head] \\
    & K(\zeta_3) \ar[d, "C_2", no head] & \\
    & K &
\end{diagram}

Let $L_1 \coloneqq K(\zeta_9)$ be the unramified $C_3$-extension of $K(\zeta_3)$, uniquely identified as the only one that admits a $C_9 \times C_9$-extension. Furthermore, let $L_2/K(\zeta_3)$ be a $C_3$-extension such that $L_1$ and $L_2$ are linearly disjoint over $K(\zeta_3)$, $L_3/K(\zeta_3)$ is Galois, and $L_3/L_2$ has Galois group $C_2^8$. As before, $L_2/K(\zeta_3)$ is totally ramified and $L_3/L_2$ is unramified by the $2$-divisibility of $v_KK$. Therefore, the inertia subfield of $L_3/K(\zeta_3)$ is a $C_2^8$-extension of $K(\zeta_3)$. This is, however, impossible: the degree of the maximal abelian Galois extension of exponent 2 of $K(\zeta_3)$ and $\IQ_2(\zeta_3)$ must coincide. This degree is equal to $16$, as can be seen by Kummer theory and the Degree Lemma~\ref{lem:dimOfPowerQuotient},
\[
    \dim_{\IF_2} \left(K(\zeta_3)^{\times}/K(\zeta_3)^{\times 2}\right) = \dim_{\IF_2}\left(\IQ_2(\zeta_3)^{\times}/\IQ_2(\zeta_3)^{\times 2}\right)= 4.
\]
This concludes our proof that $\charK Kv_K = p$.

\stepitem{2}{main_step:2}
As now $(\charK K, \charK Kv_K) = (0,p)$, we may consider the Standard Decomposition
\begin{diagram}
    v_K : K \ar[r,"v_0"] & Kv_0 \ar[r,"\overline{v_p}"] & Kv_p \ar[r,"\dbloverline{v}"] & Kv_K.
\end{diagram}
Each place in the decomposition is henselian (Fact~\ref{fact:hensel_preserve}); by Remark~\ref{rem:stand_decomp_mixed}, the valuation $\overline{v_p}$ is of mixed characteristic $(0,p)$ and rank 1.

\stepitem{2.1}{main_step:2.1}
We first show that $G_{Kv_0}$ cannot be realised as the absolute Galois group of a field of characteristic $p$. To rule this out, we argue along the same lines as in \ref{main_step:1.2}.
By Lemma~\ref{lem:C_pNotInC_p^2}, there exists a $C_{p - 1}$-extension $M_1$ of $K$ with the property that it admits a $C_p$-extension $M_2$ that does not embed into any $C_{p^2}$-extension of $M_1$. Let $v'_0$ and $v''_0$ denote the unique prolongations of $v_0$ to $M_1$ and $M_2$, respectively. Since $\charK Kv_0 = 0$, the value group $v_0K$ is $p$-divisible by the same argument as in the first part of \ref{main_step:1.5}. In particular, $v'_0M_1$ is $p$-divisible, so $M_2/M_1$ must be unramified and $M_2v''_0$ is a $C_p$-extension of $M_1v'_0$ that does not embed into any $C_{p^2}$-extension, as we could lift any such embedding to $M_1$ by henselianity.
This is a property of the absolute Galois group $G_{M_1v'_0}$; by Lemma~\ref{lem:C_pInC_p^2Charp}, this profinite group cannot be realised in characteristic $p$.

\begin{diagram}[cramped]
    M_2 \ar[d, no head, "C_p"', pos=0.53] \ar[r, "v''_0"] & M_2v''_0 \ar[d, no head, "C_p", pos=0.53] \\
    M_1 \ar[d, no head, "C_{p - 1}"', pos=0.53] \ar[r, "v'_0"] & M_1v'_0 \ar[d, no head] \\
    K \ar[r, "v_0"] & Kv_0
\end{diagram}

At the same time, note that the absolute Galois group $G_K = D_K$ of $K$ projects onto the absolute Galois group $G_{Kv_0}$ of the residue field; in particular, $G_K$ is small implies that $G_{Kv_0}$ is small as well.
Assume that $\overline{v_p}(Kv_0)$ is $p$-divisible. By the Transfer Lemma~\ref{cor:transfer_lemma1}, the profinite group $G_{Kv_0}$ can be realised as the absolute Galois group of a field of positive characteristic $p$. The same is true for the finite index open subgroup $G_{M_1v'_0}$. We arrive at a contradiction, proving that $\overline{v_p}(Kv_0)$ is not $p$-divisible.

\stepitem{2.2}{main_step:2.2}
The smallness of $G_{Kv_0}$ implies that $(Kv_0)^\times/(Kv_0)^{\times p}$ is finite (Lemma~\ref{lem:small_Gal->finite_dim}). By Pop's Lemma~\ref{lem:Pop}, $\overline{v_p}(Kv_0)$ is either discrete or $p$-divisible. We have shown the latter to be impossible. By the last part of Pop's Lemma, $Kv_p$ is finite.

\stepitem{3}{main_step:3}
As finite fields only admit trivial valuations, it follows that $\dbloverline{v}$ on $Kv_p$ is trivial and $Kv_p = Kv_K$. In particular, the rank 1 valuation $\overline{v_p}$ is the same as the valuation $\overline{v_K}$ on $Kv_0$ induced by $v_K$.

\stepitem{3.1}{main_step:3.1}
Using the finiteness of $Kv_K$, we can now identify the residue field $Kv_K$ as $\IF_p$. From~\ref{main_step:1.3}, we know that the prime-to-$p$ roots of unity of $K$ are $\mu_{p - 1}$. This implies $Kv_K = \IF_p$ by Hensel's Lemma.

\stepitem{3.2}{main_step:3.2}
To prove $v_KK \equiv \IZ$, it suffices to show that $v_KK$ has a convex subgroup $\Delta \cong \IZ$ such that $v_KK/\Delta$ is divisible by Fact~\ref{fact:Z-groups}. By~\ref{main_step:2.2}, the convex subgroup $\Delta \coloneqq \overline{v_K}(Kv_0)$ is discrete. It remains to verify that $v_KK/\Delta \cong v_0K$ is divisible.

We have already shown that $v_0K$ is $p$-divisible in~\ref{main_step:2.1}. We established $(v_KK : qv_KK) = q$ in~\ref{main_step:1.1} for all primes $q \ne p$, so the minimal positive element of $\Delta$ already generates $v_KK/qv_KK$. Therefore, $v_0K$ is divisible.

\stepitem{3.3}{main_step:3.3}
Let $L$ again denote the unique $C_{p - 1} \times C_{p - 1}$-extension of $K$. The extension $K(\zeta_p)/K$ is abelian of degree dividing $p - 1$, so it is contained in $L$, i.e., $\zeta_p \in L$.
Write $v_L$ for the unique extension of $v_K$ to $L$, and assume $v_K$ and $v_L$ are normalised such that the minimal positive elements are both $1 \in \IZ$. It follows from our analysis in \ref{main_step:1.5} that $L/K$ has ramification index $e(L/K) = p - 1$ and degree of inertia $f \coloneqq f(L/K) = p - 1$. Therefore, the ``absolute ramification index'' of $L$ is given by $v_L(p) = (p - 1)v_K(p)$.

Additionally, write $F \coloneqq \IQ_p(\zeta_{(p - 1)^2}, \zeta_p)$ for the unique $C_{p - 1} \times C_{p - 1}$-extensions of $\IQ_p$. By the Degree Lemma~\ref{lem:dimOfPowerQuotient} and Remark~\ref{rem:general_dim_formula}, we obtain
\begin{align*}
    \dim_{\IF_p} \left(F^\times/F^{\times p}\right) & = (p - 1)^2 + 2 \\
    \dim_{\IF_p} \left(L^\times/L^{\times p}\right) & = (p - 1)^2 v_K(p) + 2. 
\end{align*}
Kummer theory says that these two quantities express the dimension of the maximal abelian Galois extension of exponent $p$ of the unique $C_{p - 1} \times C_{p - 1}$-extensions of $F$ and $L$.
These dimensions are determined by the isomorphic absolute Galois groups $G_K$ and $G_{\IQ_p}$. Since they must therefore coincide, it follows that $v_K(p) = 1$ is minimal positive, as desired.
\end{itemize}